\documentclass[11pt]{amsart} 

\title[Existence of Typical Scales]
 {Existence of Typical Scales for Manifolds with Lower Ricci Curvature Bound}

\author[D.~Jansen]{Dorothea Jansen}

\email{d.jansen@uni-muenster.de}

\date{\today}



\thanks{This work was supported by
  the Gottfried Wilhelm Leibniz-Preis of Prof.~Dr.~Burkhard Wilking 
  and the SFB 878: Groups, Geometry \& Actions, at the University of M\"unster.}

\usepackage[english]{babel}

\usepackage[utf8]{inputenc}

\usepackage[T1]{fontenc}

\usepackage{xcolor}
\definecolor{pantone3282}{RGB}{0,142,150} 
\definecolor{pantone369}{RGB}{122,181,22} 

\usepackage{nameref}
\usepackage{aliascnt}

\usepackage{xspace}

\usepackage{comment}

\usepackage[hypertexnames=true]{hyperref}
\hypersetup{
 citebordercolor = pantone369!70,
 linkbordercolor = pantone3282!70, 
 urlbordercolor  = pantone3282!70, 
 pdfauthor 	 = Dorothea Jansen
}

\usepackage{tikz}
\usetikzlibrary{arrows,decorations.pathmorphing,backgrounds,positioning,fit}

\newtheorem{thm}{Theorem}[section]
  
\newaliascnt{prop}{thm}
  \newtheorem{prop}[prop]{Proposition}
  
  \aliascntresetthe{prop}
\newaliascnt{lemma}{thm}
  \newtheorem{lemma}[lemma]{Lemma}
  
  \aliascntresetthe{lemma}
\newaliascnt{cor}{thm}

  \aliascntresetthe{cor}
\newcommand{\precptnessThm}{Gromov's Pre-compactness Theorem\xspace}	
\newaliascnt{mainthm}{thm}
  \newtheorem{mainthm}[mainthm]{Theorem}
  
  \aliascntresetthe{mainthm}
\newtheorem*{mainthm*}{Main Theorem}
\newaliascnt{mainprop}{thm}
  \newtheorem{mainprop}[mainprop]{Proposition}
  
  \aliascntresetthe{mainprop}
\newtheorem*{mainprop*}{\autoref{prop:main}}
\newtheorem*{cnthm*}{Theorem\;\ref{lem_CN:Xgen}}
\newaliascnt{prodlemma}{thm}
  \newtheorem{prodlemma}[prodlemma]{Theorem}
  
  \aliascntresetthe{prodlemma}
\newcommand{\refProdLemma}{the Product Lemma\xspace}
\newaliascnt{segmineq}{thm}
  \newtheorem{segmineq}[segmineq]{Theorem}
  \aliascntresetthe{segmineq}
\newcommand{\refSegmIneq}{the Segment Inequality\xspace}	
\newaliascnt{bgthm}{thm}
  \newtheorem{bgthm}[bgthm]{Theorem}
  
  \aliascntresetthe{bgthm}
\newcommand{\refBGThm}{the Bishop-Gromov Theorem\xspace}
\newaliascnt{rescthm}{thm}
  \newtheorem{rescthm}[rescthm]{Theorem}
  
  \aliascntresetthe{rescthm}

\theoremstyle{definition}
\newaliascnt{defn}{thm}
  \newtheorem{defn}[defn]{Definition}
  
  \aliascntresetthe{defn} 
\newaliascnt{deflem}{thm}

  \aliascntresetthe{deflem} 
\newaliascnt{exm}{thm}

  \aliascntresetthe{exm} 

\newtheorem*{defn*}{Definition}
\newtheorem*{not*}{Notation}

\DeclareMathOperator{\nn}{\mathbb{N}}

\DeclareMathOperator{\rr}{\mathbb{R}}

\DeclareMathOperator{\eps}{\varepsilon}
\renewcommand{\phi}{\varphi}

\DeclareMathOperator{\diam}{diam}
\DeclareMathOperator{\rad}{rad}
\DeclareMathOperator{\vol}{vol}
\DeclareMathOperator{\dt}{\textit{dt}}
\DeclareMathOperator{\dtau}{\textit{d$\tau$}}
\DeclareMathOperator{\dV}{\textit{dV}}
\DeclareMathOperator{\Ric}{Ric}
\DeclareMathOperator{\Hess}{Hess}
\DeclareMathOperator{\Xgen}{\mathit{X}_{gen}}
\DeclareMathOperator{\Ygen}{\mathit{Y}_{gen}}
\DeclareMathOperator{\CBG}{\textit{C}_\textit{{BG}}}

\newcommand{\B}{\bar{B}} 

\newcommand{\dGH}{d_{\textit{GH}}}

\newcommand{\dgh}[5]{\dGH(B_{#1}^{#2}(#3),B_{#1}^{#4}(#5))} 
 
\newcommand{\dghpt}[5]{\dGH((#2,#3),(#4,#5)) \leq #1}

\def\limitspaces{\mathcal{X}^n}

\newcommand{\norm}[1]{\|#1\|}
\newcommand{\avgint}{{-}\hspace*{-1.05em}\int} 
\newcommand{\avgintsmall}{{-}\hspace*{-0.9em}\int} 
\DeclareMathOperator{\Mx}{Mx} 
\DeclareMathOperator{\pr}{pr} 
\DeclareMathOperator{\pt}{pt} 
\DeclareMathOperator{\prrrk}{\pr_{\rr^k}} 
\DeclareMathOperator{\id}{id} 
\DeclareMathOperator{\im}{im} 
\DeclareMathOperator{\C11}{{\color{black}\textit{C}_{{1-1}}(\textit{n})}} 
\def\limomega{\lim\nolimits_\omega}

\newcommand{\hateps}{\hat{\eps}}
\newcommand{\hatdelta}{\hat{\delta}}
\newcommand{\tildeeps}{\hatdelta}
\newcommand{\mydelta}{{\hateps}}
\newcommand{\myeps}{\delta}

\def\aand{\xspace\textrm{and}\xspace}
\def\as{\textrm{ as }}

\newcommand{\GH}{Gro\-mov-Haus\-dorff\xspace}

\newcommand{\myquote}[1]{\textquoteleft #1\textquoteright}

\newcommand{\ndim}{$n$-di\-men\-sion\-al\xspace}

\begin{document}

 
\begin{abstract}
 For collapsing sequences of Riemannian manifolds 
 which satisfy a uniform lower Ricci curvature bound
 it is shown that there is a sequence of scales
 such that for a set of good base points of large measure
 the pointed rescaled manifolds subconverge to a product of a Euclidean and a compact space. 
 All Euclidean factors have the same dimension,
 all possible compact factors satisfy the same diameter bounds
 and their dimension does not depend on the choice of the base point
 (along a fixed subsequence).
\end{abstract}


\maketitle

\setcounter{tocdepth}{1}
\tableofcontents


If a sequence of Riemannian manifolds satisfies a uniform lower sectional curvature bound,
this bound carries over to the possibly non-smooth Alexandrov limit.
If the limit is actually a smooth manifold, by Yamaguchi's Fibration Theorem, \cite{yamaguchi}, 
the manifolds fibre over the limit in the following way:
If $M_i$ is a sequence of \ndim manifolds 
with a uniform lower sectional curvature bound and a uniform upper diameter bound
converging to a compact manifold $N$ of lower dimension, then for sufficiently large $i\in\nn$ 
there are fibrations $M_i \to N$ which are close to Riemannian submersions. 
Moreover, up to a finite covering, each fibre is the total space of a fibration over a torus.

For proving the latter, a crucial argument is the following: 
Consider the pre-image of some ball under the fibration $M_i \to N$.
After rescaling the metric up, 
this pre-image converges to a product of a Euclidean and a compact space.
In fact, it is possible to replace the rescaling factors by larger ones such 
that the limit again has the form of such a product, 
but the Euclidean factor has higher dimension.
This can be iterated until, finally, the compact factor vanishes.
Such scaling factors are called \emph{typical scales}. 
Similar techniques have been used
by e.g.~Shioya and Yamaguchi in \cite{shioya-yamaguchi}
and Kapovitch, Petrunin and Tuschmann in \cite{kapovitch-petrunin-tuschmann}.

If a sequence of manifolds only satisfies a uniform lower Ricci curvature bound,
Yamaguchi's Fibration Theorem might fail.
This was proven by Anderson in \cite{anderson1992}.
However, in recent years Cheeger and Colding obtained 
deep structure results for limits of such sequences, 
\cite{cheeger-colding-I,cheeger-colding-II,cheeger-colding-III},
by using \emph{measured \GH convergence}:
After renormalizing the measure of the manifolds and passing to a subsequence,
the manifolds converge to a metric measure space 
such that the (renormalised) measures converge to the limit measure. 
The uniform lower Ricci curvature bound carries over to the limit
in the sense that the limit measure still satisfies the Bishop-Gromov Theorem.

Another difficulty occurring with only a lower Ricci curvature bound is the following:
Unlike for lower section curvature bounds, 
tangent cones of the limit space need not be metric cones. 
In fact, in the case of a collapsing sequence, 
the tangent cones at some point may depend on the choice of the rescaling sequence, 
cf.~\cite{cheeger-colding-I}.
However, Colding and Naber \cite{colding-naber} proved 
that any limit of a sequence of manifolds with uniform lower Ricci curvature bound
contains a connected subset of full measure 
such that for each point in this subset the tangent cone is unique 
and a Euclidean space of a fixed dimension $k \in \nn$.
This $k$ is called the \emph{dimension} of the limit space. 
Notice that this dimension is at most the Hausdorff-dimension of the space.
In particular, $k < n$.

If a collapsing sequence of manifolds $M_i$ satisfies 
the lower Ricci curvature bound $-\eps_i$, where $\eps_i \to 0$,
and if this sequence converges to a Euclidean space,
then the Rescaling Theorem of Kapovitch and Wilking in \cite{kapovitch-wilking}
already provides the existence of one sequence of typical scales. 
For this sequence, the blow-ups of the manifolds split 
into products of this Euclidean space and a compact factor.
The main theorem of this paper generalises this statement 
by allowing arbitrary limit spaces
and a uniform lower Ricci curvature bound.
\begin{mainthm*}
 Let $(M_i,p_i)_{i \in \nn}$ be a collapsing sequence 
 of pointed complete connected \ndim Riemannian manifolds
 which satisfy the uniform lower Ricci curvature bound $\Ric_{M_i} \geq -(n-1)$ 
 and converge to a limit $(X,p)$ of dimension $k < n$ 
 in the measured \GH sense. 
 Then for every $\eps \in (0,1)$ there exist a subset of good points
 $G_1(p_i) \subseteq B_1(p_i)$ 
 satisfying \[\vol(G_1(p_i)) \geq (1 -\eps) \cdot \vol(B_1(p_i)),\]
 a sequence $\lambda_i \to \infty$ and
 a constant $D > 0$
 such that the following holds: 
 
 For any choice of base points $q_i \in G_1(p_i)$ 
 and every sublimit $(Y,q)$ of $(\lambda_i M_i, q_i)_{i \in \nn}$ 
 there is a compact metric space $K$ of dimension $l \leq n-k$ 
 and diameter $\frac{1}{D} \leq \diam(K) \leq D$ 
 such that $Y$ splits isometrically as a product
 \[Y \cong \rr^k \times K.\] 
 Moreover, for any base points $q_i, q_i' \in G_1(p_i)$ 
 such that, after passing to a subsequence, both
 $(\lambda_i M_i, q_i) \to (\rr^k \times K, \cdot)$ 
 and $(\lambda_i M_i, q_i') \to (\rr^k \times K', \cdot)$ 
 as before,
 $\dim(K) = \dim(K')$.
\end{mainthm*}

Observe that $\dim(K) < n-k$ might occur in the situation of the theorem:
Consider the sequence of flat tori 
$M_i := S^1 \times S^1(\frac{1}{i}) \times S^1(\frac{1}{i^2})$
where $S^1(r)$ denotes a circle of radius $r>0$. 
In this example, it is easy to imagine the last 
two factors \myquote{collapsing to a point} in the limit, 
although the very last factor collapses faster than central one. 
Hence, $M_i$ converges to $S^1$. 
For $\lambda_i = i$,
the rescaled manifolds $\lambda_i M_i$ converge to $\rr \times S^1$.
Using the notation of the main theorem,
one has $n=3$, $k=1$ and $l = 1 < 2 = n-k$.

Furthermore, note that the theorem does not prove 
that all possible compact spaces need to have the same dimension,
but that the dimension only depends on the regarded subsequence 
and not on the choice of the base points:
Let $M_i = S^1(\frac{1}{i^2})$ for $i \in 2 \nn$, 
$M_i = S^1(\frac{1}{i})$ for $i \in 2\nn + 1$
and $\lambda_i = i$. 
Then $M_i$ collapses to a point, 
but $(\lambda_i M_i)_{i \in 2\nn} \to \rr^0 \times \{\pt\}$ 
and $(\lambda_i M_i)_{i \in 2\nn+1} \to \rr^0 \times S^1$. 

\par\medskip

Now turn to the proof of the main theorem.
Let $(M_i,p_i)_{i \in \nn}$ be a collapsing sequence as in the main theorem 
and let $(X,p)$ denote its limit.
First will be proven that
for points $q_i \in M_i$ and a small radius $r>0$
the conclusion of the main theorem holds correspondingly on $B_r(q_i)$ 
for a subset of good points $G_r(q_i)$ 
and a sequence of scales $\lambda_i \to \infty$,
i.e.~a \myquote{local} version of the main theorem will be established.

Recall that the set of good points $G_r(q_i) \in B_r(q_i)$ 
and the scales $\lambda_i \to \infty$ have to satisfy the following:
Any point $x_i \in G_r(q_i)$ needs to have the property 
that each sublimit of $(\lambda_i M_i, x_i)$ is the product 
of $\rr^k$ with a compact metric space
where the compact spaces (essentially) all have the same dimension.
This is achieved in two steps.

First, let $G_r^1(q_i) \subseteq B_r(q_i)$ denote the set of all points $x_i$ 
such that all sublimits of $(\mu_i M_i, x_i)$ split off an $\rr^k$-factor
where $\mu_i \to \infty$ is an arbitrary sequence.
In particular (once $\lambda_i \to \infty$ is constructed), 
for any $x_i \in G_r^1(q_i)$, 
any limit of the sequence $(\lambda_i M_i, x_i)$ splits off an $\rr^k$-factor.
That these sets $G_r^1(q_i)$ have large volume inside $B_r(q_i)$
is obtained by generalising results of Cheeger, Colding \cite{cheeger-colding-II} 
and Kapovitch, Wilking \cite{kapovitch-wilking}
involving modified distance functions coming from splitting theorems. 
For more details, see \autoref{sec:first_set}.

Next, define scales $\lambda_i \to \infty$ 
and another subset $G_r^2(q_i) \subseteq B_r(q_i)$ of large volume 
as the set of points $x_i$ 
such that $(\lambda_i M_i, x_i)$ has small distance to a product of $\rr^k$ 
and a compact metric space of diameter $1$.
The existence of such $\lambda_i \to \infty$ and $G_r^2(q_i)$ is obtained 
by using the Rescaling Theorem of Kapovitch and Wilking in \cite{kapovitch-wilking}.
For further explanations, see \autoref{sec:second_set}.

These scales $\lambda_i \to \infty$ 
and the intersection $G_r(q_i) = G_r^1(q_i) \cap G_r^2(q_i)$ 
give the splitting result in the local version of the main theorem.

In order to finish the local version, prove that two such limits have the same dimension.
First, the following special case is investigated: 
Suppose that two points $x_i, y_i \in G_r(q_i)$ are connected 
by an integral curve of a vector field whose flow is measure preserving and bi-Lipschitz 
(on a set of large enough volume).
In this situation, via \GH convergence one obtains a bi-Lipschitz map 
between subsets of positive volume inside of the limit spaces 
of $(\lambda_i M_i, x_i)$ and $(\lambda_i M_i, y_i)$, respectively.
In particular, these limits need to have the same dimension.

For the general situation,
recall that the flow of any divergence-free vector field is measure preserving.
Moreover, 
using (slight generalisations of) results in \cite{kapovitch-wilking},
it is bi-Lipschitz 
if its derivative 
is small---in some $L_\alpha$-norm, $\alpha > 1$, 
and taking some average value (in volume sense).
In fact every two points of the set $G_r^2(q_i)$ are connected 
by the curve of such a vector field 
(or are at least sufficiently close to its start and end point).
For more details, see \autoref{sec:C=>dim}.

\par\medskip
For verifying the main theorem, i.e.~defining $G_1(p_i)$ and $\lambda_i \to \infty$,
fix $r>0$ and finitely many sequences $(q_i)_{i \in \nn}$
such that the union of the subsets of good points $G_r(q_i) \subseteq B_r(q_i)$ 
has sufficiently large volume in $B_1(p_i)$.
Define $G_1(p_i)$ as the union of these $G_r(q_i)$. 
Now difficulties arise since each sequence $(q_i)_{i \in \nn}$ provides 
its own sequence of scales $\lambda_i \to \infty$
and these sequences might be pairwise different, 
but the main theorem requires only a single sequence of scales. 
This problem can be solved if, 
given two such sequences $(q_i^1)_{i \in \nn}$ and $(q_i^2)_{i \in \nn}$ 
and their corresponding scales $\lambda_i^1 \to \infty$ and $\lambda_i^2 \to \infty$, 
the local version of the main theorem still holds 
for $(q_i^2)_{i \in \nn}$ and $\lambda_i^1 \to \infty$ 
(instead of $\lambda_i^2 \to \infty$).
Indeed, this is achieved by a clever choice of the finitely many $(q_i)_{i \in \nn}$ 
utilizing the H\"older continuity result of Colding and Naber \cite{colding-naber}.
For more intuitive details on this approach, see the introduction to \autoref{sec:step2}.

\par\bigskip
This paper is structured as follows: 
First, notation is fixed in \autoref{sec:background}.
The subsequent sections deal with the proof of the main theorem:
First, \autoref{sec:step1} proves the above mentioned local version of the main theorem:
If points $q_i$ and some small $r>0$ satisfy 
that the rescaled manifolds $(\frac{1}{r} M_i, q_i)$ are sufficiently close to $\rr^k$,
then the statement of the main theorem holds on the ball $B_r(q_i)$ analogously.
As explained before, using finitely many of such sequences $q_i$ 
and taking the union of the obtained subsets $G_r(q_i) \subseteq B_r(q_i)$ 
will be used in \autoref{sec:step2} to prove the main theorem.


\par\medskip\noindent\textit{Acknowledgements.} 
The results of this paper are part of the author’s PhD thesis\footnote{Available at
\href{https://miami.uni-muenster.de/Record/a23e84b7-fb36-41b8-8d4d-5af83a15816c}
{https://miami.uni-muenster.de}.}. 
The author would like to thank her advisor Burkhard Wilking for 
his guidance and expertise.


\section{Notation}\label{sec:background}
For the sake of clarifying notation, recall the following theorem.
\begin{bgthm}[Bishop-Gromov Theorem, {\cite[Chapter 9, Lemma 1.6]{petersen}}]
 \label{lem_background:BG}
 Let $M$ be a complete \ndim Riemannian manifold 
 with lower Ricci curvature bound $\Ric_M \geq (n-1) \cdot \kappa$ for some $\kappa \in \rr$ 
 and let $p \in M$.
 Then the map 
 \[ r \mapsto \frac{\vol_M(B_r(p))}{V^n_\kappa(r)}\]
 is monotonically decreasing with limit $1$ as $r \to 0$, 
 where $V^n_\kappa(r)$ is the volume of an $r$-ball 
 in the \ndim space form of sectional curvature $\kappa$.
 
 In particular, for $R \geq r > 0$, 
 \[\vol_M (B_R(p)) \leq \frac{V^n_\kappa(R)}{V^n_\kappa(r)} \cdot \vol_M(B_r(p)).\]
 This factor is independent of $M$ and denoted by
 \[\CBG(n,\kappa,r,R) := \frac{V^n_\kappa(R)}{V^n_\kappa(r)}.\] 
\end{bgthm}

Throughout this paper, the Bishop-Gromov Theorem will always be applied 
using the notion $\CBG(n,\kappa,r,R)$.
Regarding $\CBG$ as a function of radii or curvature bound, it has the following properties:
Fix any $n \in \nn$, $c \geq 1$, $y > 0$, $-1 \leq \kappa \leq 0$ and $R \geq r > 0$.
Then
$\CBG(n,-1,r,cr) \leq \CBG(n,-1,R,cR)$,
$\CBG(n,\kappa,r,R) \leq \CBG(n,1,r,R)$ and 
$\lim_{x \to \infty} \CBG(n,-1,x,x+y) = e^{(n-1)y}$.

\par\bigskip

In the following, all metric spaces are assumed to be proper, i.e.~closed balls are compact.
Recall that a sequence $(X_i,p_i)$ of pointed metric spaces converges to $(X,p)$ 
(in the pointed \GH sense)
if \[\dgh{r}{X}{p}{Y}{q} \to 0\] for all $r > 0$
where $\dgh{r}{X}{p}{Y}{q}$ denotes the pointed \GH distance 
of the (compact) balls $(\B_r(p),p)$ and $(\B_r(q),q)$. 
In this case, $(X,p)$ is called \emph{(pointed \GH) limit} of $(X_i,p_i)$.
If $(X,p)$ occurs as limit of a convergent subsequence, it is called \emph{sublimit},
and the sequence $(X_i,p_i)$ is said to \emph{subconverge}.
Note that limits are unique up to isometry.

If $\dgh{1/\eps}{X}{p}{Y}{q} \leq \eps$ for some $\eps > 0$,
this is denoted by \[\dghpt{\eps}{X}{p}{Y}{q}.\]

\par\medskip
Recall the following correspondence of \GH convergence 
and \GH approximations
where maps $f : X \to Y$ and $g : Y \to X$ 
between compact metric spaces 
are called \emph{($\eps$-)\GH approximations} 
or \emph{$\eps$-ap\-prox\-i\-ma\-tions between the spaces $(X,p)$ and $(Y,q)$} if
$f(p) = q$, $g(q) = p$
and the following holds for all $x,x_1,x_2 \in X$ and $y,y_1,y_2 \in Y$:
\begin{align*}
 |d_X(x_1,x_2) - d_Y(f(x_1),f(x_2))| < \eps, &\quad\quad d_X(g \circ f(x), x) < \eps, \\
 |d_Y(y_1,y_2) - d_X(g(y_1),g(y_2))| < \eps, &\quad\quad d_Y(f \circ g(y), y) < \eps.  
\end{align*}

\begin{prop}\label{lem_background:GH_convergence_properties_I}
 Let $(X,d_X,p)$ and $(X_i,d_{X_i},p_i)$, $i \in \nn$, be pointed metric spaces.
 If the $X_i$ and $X$ are length spaces, the following are equivalent:
 \begin{enumerate}
  \item 
   $(X_i,p_i) \to (X,p)$. 
  \item\label{lem_background:GH_convergence_properties_I--convergence_iff_dgh_small--appr} 
   There is a sequence $\eps_i \to 0$ and $\eps_i$-ap\-prox\-i\-ma\-tions $(f_i,g_i)$ 
   between the balls $(\B_{1/\eps_i}^{X_i}(p_i),p_i)$ and $(\B_{1/\eps_i}^{X}(p),p)$.
  \item\label{lem_background:GH_convergence_properties_I--convergence_iff_dgh_small--dgh} 
   For any function $g: \rr^{> 0} \to \rr^{> 0}$ with $\lim_{x \to 0} g(x) = 0$
   there exists $\eps_i \to 0$ with $\dgh{1/\eps_i}{X_i}{p_i}{X}{p} \leq g(\eps_i)$.
 \end{enumerate}
\end{prop}
For convergent spaces $(X_i,p_i) \to (X,p)$,
maps as in \ref{lem_background:GH_convergence_properties_I--convergence_iff_dgh_small--appr} 
are always implicitly fixed.
In this situation, recall that a sequence of points $q_i \in \B^{X_i}_{1/\eps_i}(p_i)$ 
is said to \emph{converge} to a point $q \in X$, denoted by $q_i \to q$, 
if $f_i(q_i)$ converges to $q$ (in $X$).
In particular, $p_i^x := g_i(x) \to x$. 
In this situation, $(X_i,p_i^x) \to (X,x)$ as well.
\par\medskip

Further, note the following fact:
If $(X,p)$ and $(Y,q)$ are pointed length spaces and $R \geq r > 0$, 
then $\dgh{r}{X}{p}{Y}{q} \leq 16 \cdot \dgh{R}{X}{p}{Y}{q}$.
In particular,
convergence $\dgh{R}{X_i}{p_i}{X}{p} \to 0$ for some $R>0$ implies convergence
$\dgh{r}{X_i}{p_i}{X}{p} \to 0$ for all $r\leq R$.

\par\bigskip
For the sake of simpler notation,
instead of passing to a subsequence
occasionally ultralimits will be used.
Recall that a \emph{non-principal ultrafilter on $\nn$}
is a finitely additive probability measure $\omega$ on $\nn$ such that 
all subsets $S \subseteq \nn$ are $\omega$-measurable 
with value $\omega(S) \in \{0,1\}$ and $\omega(S) = 0$ if $S$ is finite,
and that for any bounded sequence $(a_i)_{i \in \nn}$ of real numbers
there exists a unique real number $\limomega a_i$ such that
\[\omega( \{i \in \nn \mid |a_i - \limomega a_i| < \eps \}) = 1\]
for every $\eps > 0$. 

Given a non-principal ultrafilter $\omega$ and 
pointed metric spaces $(X_i,d_{X_i},p_i)$, 
the metric space $\limomega (X_i,p_i) := (X_\omega, d_\omega)$
is called \emph{ultralimit} of $(X_i,p_i)$ 
where
\[X_\omega := \{ [(x_i)_{i \in \nn}] 
\mid x_i \in X_i~\aand~\sup\nolimits_{i \in \nn} d_{X_i}(x_i,p_i) < \infty\}\]
for 
$(x_i)_{i \in \nn} \sim (y_i)_{i \in \nn}$ if and only if $\limomega d_{X_i}(x_i,y_i) = 0$
and \[d_\omega([(x_i)_{i \in \nn}],[(y_i)_{i \in \nn}]) := \limomega d_{X_i}(x_i,y_i).\]

\begin{lemma}\label{lem_background:ultralimits}
 Let $(X_i,d_{X_i},p_i)$ and $(Y_i,d_{Y_i},q_i)$, $i \in \nn$, be pointed length spaces.
 \begin{enumerate}
  \item 
  Let $\omega$ be a non-principal ultrafilter on $\nn$.
  Then $\lim_\omega(X_i,p_i)$ is a sublimit in the pointed \GH sense.
  Concretely, there exists a subsequence $(i_j)_{j \in \nn}$ such that both
  \[(X_{i_j},p_{i_j}) \to \limomega (X_i,p_i) 
  \quad\aand\quad
  (Y_{i_j},q_{i_j}) \to \limomega (Y_i,q_i)\]
  as $j \to \infty$ in the pointed \GH sense.
 \item 
  The sublimit of a sequence of pointed length spaces in the pointed \GH sense
  is the ultralimit with respect to a non-principal ultrafilter.
  To be more precise: 
  If $(X,d_X,p)$ and $(Y,d_Y,q)$ are pointed length spaces 
  and $(i_j)_{j \in \nn}$ is a subsequence such that both
  \[(X_{i_j},p_{i_j}) \to (X,p)
  \quad\aand\quad
  (Y_{i_j},q_{i_j}) \to (Y,q)\]
  as $j \to \infty$ in the pointed \GH sense,
  then there exists a non-principal ultrafilter $\omega$ on $\nn$ 
  such that there are isometries
  \[\limomega (X_i,p_i) \cong (X,p) \quad\aand\quad \limomega (Y_i,q_i) \cong (Y,q).\]
 \end{enumerate}
\end{lemma}

\par\bigskip

Let $(M_i, p_i)_{i \in \nn}$ be a sequence 
of pointed complete connected \ndim Riemannian manifolds
with lower Ricci curvature bound $\Ric_{M_i} \geq -(n-1)$.
If $\vol_{M_i}(B_1(p_i)) \to 0$ as $i \to \infty$, 
this sequence is said to be \emph{collapsing}.
In this situation, \emph{renormalised limit measures} are used, 
cf.~\cite[section 1]{cheeger-colding-I}: 
Let $(M_i, p_i)_{i \in \nn}$ be a collapsing sequence as above. 
Then $(M_i, p_i)$ subconverges to a metric space $(X,p)$ 
such that a \myquote{renormalisation} of the measures $\vol_{M_i}$ 
converges to a limit measure $\vol_X$.
In fact, the following is true.

\begin{thm}[{\cite[Theorem 1.6, Theorem 1.10]{cheeger-colding-I}}]
\label{thm-background:precompactness-measured}
 Let $(M_i,p_i)_{i \in \nn}$ be a sequence 
 of pointed complete connected \ndim Riemannian manifolds 
 satisfying the uniform lower Ricci curvature bound $\Ric_{M_i} \geq -(n-1)$. 
 Then $(M_i,p_i)$ subconverges to a metric space $(X,p)$ in the pointed \GH sense 
 and there exists a Radon measure $\vol_X$ on $X$ 
 such that for all $x \in X$, $x_i \to x$ and $r > 0$,
 \[
  \frac{\vol_{M_i}(B_r^{M_i}(x_i))}{\vol_{M_i}(B_1^{M_i}(p_i))} 
  \to \vol_X(B_r^X(x)) 
  \as i \to \infty.
 \]
 Moreover, for any $R \geq r > 0$ and $x \in X$,
 \[ \frac{\vol_X(B_R^X(x))}{\vol_X(B_r^X(x))} \leq \CBG(n,-1,r,R).\]
 This $\vol_X$ is called \emph{(renormalised) limit measure}.
\end{thm}

Observe that the limit measure of a sequence $(M_i,p_i)$ 
depends on the choice of the base points and the considered subsequence, 
cf.~again \cite[section 1]{cheeger-colding-I}.
Moreover, observe the following: 
\precptnessThm ensures subconvergence 
(for pointed Riemannian manifolds of the same dimension 
with a lower Ricci curvature bound),
but the above theorem guarantees more, 
namely subconvergence (for the same class) 
including convergence of the (renormalised) measures.
Throughout this paper, 
only measured \GH convergence will be used,
i.e.~whenever a sequence $(M_i,p_i)_{i \in \nn}$ converges to a limit space $(X,p)$,
this limit is equipped with a measure $\vol_X$ as in the above theorem.
\par\medskip

Let the complete pointed metric space $(X,p)$ be the pointed \GH limit of 
a sequence of pointed connected \ndim Riemannian manifolds $(M_i, p_i)$ 
satisfying the uniform lower Ricci curvature bound $\Ric_{M_i} \geq -(n-1)$.
As introduced in \cite[p.~408]{cheeger-colding-I}, 
a \emph{tangent cone} at $x \in X$ 
is a \GH limit of $(\lambda_i X, x)$ 
where $\lambda_i \to \infty$ as $i \to \infty$. 
In general, this limit depends on the choice of $x \in X$ 
and the sequence $\lambda_i \to \infty$.
If the limit is independent of the choice of $\lambda_i \to \infty$, 
it is denoted by $C_x X$.
If $C_x X = \rr^k$, this point $x$ is called \emph{$k$-regular}
and the set of all $k$-regular points is denoted by $\mathcal{R}_k$.
Furthermore, $\mathcal{R} = \bigcup_k \mathcal{R}_k$ denotes the set of all regular points.

Moreover, Cheeger and Colding proved that there are points 
such that non-unique tangent cones of different dimensions occur, 
cf.~\cite[Example 8.80]{cheeger-colding-I}.
In particular, there are points that are not regular.
However, they proved that for any renormalised limit measure 
the complement $X \setminus \mathcal{R}$ has measure $0$, 
i.e.~almost all points are regular, 
cf.~\cite[Theorem 2.1]{cheeger-colding-I}.
Even more, it was conjectured that there is some $k$ 
such that $\mathcal{R} \setminus \mathcal{R}_k$ has measure $0$ as well, 
i.e.~almost all points are $k$-regular.
This conjecture was proven by Colding and Naber in \cite{colding-naber}.

\begin{thm}[{\cite[Theorem 1.18 and p.~1185]{colding-naber}}]\label{lem_CN:Xgen}
 Let $(M_i,p_i)_{i \in \nn}$ be a sequence 
 of pointed complete connected \ndim Riemannian manifolds
 which satisfy the uniform lower Ricci curvature bound $\Ric_{M_i} \geq -(n-1)$ 
 and converge to a limit $(X,p)$. 
 Then there is $k=k(X) \in \nn$ such that $\mathcal{R}_k$ has full measure and is connected. 
\end{thm}
 
 This $k$ is called the \emph{dimension of $X$},
 a $k$-regular point is called \emph{generic}
 and $\Xgen := \mathcal{R}_k$ denotes the set of all generic points.
Note that $k < n$ if the sequence is collapsing.


\section{Local construction}\label{sec:step1}
For a collapsing sequence of pointed complete connected 
\ndim Riemannian mani\-folds $(M_i,p_i)$ 
satisfying the uniform lower Ricci curvature bound $\Ric_{M_i} \geq -(n-1)$, 
the main proposition of this section 
provides a condition on points $q_i \in M_i$ 
such that on balls around these points with sufficiently small radius 
a \myquote{local version} of the main theorem holds:
In fact, the statement of the main theorem holds on $B_r(q_i)$
if the rescaled manifolds $(\frac{1}{r} M_i, q_i)$ are sufficiently close 
to the Euclidean space.
Applying this result to finitely many sequences 
of such points $q_i$ and radii $r$
will prove the main theorem in \autoref{sec:step2}.

This local result follows from generalising several theorems
of Cheeger and Colding in \cite{cheeger-colding-96,cheeger-colding-II} 
and Kapovitch and Wilking in \cite{kapovitch-wilking}. 
Those results make statements assuming a sequence of manifolds 
to converge to a Euclidean space $\rr^k$. 
The generalisations do not assume such a convergence
but that the manifolds are sufficiently close to $\rr^k$,
and then make similar statements as the mentioned theorems. 

In the situation of a sequence $(M_i,p_i)$ 
converging to a limit $(X,p)$ as in the main proposition, 
there is no reason why the manifolds should already be 
sufficiently close to $\rr^k$.
On the other hand, there is hope that this is true 
after rescaling all manifolds with (the same) factor: 
For a generic point $x \in X$, 
the rescaled limit space $(\lambda X,x)$ 
converges to $\rr^k$ as $\lambda \to \infty$.
In particular, $(\lambda X,x)$ is close to $\rr^k$ 
for sufficiently large $\lambda >0$.
Moreover, given any sequence of points $x_i \in M_i$ converging to $x \in X$,
the equally rescaled manifolds $(\lambda M_i, x_i)$ 
converge to $(\lambda X, x)$.
Hence, the $(\lambda M_i, x_i)$ are close to $\rr^k$ 
for sufficiently large $\lambda$ and $i$.

So, one can expect to be able to use the above explained generalisations 
for the rescaled manifolds.
In fact, those (generalised) theorems make statements 
about balls of radius $1$. 
Applied to the rescaled manifolds $\lambda M_i$, 
this corresponds to balls of radius $\frac{1}{\lambda}$ 
in the unscaled manifolds $M_i$.
Thus, in the following, instead of $\lambda$ 
the notation $\frac{1}{r}$ will be used, 
where $r > 0$ is sufficiently small,
and statements about balls of radius $r$ will be obtained.
This leads to the following local version of the main theorem,
where the choice of notation $\mydelta$ 
and $\tildeeps$---while seemingly artificial---will turn out to be helpful
when proving the main theorem by applying the \myquote{local version}.
\begin{mainprop}\label{prop:main} 
 Let $(M_i)_{i \in \nn}$ be a sequence 
 of complete connected \ndim Riemannian manifolds
 with uniform lower Ricci curvature bound $\Ric_{M_i} \geq -(n-1)$, 
 and let $k < n$. 
 Given $\mydelta \in (0,1)$, 
 there is $\tildeeps =\tildeeps(\mydelta;n,k) > 0$ 
 such that for any $0 < r \leq \tildeeps$ 
 and $q_i \in M_i$ with \[\dghpt{\tildeeps}{r^{-1} M_i}{q_i}{\rr^k}{0}\]
 there are 
 a family of subsets of good points $G_{r}(q_i) \subseteq B_r(q_i)$ 
 with \[\vol(G_r(q_i)) \geq (1 - \mydelta) \cdot \vol( B_r(q_i))\] and
 a sequence $\lambda_i \to \infty$ 
 such that the following holds:
 \begin{enumerate}
 \item For every choice of base points $x_i \in G_r(q_i)$ 
 and every sublimit $(Y,\cdot)$ of $(\lambda_i M_i, x_i)$ 
 there exists a compact metric space $K$ of dimension $l \leq n-k$
 satisfying $\frac{1}{5} \leq \diam(K) \leq 1$ 
 such that $Y$ splits isometrically as a product
 \[Y \cong \rr^k \times K.\]
 \item If $x_i^1, x_i^2 \in G_r(q_i)$ are base points such that, 
 after passing to a subsequence,
 \[(\lambda_i M_i, x_i^j) \to (\rr^k \times K_j, \cdot)\] 
 for $1 \leq j \leq 2$ as before,
 then $\dim(K_1) = \dim(K_2)$.
 \end{enumerate}
\end{mainprop}

The idea of the proof is to construct 
two families of sets $G_r^1(q_i)$ and $G_r^2(q_i)$, 
where $r$ is sufficiently small, 
with the following properties:
For any choice of points $x_i \in G_r^1(q_i)$ 
and for any rescaling sequence $\lambda_i \to \infty$, 
every (sub)limit of the sequence $(\lambda_i M_i, x_i)$ splits off 
an $\rr^k$-factor.
The second family of sets $G_r^2(q_i)$ is constructed 
together with a rescaling sequence $\lambda_i \to \infty$ 
such that for all large enough $i$ and any point $x_i \in G_r^2(q_i)$
each single rescaled manifold $(\lambda_i M_i, x_i)$ is close 
to the product of $\rr^k$ and a compact space,
where this compact space depends on the choice of the regarded $i$ 
and the base point $x_i$.
After fixing this sequence $\lambda_i \to \infty$,
the intersection of those two sets gives the result.

This section is structured as follows: 
First, $G_r^1(q_i)$ and $G_r^2(q_i)$ are constructed in
\autoref{sec:first_set} and \autoref{sec:second_set}, respectively.
Next, \autoref{sec:C=>dim} establishes a basis for proving that blow-ups have the same dimension.
Finally, \autoref{prop:main} is proven in \autoref{sec:proof_main_prop}.

\subsection{Construction of \texorpdfstring{$G_r^1(q_i)$}{first set}}
\label{sec:first_set}
This subsection deals with finding families of subsets 
$G_r^1(q_i) \subseteq B_r^1(q_i)$ 
such that all blow-ups of $M_i$ with base points from $G_r^1(q_i)$ 
split off an $\rr^k$-factor. 
Recall that a blow-up is the limit of the sequence 
of rescaled manifolds $\mu_i M_i$ for a sequence of scales $\mu_i \to \infty$.
Thus, the natural question is 
under which condition such a splitting can be guaranteed. 

By modifying certain distance functions, 
Cheeger and Colding obtained harmonic functions 
which were used to prove the following splitting theorem.
\begin{thm}
[{\cite[Theorem 6.64]{cheeger-colding-96}}]\label{CC_splitting_thm}
 Let $(M_i,p_i)_{i \in \nn}$ be a sequence 
 of pointed \ndim Riemannian manifolds
 and let $R_i \to \infty$ and $\eps_i \to 0$ be 
 sequences of positive real numbers
 such that ${B_{R_i}^{M_i}(p_i)}$ 
 has Ricci curvature at least $- (n-1) \cdot \eps_i$.
 Assume $(B_{R_i}^{M_i}(p_i),p_i)$ to converge 
 to some pointed metric space $(Y,y)$ in the pointed \GH sense.
 
 If $Y$ contains a line,
 then $Y$ splits isometrically as $Y \cong \rr \times Y'$.
\end{thm}

Assuming that the limit space already is the Euclidean space $\rr^n$ 
(of the same dimension as the manifolds of the convergent sequence),
Colding proved convergence 
of the volume of balls of radius $1$ in the manifolds 
to the volume of the $1$-ball in $\rr^n$
by using $n$ of such function (for every manifold), 
cf.~\cite[Lemma 2.1]{colding-97}.
Using both the observations there and the proof 
of the above splitting theorem,
Cheeger and Colding obtained the following statement
which is stated here as noted in \cite[Theorem 1.3]{kapovitch-wilking}.

\begin{thm}
[{\cite[section 1]{cheeger-colding-II}}]
\label{CC_existence_harmonic_functions}
 Let $(M_i, p_i) \to (\rr^k,0)$ be a sequence 
 of pointed \ndim Riemannian manifolds
 which satisfy the uniform lower Ricci curvature bound 
 $\Ric_{M_i} \geq -\frac{1}{i}$. 
 Then there are harmonic functions $b^i_1, \dots, b^i_k: B_2(p_i) \to \rr$ 
 and a constant $C(n) \geq 0$ such that 
 \begin{enumerate}
 \item 
  $|\nabla b^i_j| \leq C(n)$ for all $i$ and $j$ and
 \item 
  $\avgintsmall_{B_1(p_i)} 
  \sum_{j,l=1}^k |\langle \nabla b^i_j,\nabla b^i_l\rangle - \delta_{jl}| 
  + \sum_{j=1}^k \norm{\Hess b^i_j}^2 \dV_{M_i} 
  \to 0$ 
  as $i \to \infty$.
 \end{enumerate}
 Moreover, the maps $\Phi^i = (b^i_1,\dots,b^i_k): B_2(p_i) \to \rr^k$ 
 provide $\eps_i$-\GH approximations 
 between $B_1(p_i)$ and $B_1(0)$ with $\eps_i \to 0$.
\end{thm}

Conversely, Kapovitch and Wilking proved 
the following in \cite{kapovitch-wilking}:
If there exist $k$ functions with analogous properties 
on balls with radii $r_i \to \infty$,
then the sublimit splits off an $\rr^k$-factor.

\begin{prodlemma}
[Product Lemma, {\cite[Lemma 2.1]{kapovitch-wilking}}]\label{product-lemma}
 Let $(M_i,p_i)_{i \in\nn}$ be a pointed sequence 
 of \ndim manifolds with $\Ric_{M_i} > - \eps_i$ 
 for a sequence $\eps_i \to 0$ and let $r_i \to \infty$ 
 such that $\B_{r_i}(p_i)$ is compact for all $i \in \nn$. 
 Assume for every $i \in \nn$ and $1 \leq j \leq k$ 
 there are harmonic functions $b^i_j : B_{r_i}(p_i) \to \rr$ 
 which are $L$-Lipschitz and fulfil
 \[
  \avgint_{B_R(p_i)} 
  \sum_{j,l=1}^k |\langle \nabla b^i_j, \nabla b^i_l \rangle - \delta_{jl}| 
  + \sum_{j=1}^k \norm{\Hess b^i_j}^2 \dV_{M_i} 
  \to 0 
  \text{ for all } R>0. 
 \]
 Then $(B_{r_i}(p_i),p_i)$ subconverges in the pointed \GH sense 
 to a metric product $(\rr^k \times X, p_{\infty})$ for some metric space $X$. 
 Moreover, $(b^i_1, \dots, b^i_k)$ converges 
 to the projection onto the Euclidean factor.
\end{prodlemma}

The above theorems will be generalised to the following statements:
If all manifolds are sufficiently close to $\rr^k$,
then there exist harmonic functions similar to those of
\autoref{CC_existence_harmonic_functions}
such that the average integral does not converge to zero but only is bounded,
cf.~\autoref{prop-M}.
Consequently, an adaptation of \refProdLemma will be established:
Under the following weaker hypothesis, the same conclusion holds, 
cf.~\autoref{gen_product_lemma}: 
Only the average integral about the norms of the Hessian vanishes 
when passing to the limit
whereas the average integrals about the scalar products of the gradients 
are bounded by a small constant.

First,
maps similar to those in \autoref{CC_existence_harmonic_functions} 
will be constructed.
A crucial step of the proof will be the rescaling of maps. 
\begin{lemma}\label{prop-M} 
 Given $n \in \nn$, there exists $L = L(n) \geq 0$ 
 such that the following holds:
 For arbitrary $\mydelta > 0$, $R>0$, $k \leq n$ and $g: \rr^+ \to \rr^+$ 
 with $\lim_{x \to 0} g(x) = 0$ 
 there exists $\tildeeps = \tildeeps(\mydelta,g,R;n,k) \in (0,1)$ 
 such that the following is true for every $\myeps \leq \tildeeps$:
 For every pointed complete connected \ndim 
 Riemannian manifold $(M,p)$
 with $\Ric_M \geq -(n-1) \cdot \myeps^2$ and 
 \[\dghpt{g(\myeps)}{M}{p}{\rr^k}{0}\]
 there exist harmonic functions $f_1, \dots, f_k : B_R^M(p) \to \rr$ 
 such that $|\nabla f_j| \leq L$ and
 \[
  \avgint_{B_R^M(p)} 
  \sum_{j,l=1}^k |\langle \nabla f_j, \nabla f_l\rangle - \delta_{jl}| 
  + \sum_{j=1}^k \norm{\Hess(f_j)}^2 \dV_M 
  < \mydelta.
 \]
\end{lemma}

\begin{proof}
 Let $L:=C(n)$ be the constant of \autoref{CC_existence_harmonic_functions}.
 The proof is done by contradiction:
 Assume the statement is false 
 and let $\mydelta$, $R$, $k$ and $g$ be contradicting. 
 For every $i \in \nn$, 
 let $\tildeeps_i := (i\cdot(n-1))^{-1/2} \in (0,1)$ and 
 choose the contradicting $\myeps_i \leq \tildeeps_i$ and $(M_i,p_i)$ with
 $\Ric_{M_i} \geq - (n-1) \cdot \myeps_i^2 \geq - \frac{1}{i}$ 
 and
 $\dghpt{g(\myeps_i)}{M_i}{p_i}{\rr^k}{0}$ 
 such that for all harmonic Lipschitz maps 
 $f^i_1, \dots, f^i_k : B_1^{M_i}(p_i) \to \rr$ with $|\nabla f^i_j| \leq L$,
 \[
  \avgint_{B_1^{M_i}(p_i)} 
  \sum_{j,l=1}^k |\langle \nabla f^i_j, \nabla f^i_l\rangle - \delta_{jl}| 
  + \sum_{j=1}^k \norm{\Hess(f^i_j)}^2 \dV_{M_i} 
  \geq \mydelta.
 \]
    
 Since $g(\myeps_i) \to 0$ as $i \to \infty$,
 $(M_i,p_i) \to (\rr^k,0)$
 and so $(\frac{1}{R} \, M_i,p_i) \to (\rr^k,0)$ as well. 
 By \autoref{CC_existence_harmonic_functions}, 
 there are harmonic functions 
 $\tilde{f}^i_j: B_1^{R^{-1} M_i}(p_i) \to \rr$, $1 \leq j \leq k$, 
 with $|\nabla \tilde{f}^i_j| \leq L$ and 
 \[
  \avgint_{B_1^{R^{-1} M_i}(p_i)} 
  \sum_{j,l=1}^k |\langle \nabla \tilde{f}^i_j, 
    \nabla \tilde{f}^i_l\rangle - \delta_{jl}| 
  + \sum_{j=1}^k \norm{\Hess(\tilde{f}^i_j)}^2 \dV_{R^{-1} M_i} 
  \to 0 \as {i \to \infty}.
 \]
 In particular, the rescaled functions $f^i_j := R \cdot \tilde{f}^i_j : B_R^{M_i} \to \rr$
 satisfy $|\nabla f^i_j| \leq L$ and
 \begin{align*}
  &\avgint_{B_R^{M_i}(p_i)} 
  \sum_{j,l=1}^k |\langle \nabla {f}^i_j, \nabla {f}^i_l\rangle - \delta_{jl}| 
  + \sum_{j=1}^k \norm{\Hess({f}^i_j)}^2 \dV_{M_i} 
  \\
  &= \avgint_{B_1^{R^{-1} M_i}(p_i)} 
  \sum_{j,l=1}^k |\langle \nabla \tilde{f}^i_j, 
    \nabla \tilde{f}^i_l\rangle - \delta_{jl}| 
  + \frac{1}{R^2} 
  \cdot \sum_{j=1}^k \norm{\Hess(\tilde{f}^i_j)}^2 \dV_{M_i} 
  \displaybreak[0]\\
  &\leq \Big(1 + \frac{1}{R^2}\Big) 
  \\&\quad
  \cdot \Big(\, \avgint_{B_1^{R^{-1} M_i}(p_i)} 
  \sum_{j,l=1}^k |\langle \nabla \tilde{f}^i_j, 
  \nabla \tilde{f}^i_l\rangle - \delta_{jl}| 
   + \sum_{j=1}^k \norm{\Hess(\tilde{f}^i_j)}^2 \dV_{M_i} \Big) 
   \\
  &\to 0 \as {i \to \infty}.
 \end{align*}
 This is a contradiction.
\end{proof}

In order to generalise \refProdLemma,
the following result of Cheeger and Colding is used.
Again, the theorem is stated using the notation of 
\cite[Theorem 1.5]{kapovitch-wilking}.
\begin{segmineq}
[Segment Inequality, {\cite[Theorem 2.11]{cheeger-colding-96}}]
\label{segment-inequality}
 Given any dimension $n \in \nn$ and radius $r_0 > 0$, 
 there exists $\tau = \tau(n,r_0)$ such that the following holds:
 Let $M$ be an \ndim Riemannian manifold 
 which satisfies the lower Ricci curvature bound $\Ric_{M} \geq -(n-1)$ 
 and $g: M \to \rr^+$ be a non-negative function.
 Then for $r \leq r_0$,
 \[
  \avgint_{B_r(p) \times B_r(p)} \int_0^{d(z_1,z_2)} 
  g(\gamma_{z_1,z_2}(t)) \dt \dV(z_1,z_2) 
  \leq \tau \cdot r \cdot \avgint_{B_{2r}(p)} g(q) \dV(q),
 \]
 where $\gamma_{z_1,z_2}$ denotes a minimising geodesic from $z_1$ to $z_2$.
\end{segmineq}

The following lemma is a generalisation of \refProdLemma
where the average integral of scalar products of the gradients 
does not have to vanish, but only needs to be bounded.
\begin{lemma}\label{gen_product_lemma}
 Let $(M_i)_{i \in \nn}$ be a sequence 
 of connected \ndim Riemannian manifolds
 with $\Ric_{M_i} \geq -(n-1) \cdot \eps_i$ 
 where $\eps_i \to 0$.
 Let $r_i \to \infty$ and $q_i \in M_i$ be points 
 such that the balls $\B_{r_i}(q_i)$ are compact. 
 Furthermore, let $k \leq n$ 
 and assume that for every $1 \leq j \leq k$ 
 there is a harmonic $L$-Lipschitz map 
 $b^i_j: B_{r_i}(q_i) \to \rr$ satisfying
 $\avgintsmall_{B_r(q_i)} \sum_{j=1}^k \norm{\Hess(b^i_j)}^2 \dV \to 0$ 
 and 
 $\avgintsmall_{B_r(q_i)} \sum_{j,l=1}^k |\langle \nabla b^i_j, 
  \nabla b^i_l\rangle - \delta_{jl}| \dV 
  \leq 10^{-n^2}$ 
 for all $r \leq r_i$.
 
 Then every sublimit of $(B_{r_i}(q_i),q_i)$ is isometric 
 to a product $(\rr^k \times X, q_{\infty})$ for some metric space $X$ 
 and some point $q_{\infty} \in \rr^k \times X$.
\end{lemma}

\begin{proof}
 Let $(Y, y)$ be an arbitrary sublimit of $(B_{r_i}(q_i),q_i)$. 
 Without loss of generality, 
 assume convergence $(B_{r_i}(q_i),q_i) \to (Y,y)$. 
 The concept of the proof is the following: 
 For well chosen $\hat{q}_i \in B_{1/2}(q_i)$,
 $c^i_{jl} := \langle \nabla b^i_j, \nabla b^i_l\rangle (\hat{q}_i)$
 and after passing to a subsequence, 
 $\avgintsmall_{B_1(q_i)} 
 |\langle \nabla b^i_j, \nabla b^i_l\rangle - c^i_{jl}| 
 \dV \to 0$. 
 In a second step, 
 the corresponding statement for balls of arbitrary radius will be shown.
 Finally, after passing to a subsequence 
 such that every $(c^i_{jl})_{i \in \nn}$ converges to some limit $c_{jl}$
 and defining $h_{jl}$ via the matrix identity 
 $\big((h_{jl})_{1 \leq j,l \leq k}\big)^2 
 = \big((c_{jl})_{1 \leq j,l \leq k}\big)^{-1}$,
 the linear combinations $d^i_j := \sum_{l=1}^k h_{jl} b^i_l$ 
 satisfy the hypothesis of \refProdLemma, 
 and thus, prove the claim.
 
 \par\smallskip\noindent\textit{First step:} 
 Fix $1 \leq j,l \leq k$ 
 and 
 suppose there exists $\eps >0$ such that for every $N \in \nn$ 
 there is $i \geq N$ with
 \[
  \avgint_{B_1(q_i)} \int_0^1 (\norm{\Hess(b^i_j)} 
  + \norm{\Hess(b^i_l)})(\gamma_{\hat{q}_i x_i}(t)) \dt \dV(x_i) 
  \geq \eps
 \]
 for all $\hat{q}_i \in B_{1/2}(q_i)$
 where $\gamma_{\hat{q}_i x_i}$ denotes a minimising geodesic from $\hat{q}_i$ to $x_i$.
 For such an $i$, 
 \begin{align*}
  &\avgint_{B_1(q_i) \times B_1(q_i)} \int_0^1 (\norm{\Hess(b^i_j)} 
  + \norm{\Hess(b^i_l)})(\gamma_{\hat{q}_i x_i}(t)) \dt \dV(\hat{q}_i,x_i)\\
  &\geq \frac{\avgintsmall_{B_{1/2}(q_i) \times B_1(q_i)} \int_0^1 
  (\norm{\Hess(b^i_j)} + \norm{\Hess(b^i_l)}) (\gamma_{\hat{q}_i x_i}(t)) 
  \dt \dV(\hat{q}_i,x_i)}{\CBG(n,-1,\frac{1}{2},1)}
  \displaybreak[0]\\
  &\geq \frac{\eps}{\CBG(n,-1,\frac{1}{2},1)}
  >0.
 \end{align*}
 On the other hand,
 \refSegmIneq
 provides $\tau = \tau(n,1)$ such that
 \begin{align*}
  &\avgint_{B_1(q_i) \times B_1(q_i)} \int_0^1 (\norm{\Hess(b^i_j)} 
  + \norm{\Hess(b^i_l)}) (\gamma_{\hat{q}_i x_i}(t)) \dt \dV(\hat{q}_i,x_i)\\
  &\leq \tau \cdot \avgint_{B_2(q_i)} (\norm{\Hess(b^i_j)} + \norm{\Hess(b^i_l)}) \dV 
  \displaybreak[0]\\
  &\leq 2\tau \cdot 
    \Big(\, \avgint_{B_2(q_i)} \sum_{j=1}^k \norm{\Hess(b^i_j)}^2 \dV \Big)^{1/2}\\
  &\to 0 \as i \to \infty.  
  \end{align*}
 This is a contradiction.
 Therefore, after passing to a subsequence,
 there exist $\hat{q}_i \in B_{1/2}(q_i)$ with
 $\avgintsmall_{B_1(q_i)} \int_0^1 
  (\norm{\Hess(b^i_j)} + \norm{\Hess(b^i_l)}) (\gamma_{\hat{q}_iq}(t)) 
  \dt \dV(q) \to 0.$
 Define 
 $c^i_{jl} := \langle \nabla b^i_j, \nabla b^i_l\rangle (\hat{q}_i)$.
 Then 
 \begin{align*}
  &{\avgint_{B_1(q_i)} 
  |\langle \nabla b^i_j, \nabla b^i_l \rangle(x_i) - c^i _{jl} | \dV(x_i)} \\
    &\leq \avgint_{B_1(q_i)} \int_0^1 \Big| \frac{d}{dt}_{|t=\tau} 
  \langle \nabla b^i_j (\gamma_{\hat{q}_iq} (t)), 
  \nabla b^i_l (\gamma_{\hat{q}_iq}(t))\rangle 
  \Big| \dtau \dV(q)
  \\&= \avgint_{B_1(q_i)} \int_0^1 \Big| 
  \langle \Hess(b^i_j) \cdot \dot{\gamma}_{\hat{q}_i x_i} (\tau), \nabla b^i_l \rangle 
  (\gamma_{\hat{q}_i x_i}(\tau))\\
  & \quad\quad\quad\quad\quad\quad 
  + \langle \nabla b^i_j , \Hess(b^i_l) \cdot (\dot{\gamma}_{\hat{q}_i x_i}(\tau)) \rangle 
  (\gamma_{\hat{q}_i x_i} (\tau)) 
  \Big| \dtau \dV(x_i)
  \displaybreak[0]\\
  &\leq \avgint_{B_1(q_i)} \frac{3L}{2} \cdot \int_0^1 (\norm{\Hess(b^i_j)} 
  + \norm{\Hess(b^i_l)}) \circ \gamma_{\hat{q}_i x_i} (\tau) \dtau \dV(x_i)\\
  &\to 0 \as {i \to \infty}.
 \end{align*}

 \par\smallskip\noindent\textit{Second step:}
 Fix an arbitrary $R > 0$.
 Analogously to the first step,
 one can prove the existence of $\bar{q}_i \in B_1(q_i)$ 
 such that
 \begin{align*}
  &|\langle \nabla b^i_j, \nabla b^i_l\rangle(\bar{q}_i) - c^i_{jl}| 
  \to 0 \quad\aand \\ 
  &\avgint_{B_R(q_i)} \int_0^1  (\norm{\Hess(b^i_j)} 
  + \norm{\Hess(b^i_l)}) (\gamma_{\bar{q}_i x_i}(t))\dt \dV(x_i) 
  \to 0 
 \end{align*}
 as $i \to \infty$.
 As in the first step,
 \begin{align*}
  &\avgint_{B_R(q_i)}|\langle \nabla b^i_j, \nabla b^i_l\rangle (q) 
  - c^i_{jl}| \dV(q)\\
  &\leq \avgint_{B_R(q_i)} \int_0^1 \Big|\frac{d}{dt}_{|t=\tau} 
  \langle \nabla b^i_j, \nabla b^i_l\rangle \circ \gamma_{\bar{q_i}q}(t) \Big| 
    + \Big|\langle \nabla b^i_j, \nabla b^i_l\rangle(\bar{q_i}) 
    - c^i_{jl} \Big| \dtau \dV(q) 
  \\&\to 0 \as i \to \infty.
 \end{align*}

 \par\smallskip\noindent\textit{Third step:}
 As the $b^i_j$ are $L$-Lipschitz, 
 $c^i_{jl} 
 := \langle \nabla b^i_j, \nabla b^i_l \rangle (\hat{q}_i) 
 \in [-L^2,L^2]$ 
 is a bounded sequence, and thus, has a convergent subsequence.
 Pass to a subsequence 
 such that all sequences $(c^i_{jl})_{i \in \nn}$ converge
 and denote the limits by 
 $c_{jl} := \lim_{i \to \infty} c^i_{jl} \in [-L^2,L^2]$.
 Then
 \begin{align*}
  |c_{jl} - \delta_{jl}|
  &\leq \lim_{i \to \infty}\avgint_{B_R(q_i)} 
  |c^i_{jl} - \langle \nabla b^i_j, \nabla b^i_l \rangle| \dV 
  + \avgint_{B_R(q_i)} |\langle \nabla b^i_j, \nabla b^i_l \rangle 
  - \delta_{jl}| \dV 
  \\&\leq \lim_{i \to \infty}
  \avgint_{B_R(q_i)} \sum_{j,l=1}^k 
  |\langle \nabla b^i_j, \nabla b^i_l \rangle - \delta_{jl}| \dV 
  \leq 10^{-n^2}.
 \end{align*}
 Hence, the matrix $C := (c_{jl})_{1 \leq j,l \leq k}$ 
 is invertible, symmetric and positive definite. 
 In particular, its inverse $C^{-1}$ is diagonalisable 
 with positive eigenvalues. 
 Let $C_D^{-1}$ denote the diagonal matrix and 
 $S$ the invertible matrix with $C^{-1} = S \cdot C_D^{-1} \cdot S^{-1}$
 and define $C_D^{-1/2}$ as the diagonal matrix 
 whose entries are the square roots of the diagonal entries of $C_D^{-1}$. 
 Then the matrix $H := (h_{jl})_{jl} := S \cdot C_D^{-1/2} \cdot S^{-1}$ 
 satisfies $H^2 = C^{-1}$.
 Now define $d^i_j := \sum_{l=1}^k h_{jl} b^i_l$. 

 Obviously, these are Lipschitz and harmonic. 
 Furthermore, 
 it is straightforward to see
 \begin{align*}
  {\avgint_{B_R(q_i)} 
  \sum_{j_1,j_2 = 1}^k 
  \big|\langle \nabla d^i_{j_1}, \nabla d^i_{j_2} \rangle - \delta_{j_1j_2} \big| 
  + \sum_{j=1}^k \norm{\Hess(d^i_j)}^2 
  \dV}
  &\to 0 \as i \to \infty.
 \end{align*}
 By \refProdLemma, 
 after passing to a subsequence, 
 $(B_{r_i}(q_i),q_i)$ converges to $(\rr^k \times X, (0,q_{\infty}))$ 
 for some metric space $X$ and $q_{\infty} \in X$.
 Since $(B_{r_i}(q_i),q_i)$ converges to $(Y,y)$, 
 this proves that $Y$ is isometric to $\rr^k \times X$.
\end{proof}

Applying the previous two lemmata proves 
that for sufficiently small balls 
there is a subset of good base points of arbitrary good volume
such that all sublimits of sequences with respect to those base points 
split off an $\rr^k$-factor.
In order to verify this, the following statement, 
which in its first form was proven by Stein in \cite[p.~13]{stein},
is needed for estimating the volume of a set 
where the so called $\rho$-maximum function is bounded from above.
Again, the notation of \cite[Lemma~1.4b)]{kapovitch-wilking} is used.

\begin{thm}[Weak type 1-1 inequality]\label{weak_1_1_inequality}
 Let $M$ be an \ndim Riemannian manifold 
 with lower Ricci curvature bound $\Ric_M \geq -(n-1)$.
 For a non-negative function $f: M \to \rr$ and $\rho > 0$,
 define the \emph{$\rho$-maximum function of $f$} as 
 \[\Mx_{\rho}f(p) := \sup_{r \leq \rho} \avgint_{B_r(p)} f.\]
 Especially, put $\Mx f(p) = \Mx_2f(p)$. 
 Then there is $\C11 > 0$ 
 such that for any non-negative function $f \in L^1(M)$ and $c > 0$, 
 \[
  \vol(\{x \in M \mid \Mx_{\rho}f(x) > c\}) 
  \leq \frac{\C11}{c} \int_M f \dV_M.
 \]
\end{thm}
In the following, $\C11$ will always denote the constant of the weak type 1-1 inequality.
\par\medskip

Now the first set needed for \autoref{prop:main} can be constructed.
\begin{lemma}\label{lem:set_where_all_blow_ups_split__locally}
 Let $(M_i)_{i \in \nn}$ be a sequence 
 of complete connected \ndim Riemannian manifolds
 which satisfies the uniform lower 
 Ricci curvature bound $\Ric_{M_i} \geq -(n-1)$, 
 and let $k < n$. 
 For every $\mydelta \in (0,1)$ 
 there exists $\tildeeps_1 = \tildeeps_1(\mydelta;n,k) > 0$ 
 such that for all $0 < r \leq \tildeeps_1$ and $q_i \in M_i$ with
 \[\dghpt{\tildeeps_1}{r^{-1} M_i}{q_i}{\rr^k}{0}\]
 there is a family of subsets of good points 
 $G_{r}^1(q_i) \subseteq B_{r}(q_i)$ satisfying
 \[ \vol_{M_i}(G_{r}^1(q_i)) \geq (1-\mydelta) \cdot \vol_{M_i}(B_{r}(q_i)) \]
 such that for every choice of $x_i \in G_r^1(q_i)$,
 every $\lambda_i \to \infty$ 
 and every sublimit $(Y,\cdot)$ of $(\lambda_i M_i, q_i)$
 there exists $Y'$ such that $Y \cong \rr^k \times Y'$ isometrically.
\end{lemma}
 
\begin{proof}
 Define $C(n) := \C11 \cdot\ 10^{n^2} \cdot \CBG(n,-1,1,2)$, 
 and let $\tildeeps_1 = \tildeeps_1(\mydelta;n,k)$ 
 be the constant $\tildeeps(\frac{\mydelta}{C(n)},\id,2;n,k)$ 
 and $L(n)$ be as in \autoref{prop-M}.
 Further, let $0 < r \leq \tildeeps_1$ and $q_i \in M_i$ satisfy
 \[\dghpt{\tildeeps_1}{r^{-1} M_i}{q_i}{\rr^k}{0}.\]
 Then there are harmonic and $L$-Lipschitz functions 
 $f^i_j : B_{2}^{r^{-1} M_i}(q_i) \to \rr$, $1 \leq j \leq k$, satisfying
 \[
  \avgint_{B^{r^{-1} M_i}_2(q_i)} \psi_\nabla(f^i) 
  + \psi_H(f^i) \dV_{r^{-1} M_i} 
  \leq \frac{\mydelta}{C(n)}
 \]
 where
 $\psi_\nabla(f^i) 
 := \sum_{j,l=1}^k |\langle \nabla f^i_j, \nabla f^i_l \rangle - \delta_{jl}|$ 
 and 
 $\psi_H(f^i) 
 := \sum_{j=1}^k \norm{\Hess(f^i_j)}^2$.
 Define
 \[
  G_i 
  := \{x_i \in B_{1}^{r^{-1} M_i}(q_i) 
  \mid \Mx_{1}^{r^{-1} M_i} (\psi_\nabla(f^i) + \psi_H(f^i)) (x_i) 
  < 10^{-n^2}\}
 \]
 where the $1$-maximum function is taken 
 with respect to $d_{r^ {-1} M_i} = \frac{1}{r} d_{M_i}$.
 Using \autoref{weak_1_1_inequality},
 the volume of this set can be estimated by
 \begin{align*}
  &\vol_{r^{-1} M_i}(B_{1}^{r^{-1} M_i}(q_i) \setminus G_i) \\
  &\leq \frac{C(n)}{\CBG(n,-1,1,2)} \cdot \frac{\mydelta}{C(n)} 
  \cdot \vol_{r^{-1} M_i}({B_{2}^{r^{-1} M_i}(q_i)}) \\
  &\leq \mydelta \cdot \vol_{r^{-1} M_i}({B_{1}^{r^{-1} M_i}(q_i)}).
 \end{align*}
 Hence, regarding $G_{r}^1(q_i) := G_i$ as a subset of $M_i$,
 \begin{align*}
  \frac{\vol_{M_i}(G_{r}^1(q_i))}{\vol_{M_i}(B_{r}^{M_i}(q_i))}
  = \frac{\vol_{r^{-1} M_i}(G_i)}{\vol_{r^{-1} M_i}(B_{1}^{r^{-1} M_i}(q_i))}
  \geq 1-\mydelta.
 \end{align*}
 Now let $x_i \in G_r^1(q_i)$ and $\lambda_i \to \infty$ be arbitrary.
 Define $r_i := \lambda_i \cdot r \to \infty$ and let $0 < \rho \leq r_i$.
 Since 
 $ B_{r_i}^{\lambda_i M_i}(x_i) 
 = B_1^{r^{-1} M_i}(x_i) 
 \subseteq B_2^{r^{-1} M_i}(q_i)$,
 the rescaled maps 
 \[\tilde{f}^i_j := r_i \cdot f^i_j : B_{r_i}^{\lambda_i M_i}(x_i) \to \rr\]
 are well defined, harmonic and $L$-Lipschitz. 
 It is straightforward to see
 \begin{align*}
  &\avgint_{B_{\rho}^{\lambda_i M_i}(x_i)} \psi_\nabla(\tilde{f}^i) \dV_{\lambda_i M_i}
  \leq 10^{-n^2}
  \quad\aand\\
  &\avgint_{B_{\rho}^{\lambda_i M_i}(x_i)} \psi_H(\tilde{f}^i) \dV_{\lambda_i M_i} 
  \leq \frac{10^{-n^2}}{r_i^2} \to 0 \as i \to \infty.    
 \end{align*}
 By \autoref{gen_product_lemma}, any sublimit of $(\lambda_i M_i, x_i)$ 
 has the form $(\rr^k \times Y', \cdot)$ for some metric space $Y'$.
\end{proof}

\subsection{Construction of 
\texorpdfstring{$G_r^2(q_i)$}{second set} 
and \texorpdfstring{$\lambda_i$}{scales}}
\label{sec:second_set}
The aim of this subsection is to find 
a rescaling sequence $\lambda_i \to \infty$ 
and a family of subsets $G_r^2(q_i) \subseteq B_r(q_i)$
with the following two properties: 
On the one hand, every single rescaled manifold $\lambda_i M_i$ 
(with a base point from $G_r^2(q_i)$) 
is close to a product of $\rr^k$ and a compact metric space.
On the other hand, the sublimits of sequences $(\lambda_i M_i,x_i)$ 
with base points $x_i \in G_r^2(q_i)$ have the same dimension 
(depending not on the base points but only on the choice of the subsequence).

Before motivating the procedure, 
recall the definition of time-dependent vector fields, cf.~\cite{Lee}: 
A time-dependent vector field on a Riemannian manifold $M$ 
is a continuous map $X: I \times M \to TM$, where $I \subseteq \rr$ is an interval,
such that $X^t$ is a vector field for all $t \in I$,
i.e.~$X^t$ satisfies $X^t_p := X^t(p) := X(t,p) \in T_pM$ for all $(t,p) \in I \times M$.
Such a time-dependent vector field $X: I \times M \to TM$ 
is called \emph{piecewise constant in time} 
if there exist disjoint sub-intervals $I = I_1 \amalg \ldots \amalg I_n$ 
such that $X^s = X^t$ for all $1 \leq i \leq n$ and $s,t \in I_i$. 
 
For arbitrary $s \in I$ and $I-s := \{\tau-s \mid \tau \in I\}$, 
a curve $c : I-s \to M$ is called \emph{$s$-integral curve of $X$} if
$c'(t) = X^{s+t}_{c(t)}$ 
for all $t \in I-s$.
A $0$-integral curve is also called \emph{integral curve of $X$}.

Furthermore,
there exist an open set 
$\Omega \subseteq \bigcup_{s\in I} \{s\} \times (I-s) \times M$ 
and a map $\Phi : \Omega \to M$
satisfying the following: 
for any $(s,p) \in I \times M$, the set 
$\Omega^{(s,p)} := \{t \in I-s \mid (s,t,p) \in \Omega\}$ 
is an open interval which contains $0$,
and for any fixed $(s,p) \in I \times M$ 
the map $c: \Omega^{(s,p)} \to M$ defined by $c(t) := \Phi(s,t,p)$ 
is the unique maximal $s$-integral curve of $X$ with starting point $p$.
Using the notation $\phi^s_t := \Phi(s,t,\cdot)$, 
this is equivalent to $\phi$ being a maximal solution of
$\frac{d}{dt}_{|t=t_0} \phi^s_t(p) = X^{s+t_0}_{\phi^s_{t_0}(p)}$ 
and
$\phi^s_0 = \id$.
Such a $\Phi$ is called \emph{flow of $X$}.
Moreover, the following is true:
If $p \in M$ and $s,t,u$ are times 
with $(s,t,p) \in \Omega$ and $(s+t, u, \phi^s_t(p)) \in \Omega$,
then $(s,t+u,p) \in \Omega$ and 
$\phi^{s+t}_u \circ \phi^s_t (p) = \phi^s_{t+u} (p)$.
In particular, if defined, $\phi^{s+t}_{-t}$ is the inverse of $\phi^s_t$.

A time-dependent vector field $X$ has \emph{compact support} 
if there exists a compact set $K \subseteq M$ 
such that for all $t \in I$ 
the vector fields $X^t$ have support $K$.
In this case, the flow $\Phi$ exists for all times.

\par\medskip
In order to prove that two blow-ups have the same dimension, 
the following will be established and used:
Let $X_i : [0,1] \times M_i \to TM_i$ be time-dependent vector fields with 
$\int_0^1 (\Mx_{2r}(\norm{\nabla.X_i^t}^{3/2}) (c_i(t)) )^{2/3} \dt 
< \mydelta$ 
for all $i \in \nn$
where $\mydelta > 0$ and the $c_i$ are integral curves. 
Moreover, let the $X_i$ be divergence free, 
i.e.~the flows are measure preserving.
Then any blow-ups coming from the sequences 
with base points $c_i(0)$ and $c_i(1)$, respectively,
have the same dimension,
i.e.~if $\lambda_i \to \infty$ with
\[ 
 (\lambda_i M_i, c_i(0)) \to (Y_0,y_0) 
 \quad\aand\quad 
 (\lambda_i M_i, c_i(1)) \to (Y_1,y_1),
\]
then $\dim(Y_0) = \dim(Y_1)$.
This will be proven in \autoref{sec:C=>dim}.

Since \GH convergence is preserved by shifting base points a little bit,
the same statement is true 
if the base points $c_i(0)$ and $c_i(1)$, respectively,
are replaced by points $x_i$ and $y_i$, respectively,
where $\lambda_i \cdot d(c_i(0),x_i) < C$ 
and $\lambda_i \cdot d(c_i(1),y_i) < C$ 
for some $C > 0$ (independent of $i$).
This motivates the following definition.

\begin{defn}\label{dfn:close_to_zooming_in}
 Let $M$ be a complete connected \ndim Riemannian manifold 
 and $r, C, \mydelta>0$.
 A subset $B_r(q)' \subseteq B_r(q)$ has the 
 \emph{$\mathcal{C}(M,r,C,\mydelta)$-property} if the following holds:
 
 For all pairs of points $x,y \in B_r(q)'$
 there exists a time-dependent vector field $X : [0,1] \times M \to TM$ 
 which is piecewise constant in time and has compact support 
 and an integral curve $c:[0,1] \to M$ satisfying
 the following conditions:
 \begin{enumerate}
  \item 
   The vector field $X^t$ is divergence free on $B_{10r}(c(t))$ 
   for all $0 \leq t \leq 1$,
  \item 
   $d(x,c(0)) < C$, $d(y,c(1)) < C$ and
  \item 
   $\int_0^1 (\Mx_{2r}(\norm{\nabla.X^t}^{3/2}) (c(t)) )^{2/3} \dt 
   < \mydelta$.
 \end{enumerate}
\end{defn}
Note the following:
If a subset $B_r(q)' \subseteq B_r(q)$ has the $\mathcal{C}(M,r,C,\mydelta)$-property
and is regarded as a subset $B_r(q)' \subseteq B_{\lambda r}^{\lambda M}(q)$, where $\lambda > 0$,
then it has the $\mathcal{C}(\lambda M,\lambda r, \lambda C,\mydelta)$-property.
For this reason, the notation of the $\mathcal{C}$-property contains the manifold $M$.
\par\medskip

In order to construct the subset $G_r^2(q_i)$, 
the following statement is used:
There is a rescaling factor such that, 
if a manifold is sufficiently close to $\rr^k$, 
the rescaled manifold is close to a product. 
This statement will be proven by contradiction 
using the following theorem of Kapovitch and Wilking 
where the first part is the first part of the original theorem 
and the second is taken from its proof.

\begin{rescthm}
[Rescaling Theorem {\cite[Theorem 5.1]{kapovitch-wilking}}]
\label{rescaling-theorem}
 Let $(M_i,p_i)_{i \in \nn}$ be a sequence 
 of pointed \ndim Riemannian manifolds
 and $r_i \to \infty$ and $\mu_i \to 0$ 
 be sequences of positive real numbers 
 such that $B_{r_i}^{M_i}(p_i)$ has Ricci curvature larger than $- \mu_i$ 
 and $\B_{r_i}^{M_i}(p_i)$ is compact.
 Suppose that $(M_i,p_i)$ converges to $(\rr^k,0)$ for some $k < n$. 
 After passing to a subsequence, 
 there is a compact metric space $K$ with $\diam(K) = 10^{-n^2}$, 
 a family of subsets $G_1(p_i) \subseteq B_1(p_i)$ 
 with $\frac{\vol(G_1(p_i))}{\vol(B_1(p_i))} \to 1$,
 a sequence $\lambda_i \to \infty$ 
 and a sequence $\mydelta_i \to 0$ 
 such that the following holds:
 \begin{enumerate}
  \item 
   For all $q_i \in G_1(p_i)$, 
   the isometry type of the limit of any convergent subsequence 
   of $(\lambda_i M_i, q_i)$ 
   is given by the metric product $\rr^k \times K$.
  \item 
   The set $G_1(p_i)$ has the $\mathcal{C}(M_i,1,\frac{9^n}{\lambda_i},\mydelta_i)$-property. 
 \end{enumerate}
\end{rescthm}

\begin{lemma}\label{cor_rescaling_theorem}
 For every $\mydelta \in (0,1)$, $R>0$, $\eta > 0$ and $k < n$
 there exists a bound $\tildeeps = \tildeeps(\mydelta,R,\eta;n,k) > 0$ 
 such that for all pointed \ndim Riemannian manifolds $(M,p)$ 
 with $\Ric_M \geq -(n-1) \cdot \tildeeps^2$
 satisfying that $\B_{1/\tildeeps}(p)$ is compact 
 and \[\dghpt{\tildeeps}{M}{p}{\rr^k}{0}\]
 there is a factor $\lambda > 0$ 
 such that the following holds: 
 \begin{enumerate}
 \item\label{cor_rescaling_theorem--a} 
  There are 
  a subset of good points $G_1(p) \subseteq B_1(p)$ satisfying 
  \[\vol_M(G_1(p)) \geq (1- \mydelta) \cdot \vol_M(B_1(p))\] and 
  a compact metric space $K$ of diameter $1$
  such that for all $q \in G_1(p)$ 
  there is a point $\tilde{q} \in \rr^k \times K$ with 
  \[\dgh{R}{\lambda M}{q}{\rr^k \times K}{\tilde{q}} \leq \eta.\]
 \item\label{cor_rescaling_theorem--b} 
  The set $G_1(p)$ has the 
  $\mathcal{C}(M,1,\frac{9^n\cdot 10^{n^2}}{\lambda},\mydelta)$-property.
 \end{enumerate}  
\end{lemma}

\begin{proof}
 This is a straightforward contradiction argument 
 rescaling both the sequence $\lambda_i \to \infty$ and the compact space $K$
 occurring in \autoref{rescaling-theorem} by a factor $10^{n^2}$.
\end{proof}

Now rescaling the sequence $M_i$ 
such that each element is close enough to $\rr^k$ 
and applying the previous result,
one obtains factors $\lambda_i$ 
which basically are the sought-after rescaling sequence.
However, the lemma does provide $\lambda_i$ for every $i$, 
but does not give any information about whether or not 
$\lambda_i \to \infty \as i \to \infty$. 
In order to prove $\lambda_i \to \infty$, 
the fact is needed that spaces of different dimensions are not close. 
This in turn follows from the fact 
that sequences of limit spaces do not increase dimension.
For this, the following lemma is needed which states that, 
given a converging sequence of proper length spaces, 
there exists a rescaling sequence
such that the rescaled sequence converges to a tangent cone.

\begin{lemma}\label{lem:rescaling_converging_to_tangent_cone}
 Let $(X_i,p_i) \to (X,p)$ be a converging sequence of proper length spaces. 
 Then there exists $\mu_i \to \infty$ 
 such that for all $\lambda_i \to \infty$ with $\lambda_i \leq \mu_i$,
 $(\lambda_i X_i,p_i)$ subconverges to a tangent cone of $(X,p)$.
\end{lemma}

\begin{proof}
 For $\eps_i \to 0$ such that $\dghpt{\eps_i}{X_i}{p_i}{X}{p}$,
 let $\mu_i := \eps_i^{-1/2}$. 
 For fixed $r > 0$, let $i$ be large enough such that $r \leq \eps_i^{-3/2}$.
 Then $\frac{r}{\mu_i} \leq \frac{1}{\eps_i}$ and 
 \begin{align*}
  \dgh{r/\mu_i}{X_i}{p_i}{X}{p}
  \leq 16 \cdot \eps_i
  \to 0 \as i \to \infty.
 \end{align*}
 After passing to a subsequence, 
 $(\mu_i X,p)$ converges to a tangent cone $(Y,q)$.
 Then 
 \begin{align*}
  &\dgh{r}{\mu_i X_i}{p_i}{Y}{q}
  \\&\leq \mu_i \cdot \dgh{r/\mu_i}{X_i}{p_i}{X}{p} + \dgh{r}{\mu_i X}{p}{Y}{q}
  \\&\leq 16 \cdot \eps_i \cdot \mu_i + \dgh{r}{\mu_i X}{p}{Y}{q}
  \\&\to 0,
 \end{align*} 
 and this proves that $(\mu_i X_i, p_i)$ subconverges to $(Y,q)$.

 Now let $\lambda_i \to \infty$ with $\lambda_i \leq \mu_i$. 
 After passing to a further subsequence, 
 $\frac{\lambda_i}{\mu_i} \to \alpha$
 for some $\alpha \leq 1$, 
 $(\lambda_i X,p) \to (\alpha Y, q)$ 
 and $(\lambda_i M_i, p_i) \to (\alpha Y,q)$.
 In particular, $(\lambda_i X_i, p_i)$ subconverges 
 to a tangent cone of $(X,p)$.
\end{proof}

Let $\limitspaces$ denote the class of all pointed metric spaces 
that can occur as \GH limit of a sequence 
of pointed complete connected \ndim Riemannian manifolds $M_i$ 
with $\Ric_{M_i} \geq -(n-1)$.

\begin{lemma}\label{lem:limit_of_limits_decreases_dimension}
 Let $(X_i,p_i) \to (X,p)$ be converging spaces in $\limitspaces$
 such that all $X_i$ have the same dimension $\dim(X_i) = k$. 
 Then $\dim(X) \leq k$.
\end{lemma}

\begin{proof}
 In order to estimate $l := \dim(X)$, 
 take a generic point $x \in \Xgen$ 
 and construct a tangent cone $\rr^l$.
 The idea of this construction is 
 to consider sequences of manifolds $M_{ij}$ converging to $X_i$.
 For large $i$, these are sufficiently close to $X$, 
 and applying \autoref{prop-M} and \autoref{gen_product_lemma} 
 will give the claim.
 So, let $(M_{ij},p_{ij}) \to (X_i,p_i)$ as $j \to \infty$. 
 Without loss of generality, 
 assume $(X,p) = (\rr^l,0)$ and $\Ric_{M_{ij}} \geq -(n-1) \cdot \delta_i$ 
 for some monotonically decreasing sequence $\delta_i \to 0$.
 
 Choose $\eps_i \to 0$ 
 such that
 $\dgh{1/\eps_i}{X_i}{p_i}{\rr^l}{0} \leq \frac{\eps_i}{2}$.
 Without loss of generality, 
 $\dgh{1/\eps_i}{M_{ij}}{p_{ij}}{X_i}{p_i} \leq \frac{\eps_i}{2}$ 
 for all $j \in \nn$.
 Hence, 
 \[\dgh{1/\eps_i}{M_{ij}}{p_{ij}}{\rr^l}{0} \leq \eps_i.\]
 
 Define $g : \rr^+ \to \rr^+$ by 
 \[g(x) := \begin{cases}
            \eps_i 	& \text{if } \delta_i \leq x < \delta_{i-1},\\
            1		& \text{if } x \geq \delta_1,
           \end{cases}
 \]
 $c=c(n) := 2 \cdot \C11 \cdot \CBG(n,-1,\frac{1}{2},1)$
 and choose $\tildeeps = \tildeeps(\frac{10^{-n^2}}{c},g,1;n,l)$ 
 as in \autoref{prop-M}. 
 Let $i \in \nn$ be sufficiently large
 such that $\delta_i \leq \tildeeps$.
 Then $\Ric_{M_{ij}} \geq -(n-1) \cdot \delta_i$ and, 
 since $g(\delta_i) = \eps_i$,
 \[\dgh{1/g(\delta_i)}{M_{ij}}{p_{ij}}{\rr^l}{0} \leq g(\delta_i).\]
 So, there are $L=L(n)$ and harmonic $L$-Lipschitz maps 
 $f^{ij}_h:B_1^{M_{ij}}(p_{ij}) \to \rr$, $1 \leq h \leq l$,
 such that 
 \[
  \avgint_{B_1^{M_{ij}}(p_{ij})} \sum_{h_1,h_2=1}^l 
  |\langle \nabla f^{ij}_{h_1}, \nabla f^{ij}_{h_2} \rangle - \delta_{h_1h_2}| 
  + \sum_{h=1}^l \norm{\Hess(f^{ij}_h)}^2 \dV_{M_{ij}} < \frac{10^{-n^2}}{c}.
 \]
 In order to simplify notation, let 
 \[
  F^{ij} := \sum_{h_1,h_2=1}^l 
  |\langle \nabla f^{ij}_{h_1}, \nabla f^{ij}_{h_2} \rangle - \delta_{h_1h_2}| 
  + \sum_{h=1}^l \norm{\Hess(f^{ij}_h)}^2.
 \]
 With the usual argumentation,
 the set 
 \[
  G_{ij} := \{p\in \B_{1/2}^{M_{ij}}(p_{ij}) 
  \mid \Mx_{1/2}(F^{ij})(p) \leq 10^{-n^2} \}
 \] 
 is compact with 
 $\vol(G_{ij}) \geq \frac{1}{2} \cdot \vol(B_{1/2}^{M_{ij}}(p_{ij}))$
 and the sequence $(G_{ij})_{j \in \nn}$ subconverges 
 to a set $G_i \subseteq \B^{X_i}_{1/2}(p_i)$ with positive volume. 
 In particular, the intersection with $(X_i)_\textrm{gen}$ is nonempty.
 Without loss of generality, 
 assume that $(G_{ij})_{j \in \nn}$ itself already converges to $G_i$,
 and choose $q_{ij} \in G_{ij}$ converging to a $q_i \in (X_i)_\textrm{gen}$.
 
 Since $(M_{ij},q_{ij}) \to (X_i,q_i)$, 
 there exists $\mu^i_j \to \infty$
 as in \autoref{lem:rescaling_converging_to_tangent_cone}
 such that 
 $(\mu^i_j M_{ij},q_{ij}) \to (\rr^k,0)$ as $j \to \infty$.
 On the other hand, the rescaled maps
 \[
  \tilde{f}^{ij} := \mu^i_j f^{ij} : 
  B_{\mu^i_j}^{\mu^i_j M_{ij}}(q_{ij}) \to \rr
 \] 
 are harmonic and $L$-Lipschitz.
 Furthermore, for arbitrary $r>0$ and $j$ large enough 
 such that $2r < \mu^i_j$,
 \begin{align*}
  &\avgint_{B_{r}^{\mu^i_j M_{ij}}(q_{ij})} \sum_{h_1,h_2=1}^l 
  |\langle \nabla \tilde{f}^{ij}_{h_1}, \nabla \tilde{f}^{ij}_{h_2} \rangle 
  - \delta_{h_1h_2}| \dV_{\mu^i_j M_{ij}} 
  \leq 10^{-n^2}
  \quad\aand\\
  &\avgint_{B_{r}^{\mu^i_j M_{ij}}(q_{ij})} \sum_{h=1}^l 
  \norm{\Hess(\tilde{f}^{ij}_h)}^2 \dV_{\mu^i_j M_{ij}} 
  \to 0  \as j \to \infty.
 \end{align*}
 By \autoref{gen_product_lemma}, there exists a metric space $Z$ 
 and a point $z \in \rr^l \times Z$ such that 
 \[(\mu^i_j M_{ij},q_{ij}) \to (\rr^l \times Z, z) \as j \to \infty.\]
 In particular, $\rr^k \cong \rr^l \times Z$, and thus, $k \geq l$.
\end{proof}

\begin{lemma}\label{lem:spaces_of_different_dimension_are_not_close}
 For all $k < n$ there is $\eps_0 = \eps_0(n,k) \in (0,\frac{1}{100})$ 
 such that the following is true:
 If $(X,p), (\rr^k \times K,q) \in \limitspaces$ 
 for a compact metric space $K$ 
 with $\diam(K) = 1$ and $\dim(X) = k$,
 then 
 \[\dgh{1/\eps_0}{X}{p}{\rr^k \times K}{q} > \eps_0.\]
\end{lemma}

\begin{proof}
 Assume the statement is false and let $k < n$ be contradicting.
 For every $i \in \nn$ with $i > 100$
 choose sequences $(M_{ij},p_{ij}) \to (X_i, p_i)\in \limitspaces$ 
 and $(N_{ij},q_{ij}) \to (\rr^k \times K_i, q_i)\in \limitspaces$ as $j \to \infty$ 
 with $\diam(K_i)=1$, $\dim(X_i) = k$ and
 $\dghpt{\frac{1}{i}}{X_i}{p_i}{\rr^k \times K_i}{q_i}$.
 
 For every $i \in \nn$ there is $J(i) \in \nn$ such that 
 $\dghpt{\frac{1}{i}}{M_{ij}}{p_{ij}}{X_i}{p_i}$ for all $j \geq J(i)$.
 Define inductively $j_1 := J(1)$ and $j_{i+1} := \max\{J(i+1), j_{i}+1\}$. 
 In particular, $j_i \to \infty$ as $i\to \infty$ and
 $\dghpt{\frac{1}{i}}{M_{ij_i}}{p_{ij_i}}{X_i}{p_i}$.
 
 After passing to a subsequence, 
 $(M_{ij_i},p_{ij_i})$ converges to some $(X,p) \in \limitspaces$ as $i \to \infty$.
 For arbitrary $r > 0$, 
 \begin{align*}
  &\dgh{r}{X_{i}}{p_{i}}{X}{p}\\
  &\leq \dgh{r}{X_i}{p_i}{M_{ij_i}}{p_{ij_i}} 
   + \dgh{r}{M_{ij_i}}{p_{ij_i}}{X}{p}
  \to 0 \as i \to \infty.
 \end{align*}
 Hence, $(X_i,p_i) \to (X,p)$.
 With analogous argumentation, 
 after passing to a further subsequence,
 $(\rr^k \times K_i,q_i) \to (\rr^k \times K,q)\in \limitspaces$ 
 for some compact metric space $K$ 
 with $\diam(K) = 1$.
 On the other hand, 
 for $r > 0$ and $i \geq r$,
 \begin{align*}
  &\dgh{r}{\rr^k \times K_i}{q_i}{X}{p} \\
  &\leq \dgh{r}{\rr^k \times K_i}{q_i}{X_i}{p_i} + \dgh{r}{X_i}{p_i}{X}{p}
  \to 0 \as i \to \infty.
 \end{align*}
 Hence, $(\rr^k \times K_i,q_i) \to (X,p)$.
 In particular, $X \cong \rr^k \times K$ and $\dim(X) > k$. 
 This is a contradiction to $\dim(X) \leq k$ by 
 \autoref{lem:limit_of_limits_decreases_dimension}. 
\end{proof}

Using this lemma, the sought-after rescaling sequence 
and family of sets can finally be constructed.
\begin{lemma}\label{nearly_splitting_after_rescaling_locally}
 Let $(M_i,p_i)_{i \in \nn}$ be a sequence
 of complete connected \ndim Riemannian manifolds
 which satisfy the uniform lower Ricci curvature bound 
 $\Ric_{M_i} \geq -(n-1)$, 
 and let $k < n$. 
 For every $\mydelta \in (0,1)$ 
 there exists $\tildeeps_2 = \tildeeps_2(\mydelta;n,k) > 0$ such that 
 for all $0 < r \leq \tildeeps_2$ and $q_i \in M_i$ with
 \[\dghpt{\tildeeps_2}{r^{-1} M_i}{q_i}{\rr^k}{0}\]
 there are
 a family of subsets of good points $G_{r}^2(q_i) \subseteq B_{r}(q_i)$ 
 with \[\vol(G_r^2(q_i)) \geq (1- \mydelta) \cdot \vol(B_r(q_i))\]
 and a sequence $\lambda_i \to \infty$
 which satisfy:
 \begin{enumerate}
 \item 
  For each $x_i \in G_r^2(q_i)$ there is a compact metric space $K_i$ 
  with diameter $1$ and a point $\tilde{x}_i \in \{0\} \times K_i$ such that
  \[
   \dgh{1/\eps_0}{\lambda_i M_i}{x_i}{\rr^k \times K_i}{\tilde{x}_i} 
   \leq \frac{\eps_0}{200}
  \]
  for $\eps_0 = \eps_0(n,k)$ 
  as in \autoref{lem:spaces_of_different_dimension_are_not_close}.
 \item 
  The sets $G_{r}^2(q_i)$ have the 
  $\mathcal{C}(M_i,r,\frac{9^n\cdot 10^{n^2}}{\lambda_i},\mydelta)$-property.
 \end{enumerate}
\end{lemma}

\begin{proof}
 For $\mydelta \in (0,1)$, 
 let $\tildeeps_2 = \tildeeps_2(\mydelta;n,k)$ denote 
 the $\tildeeps(\mydelta,\frac{1}{\eps_0},\frac{\eps_0}{200};n,k)$ 
 of \autoref{cor_rescaling_theorem}.
 Fix an arbitrary $0 < r \leq \tildeeps_2$ and points $q_i \in M_i$ satisfying
 $\dghpt{\tildeeps_2}{r^{-1} M_i}{q_i}{\rr^k}{0}$.
 
 For these $(r^{-1}M_i,q_i)$, 
 let $\tilde{\lambda}_i > 0$, $\tilde{G}_1(q_i) \subseteq B^{r^{-1} M_i}_1(q_i)$ and $K_i$
 be as in \autoref{cor_rescaling_theorem}, 
 define $\lambda_i := \frac{\tilde{\lambda_i}}{r}$
 and regard 
 $G_{r}^2(q_i) := \tilde{G}_1(q_i)$ 
 as a subset of $M_i$. 
 Then 
 \[
  \frac{\vol_{M_i}(G_{r}^2(q_i))}{\vol_{M_i}(B^{M_i}_{r}(q_i))} 
  = \frac{\vol_{r^{-1} M_i}(\tilde{G}_1(q_i))}
     {\vol_{r^{-1} M_i}(B^{r^{-1} M_i}_{1}(q_i))} 
  \geq 1-\mydelta,\]
 for every $x_i \in G_r^2(q_i)$ 
 there is $\tilde{x}_i \in \{0\} \times K_i$ with
 \[
  \dgh{1/\eps_0}
  {\lambda_i M_i}{x_i}
  {\rr^k \times K_i}{\tilde{x}_i} 
  \leq \frac{\eps_0}{200}
 \]
 and $G_{r}^2(q_i)$ has the 
 $\mathcal{C}(M_i,r,\frac{9^n \cdot 10^{n^2}}{\lambda_i},\hat{\eps})$-property.
 
 Assume that the sequence $(\lambda_i)_{i \in \nn}$ is bounded.
 After passing to a subsequence, 
 $\lambda_i \to \alpha$ and $(\lambda_i M_i,x_i) \to (\alpha X,q)$ 
 for some $q \in X$.
 Then
 \[\dgh{1/\eps_0}{\alpha X}{q}{\rr^k \times K_i}{\tilde{x}_i}  \leq \eps_0\]
 for all $i$ large enough
 in contradiction to 
 \autoref{lem:spaces_of_different_dimension_are_not_close}.
 Hence, $\lambda_i \to \infty$.
\end{proof}

This concludes the construction of $G_r^1(q_i)$, $G_r^2(q_i)$ and $\lambda_i$.

\subsection{The \texorpdfstring{$\mathcal{C}$}{C}-property 
and the dimension of blow-ups}
\label{sec:C=>dim}
In order to prove that the blow-ups 
with base points $x_i$ and $y_i$, respectively, have the same dimension,
a crucial argument is that the flow of a time-dependent vector field 
as in the definition of the $\mathcal{C}$-property
is bi-Lipschitz on some small set. 
This result and the implication about dimensions 
are proven in this subsection.

For the proof it is important to know 
under which conditions large subsets of two balls intersect.
The following lemmata deal with this question.

\begin{lemma}\label{lemma-criterion_intersection_subsets} 
 Let $(X,d,\vol)$ be a metric measure space and 
 let $A' \subseteq A \subseteq X$ and $B' \subseteq B \subseteq X$ 
 be measurable subsets with 
 \begin{align*}
  \vol(A') &\geq (1 - \mydelta) \cdot \vol(A), \\
  \vol(B') &\geq (1 - \mydelta) \cdot \vol(B), \\
  \vol(A \cap B) &> 2 \mydelta \cdot \max\{\vol(A),\vol(B)\}
 \end{align*}
 for some $\mydelta > 0$.
 Then $\vol(A' \cap B') > 0$, in particular, $A' \cap B' \ne \emptyset$.
\end{lemma}
\begin{proof}
 Observe 
 $\vol((A \cap B) \setminus A') 
 \leq \vol( A \setminus A') 
 \leq \mydelta \cdot \vol(A) 
 < \frac{1}{2} \cdot \vol(A \cap B)$
 by hypothesis. 
 Analogously, 
 $\vol((A \cap B) \setminus B') 
 < \frac{1}{2} \cdot \vol(A \cap B)$.
 Thus,
 \begin{align*}
  \vol(A' \cap B') 
  &= \vol(A \cap B) - \vol((A \cap B) \setminus (A' \cap B')) 
 \\&\geq \vol(A \cap B) - ( \vol((A \cap B) \setminus A') + \vol((A \cap B) \setminus B') )
  >0.
  \qedhere
 \end{align*}
\end{proof}

\begin{lemma}\label{lemma-criterion_volume_intersection}
Let $(M,g)$ be a complete connected \ndim Riemannian manifold 
with lower Ricci curvature bound $\Ric_M \geq -(n-1)$.
Given $0<\mydelta<\frac{1}{2}$ and $s>0$,
there exist $d_0(n,\mydelta,s) > 0, \delta_0(n,\mydelta,s) > 0$ 
such that $(\delta_0(n), \frac{1}{s d_0} - \frac{1}{2})$ is non-empty 
and for $\delta \in (\delta_0(n), \frac{1}{s d_0} - \frac{1}{2})$,
points $p,q \in M$ with distance $d := d(p,q) < d_0$ 
and $R : = \frac{d}{2} + \delta d < \frac{1}{s}$,
\[
 \vol(B_R(p) \cap B_R(q)) 
 > 2\mydelta \cdot \max \{ \vol(B_R(p)) , \vol(B_R(q)) \}.
\]
Moreover, this $\delta_0$ can be chosen 
to be monotonically increasing in $d_0$.
\end{lemma}

\begin{proof}
 Let $p,q \in M$ be arbitrary, 
 $\gamma : [0,d] \to M$ be a shortest geodesic connecting $p$ and $q$ 
 and define $m := \gamma(\frac{d}{2})$ as the midpoint of this geodesic, 
 i.e.~$d(p,m) = d(q,m) = \frac{d}{2}$.

 First, let $r > 0$ be arbitrary.
 Since $B_{r}(m) \subseteq B_{d/2+r}(p) \cap B_{d/2+r}(q)$ 
 and $B_{d/2+r}(p) \subseteq B_{d+r}(m)$,
 \begin{align*}
  \frac{\vol(B_{d/2+r}(p) \cap B_{d/2+r}(q))}{\vol(B_{d/2+r}(p))}
  \geq	\frac{\vol(B_{r}(m))}{\vol(B_{d+r}(m))}
  \geq	 \frac{1}{\CBG(n,-1,r, d + r)}.
 \end{align*}
 
 Now let $C_0 := \frac{1}{2 \mydelta} > 1$ 
 and $\hat{r} := \hat{r}(n,\mydelta,s) 
 := \min\big\{\frac{\ln(C_0)}{2(n-1)}, \frac{1}{s}\big\}$.
 For any fixed $0 < d < 2 \hat{r}$, define
 \[
  \tilde{\delta}_0(n,\mydelta,d) 
  := \inf\{ \delta' > 0 
  \mid \forall \delta > \delta' : f_{n,\delta}(\delta d) < C_0\} 
  \in [0,\infty]
 \]
 where for $\delta > 0$ and $r > 0$, 
 \[f_{n,\delta}(r) 
 := \CBG\big(n,-1,r, \Big(1 + \frac{1}{\delta}\Big) \cdot r\big).\] 
 In fact, this $\tilde{\delta}_0(n,\mydelta,d)$ is finite 
 and monotonically increasing in $d$ as will be proven next:
 Assume $\tilde{\delta}_0(n,\mydelta,d) = \infty$, 
 i.e.~there exists $\delta_m \to \infty$ 
 such that $f_{n,\delta_m}(\delta_m d) \geq C_0$. 
 Then
 \begin{align*}
  f_{n,\delta_m}(\delta_m d) 
  &= \CBG(n,-1,\delta_m d, \delta_m d + d) 
  \to e^{(n-1)d} \as {m \to \infty},
 \end{align*}
 and this implies $e^{(n-1)d} \geq C_0$.
 On the other hand, $e^{(n-1)d} < e^{2(n-1)\hat{r}} \leq C_0$. 
 This is a contradiction.
 Thus, $\tilde{\delta}_0(n,\mydelta,d) < \infty$. 

 Now let $d_1 < d_2$ and $\delta > \tilde{\delta}_0(n,\mydelta,d_2)$.
 Since $f_{n,\delta}$ is monotonically increasing in $r$,
 \begin{align*}
  C_0 > f_{n,\delta}(\delta d_2) \geq f_{n,\delta}(\delta d_1),
 \end{align*}
 i.e.~$\delta \geq \tilde{\delta}_0(n,\mydelta,d_1)$, 
 and this proves the monotonicity of $\tilde{\delta}_0(n,\mydelta,\cdot)$.   

 Hence, $\tilde{\delta}_0(n,\mydelta,d)$ decreases for decreasing $d$ 
 whereas $\frac{1}{sd} - \frac{1}{2}$ increases.
 Therefore, there exists $d_0 = d_0(n,\mydelta,s) \leq 2 \hat{r}$ 
 such that $\tilde{\delta}_0(n,\mydelta,d) \leq \frac{1}{sd} - \frac{1}{2}$ 
 for $d \leq d_0$.
 Let 
 \[
  \delta_0 
  = \delta_0(n,\mydelta,s) 
  := \tilde{\delta}_0(n,\mydelta,d_0(n,\mydelta,s)) 
  = \max\{\tilde{\delta}_0(n,\mydelta,d) \mid 0 < d \leq d_0\}
 \]
 where the monotonicity of $\tilde{\delta}_0$ is used in the last equality.
 Finally, for $d \leq d_0$ and $\delta_0 < \delta < \frac{1}{sd_0}-\frac{1}{2}$, let 
 $R
  := \frac{d}{2} + \delta d 
  = \Big(\frac{1}{2} + \delta\Big) \cdot d 
  < \Big(\frac{1}{2} + \frac{1}{sd_0}-\frac{1}{2} \Big) \cdot d_0 
  = \frac{1}{s}$. 
 Then
 \begin{align*}
  &\frac{\vol(B_{R}(p) \cap B_{R}(q))}{\vol(B_{R}(p))} 
  \geq	\frac{1}{\CBG(n,-1,\delta d, d + \delta d)}
  =	\frac{1}{f_{n,\delta}(\delta d)}
  > \frac{1}{C_0} = 2 \mydelta
  .\qedhere
 \end{align*}
\end{proof}

The next lemma will only be needed in \autoref{sec:step2} 
but is already given here
since its statement and the proof are similar to the previous one.

\begin{lemma}\label{lemma-criterion_volume_intersection2} 
 Let $(M,g)$ be a complete connected \ndim Riemannian manifold 
 with lower Ricci curvature bound $\Ric_M \geq -(n-1)$.
 For all $0<\mydelta<\frac{1}{2}$ and $R > 0$
 there is $d_0 = d_0(n,\mydelta,R) > 0$ such that for all 
 $p,q \in M$ with $d(p,q) < d_0$,
 \[
  \vol(B_R(p) \cap B_R(q)) 
  > 2\mydelta \cdot \max \{ \vol(B_R(p)) , \vol(B_R(q)) \}.
 \]
\end{lemma}
 
\begin{proof}
 Similarly to the proof of \autoref{lemma-criterion_volume_intersection},
 for arbitrary points $p,q \in M$ with distance $d := d(p,q) < 2R$, observe 
 \begin{align*}
  \frac{\vol(B_R(p) \cap B_R(q))}{\vol(B_R(p))} 
  \geq \frac{1}{\CBG(n,-1,R-\frac{d}{2}, R + \frac{d}{2})}.
 \end{align*}
 As $\CBG(n,-1,R-\frac{d}{2}, R + \frac{d}{2}) \to 1$ as $d \to 0$,
 there is $d_0 = d_0(n,\mydelta,R) \in (0,2R)$ 
 such that $\CBG(n,-1,R-\frac{d}{2}, R + \frac{d}{2}) < \frac{1}{2\mydelta}$
 for all $d \leq d_0$.
 In particular, for points $p,q \in M$ with $d(p,q) < d_0$,
 \[
  \vol(B_R(p) \cap B_R(q)) 
  > 2\mydelta \cdot \max \{ \vol(B_R(p)) , \vol(B_R(q)) \}.
  \qedhere
 \]
\end{proof}

An important notion for investigating the $\mathcal{C}$-property
is the \emph{distortion} of a function
which describes
how much a function changes the distance of two points.
In particular, it will be important to know 
how much the flow of a vector field changes the distance of two points up to some fixed time.
Recall that the distortion of a map $f: M \to N$ between Riemannian manifolds
is the function 
$\dt^f : M \times M \to [0,\infty)$ 
defined by
\begin{align*}
 \dt^f(p,q) &:= |d_N(f(p),f(q)) - d_M(p,q)|.
 \intertext{If $\Phi$ is the flow of a time-dependent vector field on $M$, 
 $t \in [0,1]$, $p,q \in M$ and $r \geq 0$, denote $\phi_t := \Phi(0,t,\cdot)$ and define}
 \dt(t)(p,q) &:= \max\{ \dt^{\phi_\tau}(p,q) \mid 0 \leq \tau \leq t\}  \quad\aand\\
 \dt_r(t)(p,q) &:= \min\{r, \dt(t)(p,q)\}.
\end{align*}

The subsequent lemma generalises \cite[Lemma 3.7]{kapovitch-wilking} 
and can be proven completely analogously to it.
\begin{lemma}\label{lemma:generalization_kw_3.7}
 For $\tilde{\alpha} \in (1,2)$ 
 there exist $C = C(n,\tilde{\alpha})$ 
 and $\hat{C} = \hat{C}(n,\tilde{\alpha})$ 
 such that the following holds
 for any $0 < R \leq 1$:
 Let $M$ be an \ndim Riemannian manifold 
 with $\Ric_M \geq -(n-1)$ 
 and $X : [0,1] \times M \to TM$ be a time dependent, 
 piecewise constant in time vector field with compact support and flow $\Phi$,
 define $\phi^{s}_t := \Phi(s,t,\cdot)$ 
 and let $c : [0,1] \to M$ be an integral curve of $X$ 
 such that $X^t$ is divergence free on $B_{10R}(c(t))$ 
 for all $t \in [0,1]$.
 
 Let $\tilde{\eps} := \int_0^1 (\Mx_R(\norm{\nabla.X^t}))\circ c(t) \dt$.
 Then for any $r \leq \frac{R}{10}$, 
 \begin{align*}
  \avgint_{B_r(c(s)) \times B_r(c(0))} \dt_r(1)(p,q) \dV(p,q) 
  \leq C r \cdot \tilde{\eps}.
 \end{align*}
 Furthermore, there is a subset $B_r(c(0))' \subseteq B_r(c(0))$ 
 with $c(0) \in B_r(c(0))'$ such that
 \[\vol(B_r(c(0))') \geq (1-C \tilde{\eps}) \cdot \vol(B_r(c(0))).\] 
 Finally, for any $t \in [0,1]$,
 \[\phi^{0}_t(B_r(c(0))') \subseteq B_{\tilde{\alpha} r}(c(t))
 \quad\aand \quad
 \vol(B_r(c(t))) \leq \hat{C} \cdot \vol(B_r(c(0))).\]
\end{lemma}

\begin{proof}
 The proof can be done completely analogously to the one 
 of \cite[Lemma 3.7]{kapovitch-wilking} 
 by replacing $\frac{r}{10}$ in the induction by $\frac{r}{m}$ 
 where $m = 2 \cdot \frac{\tilde{\alpha} + 1}{\tilde{\alpha} - 1} > 0$.
 Again, the constants $C$ and $\hat{C}$ can be made explicit 
 in terms of the constant appearing in \refBGThm.
\end{proof}

The following lemma states 
that the flow of a time dependent vector field 
as in the definition of the $\mathcal{C}$-property 
is Lipschitz on certain small sets.

\begin{lemma}\label{lemma:flow_Lipschitz_on_subset}
 Given $\alpha \in (1,2)$, 
 there exist constants $C_0 = C_0(n,\alpha)$ and $C_0' = C_0'(n,\alpha)$ such that 
 for $0 < \mydelta < \frac{1}{2 C_0}$ and $0 < R \leq 1$ 
 there is a number $\hat{r}_0 = \hat{r}_0(n,\mydelta,\alpha,R) < \frac{R}{20 \alpha}$
 satisfying the following:
 
 Let $M$ be an \ndim Riemannian manifold
 with $\Ric_M \geq -(n-1)$, 
 $X : [0,1] \times M \to TM$ a time dependent, 
 piecewise constant in time vector field with compact support and flow $\Phi$,
 $\phi^{s}_t := \Phi(s,t,\cdot)$, $\phi_t := \phi^0_t$,
 $c : [0,1] \to M$ an integral curve of $X$ 
 such that $X^t$ is divergence free on $B_{10R}(c(t))$ for all $t \in [0,1]$
 and 
 $\int_0^1 (\Mx_{2R}(\norm{\nabla.X^t}^{3/2}) (c(t))) ^{2/3}\dt < \mydelta$.
 
 Let $p := c(0)$ and $0 < r < \hat{r}_0$.
 Then there exists a subset $B_r(p)'' \subseteq B_r(p)$ containing $p$ 
 with $\vol(B_r(p)'')> (1-C_0' \sqrt{\mydelta}) \cdot \vol(B_r(p))$ 
 such that $\phi_t$ is $\alpha$-bi-Lipschitz on $B_r(p)''$ 
 for any $t \in [0,1]$.
\end{lemma}

\begin{proof}
 Define 
 $\tilde{\alpha} := \frac{\alpha + 1}{2} \in (1,\frac{3}{2}) \subseteq (1,2)$
 and fix the following constants:
 \begin{itemize}
 \item 
  Let $C = C(n,\alpha)$ be the $C(n,\tilde{\alpha})$ 
  and $\hat{C} = \hat{C}(n,\alpha)$ be the $\hat{C}(n,\tilde{\alpha})$ 
  appearing in \autoref{lemma:generalization_kw_3.7}.
  \item 
   Let $\tilde{C} = \tilde{C}(n)>0$ be the constant 
   of \cite[formula (6)]{kapovitch-wilking} satisfying 
   \[
    \Mx_{\rho}[\Mx_{\rho}(f)](x) 
    \leq \tilde{C}(n) \cdot \big( \Mx_{2\rho}(f^{3/2})(x)\big)^{2/3}
   \] 
   for any $f \in L^{3/2}(M)$ and $0 < \rho \leq 1$.
 \item 
  Let $C_0 = C_0(n,\alpha) := \tilde{C} \cdot C$.
 \item 
  Let $\hat{C}_0 = \hat{C}_0(n,\alpha) := 2 \hat{C} \cdot \CBG(n,-1,\frac{1}{10},\frac{\alpha}{10})$.
 \item 
  Let $C_0' = C_0'(n,\alpha) := \hat{C}_0 + \sqrt{\frac{C_0}{2}}$.
 \end{itemize}
 Fix $0 < \mydelta < \min\big\{\frac{1}{2 C_0},1\big\}$ and $0 < R \leq 1$.
 First, observe
 \begin{align*}
  \tilde{\eps} 
  :=
  \int_0^1 (\Mx_R(\norm{\nabla.X^t})) (c(t)) \dt 
  &\leq \int_0^1 \Mx_R(\Mx_R(\norm{\nabla.X^t})) (c(t)) \dt \\
  &\leq \tilde{C} 
  \cdot \int_0^1 \Mx_{2R}(\norm{\nabla.X^t}^{3/2})^{2/3} (c(t)) \dt \\
  &< \tilde{C} \mydelta.
 \end{align*}
 In particular, $C \tilde{\eps} < C_0\,\mydelta < \frac{1}{2}$.
 By \autoref{lemma:generalization_kw_3.7},
 for all $r \leq \frac{R}{10}$,
 \[
  \avgint_{B_r(p) \times B_r(p)} \dt_{r}(1) \dV 
  \leq C r \cdot \tilde{\eps} < C_0\,\mydelta \cdot r
 \]
 and there is a subset $B_r(p)' \subseteq B_r(p)$ containing $p$ with 
 \[
  \vol(B_r(p)') 
  \geq (1-C_0\,\mydelta) \cdot \vol(B_r(p)) 
  > \frac{1}{2} \cdot \vol(B_r(p))
 \]
 and $\phi_t(B_r(p)') \subseteq B_{\alpha r}(c(t))$ for all $t \in [0,1]$.
 Furthermore,
 for $r \leq \frac{R}{10}$,
 \begin{align*}
  \frac{\vol(B_{\alpha r}(c(t)))}{\vol(B_r(p)')}
  &\leq	\CBG(n,-1,r,\alpha r) \cdot \frac{\vol(B_r(c(t)))}{\vol(B_r(p)')} \\
  &\leq	\CBG\Big(n,-1,\frac{1}{10},\frac{\alpha}{10}\Big) \cdot \hat{C} 
   \cdot \frac{\vol(B_r(p))}{\vol(B_r(p)')} \\
  &\leq \hat{C}_0.
 \end{align*}
 Moreover,
 \begin{align*}
  &\avgint_{B_r(p)'} \int_0^1 \Mx_R(\norm{\nabla.X^t}) 
  \circ \phi_{t} (x) \dt \dV(x) \\
  &= \int_0^1 \frac{1}{\vol(B_r(p)')} 
  \cdot \int_{\phi_{t}(B_r(p)')} \Mx_R(\norm{\nabla.X^t}) (x) \dV(x) \dt 
  \displaybreak[0]\\
  &\leq \int_0^1 \frac{\vol(B_{\alpha r}(c(t)))}{\vol(B_r(p)')} 
  \cdot \avgint_{B_{\alpha r}(c(t))} \Mx_R(\norm{\nabla.X^t}) (x) \dV(x) \dt 
  \displaybreak[0]\\
  &\leq \hat{C}_0 \cdot \int_0^1 \max_{0 \leq \rho \leq R} 
  \avgint_{B_{\rho}(c(t))} \Mx_R(\norm{\nabla.X^t}) (x) \dV(x) \dt 
  \displaybreak[0]\\
  &\leq	\hat{C}_0 \cdot \tilde{C} \int_0^1 
  (\Mx_{2R}(\norm{\nabla.X^t}^{3/2}))^{2/3}(c(t)) \dt\\
  &< 	\hat{C}_0 \cdot \tilde{C} \mydelta.
 \end{align*}

 Define 
 \[
  B_r(p)'' 
  := \{x \in B_r(p)' 
  \mid \int_0^1 \Mx_R(\norm{\nabla.X^t}) \circ \phi_t (x) \dt 
  < \tilde{C} \cdot \sqrt{\mydelta}\}.
 \]
 Note $p \in B_r(p)''$ due to 
 $\int_0^1 \Mx_R(\norm{\nabla.X^t}) \circ c(t) \dt 
 < \tilde{C} \cdot \mydelta \leq \tilde{C} \cdot \sqrt{\mydelta}$
 and 
 $\phi_t(p) = c(t)$. 
 Furthermore,
 \begin{align*}
  \vol(B_r(p)'')
  &> (1- \hat{C}^2\sqrt{\mydelta}) \cdot (1 - C_0\mydelta)
  \cdot \vol(B_r(p)) \\
  &> (1-C_0' \sqrt{\mydelta}) \cdot \vol(B_r(p))
 \end{align*}
 using 
 $C_0' 
 = \hat{C}_0 + \sqrt{\frac{C_0}{2}} 
 = \hat{C}_0 + C_0 \cdot \sqrt{\frac{1}{2C_0}} 
 > \hat{C}_0 + C_0 \sqrt{\mydelta}$.

 Moreover, points in $ B_r(p)''$ satisfy the following: 
 Fix $t \in [0,1]$,
 $a \in B_r(p)''$ and let $\tilde{a} := \phi_t(a)$.
 In particular, 
 $\int_0^1 \Mx_R(\norm{\nabla.X^t}) \circ \phi_t(a) \dt 
 < \tilde{C} \cdot \sqrt{\mydelta}$
 and, by \autoref{lemma:generalization_kw_3.7},
 for any $\rho \leq \frac{R}{10}$ 
 there are subsets $B_{\rho}(a)' \subseteq B_{\rho}(a)$ 
 and $B_{\rho}(\tilde{a})' \subseteq B_{\rho}(\tilde{a})$ 
 such that
 \begin{align*}
  \vol(B_{\rho}(a)') \geq (1-C_0 \sqrt{\mydelta}) \cdot \vol(B_{\rho}(a))
  &\quad\aand\quad 
  \phi_t(B_{\rho}(a)') \subseteq B_{\tilde{\alpha} {\rho}}(\tilde{a}),
  \\
  \vol(B_{\rho}(\tilde{a})') 
  \geq (1-C_0 \sqrt{\mydelta}) \cdot \vol(B_{\rho}(\tilde{a}))
  &\quad\aand\quad 
  \phi^{t}_{-t}(B_{\rho}(\tilde{a})') \subseteq B_{\tilde{\alpha} {\rho}}(a)
 \end{align*}
 where $C_0 \sqrt{\mydelta} < C_0\cdot \mydelta < \frac{1}{2}$.

 Let $d_0 = d_0(n,\mydelta,\alpha, R)$ 
 and $\delta_0 = \delta_0(n,\mydelta,\alpha, R)$, respectively,
 denote the constants $d_0(n,C_0\sqrt{\mydelta},\frac{10}{R})$ 
 and $\delta_0(n,C_0\sqrt{\mydelta},\frac{10}{R})$, respectively,
 of \autoref{lemma-criterion_volume_intersection}.
 This $d_0 < \frac{R}{10}$ can be chosen so small that 
 $\delta_0 \leq \frac{1}{2} - \frac{1}{\alpha+1} 
 = \frac{1}{2} - \frac{1}{2\tilde{\alpha}} < \frac{1}{2}$. 
 Define 
 \[
  \hat{r}_0 
  = \hat{r}_0(n,\mydelta,\alpha,R) 
  := \frac{d_0}{2\alpha} 
  < \frac{R}{20 \alpha}.
 \]

 From now on, assume $r < \hat{r}_0$ 
 and let $b \in B_r(p)''$ be another point. 
 In particular, $d := d(a,b) < 2r < \frac{d_0}{\alpha} < d_0$. 
 For arbitrary $\delta_0 < \delta < \frac{1}{2}$, 
 let $\rho := (\frac{1}{2} + \delta) d < \frac{R}{10}$.
 By \autoref{lemma-criterion_intersection_subsets} 
 and \autoref{lemma-criterion_volume_intersection}, 
 there exists a point $z \in B_{\rho}(a)' \cap B_{\rho}(b)'$.
 Thus,
 \begin{align*}
  d(\phi_t(a), \phi_t(b))
  &\leq d(\phi_t(a), \phi_t(z)) + d(\phi_t(z), \phi_t(b))\\
  &< 2 \cdot \tilde{\alpha} \rho
  \\&= \tilde{\alpha} \cdot(2\delta+1) d.
 \end{align*}
 Since $\delta > \delta_0$ was arbitrary 
 and $(2\delta_0 + 1) \cdot \tilde{\alpha} \leq \alpha$, 
 this proves
 \[d(\phi_t(a), \phi_t(b)) \leq \alpha \cdot d(a,b).\] 

 As before, let $\tilde{a} = \phi_t(a)$ and $\tilde{b} = \phi_t(b)$. 
 These points have distance
 $d(\tilde{a},\tilde{b}) = d(\phi_t(a), \phi_t(b)) \leq \alpha \cdot 2r < d_0$, 
 and an analogous argumentation gives
 \[
  d(a,b) 
   = d(\phi^t_{-t}(\tilde{a}), \phi^t_{-t}(\tilde{b})) 
   \leq \alpha \cdot d(\tilde{a}, \tilde{b}) 
   = \alpha \cdot d(\phi_t(a), \phi_t(b)).
  \] 
 Thus, $\phi_t$ is $\alpha$-bi-Lipschitz on $B_r(p)''$ for $r < \hat{r}_0$.
\end{proof}

If a sequence of manifolds satisfies the previous lemma 
and the rescaled manifolds 
endowed with the end points of the integral curve as base points converge,
the limits have the same dimension.
\begin{prop}\label{prop:bilipschitz_implies_dimension}
 Let $(M_i)_{i \in \nn}$ be a sequence of \ndim Riemannian manifolds 
 with 
 $\Ric_{M_i} \geq -(n-1)$.
 For every $i \in \nn$, 
 let $X_i : [0,1] \times M \to TM$ be a time dependent, 
 piecewise constant in time vector field with compact support 
 and flow $\Phi_i$,
 $\phi_i^t := \Phi_i(0,t,\cdot)$, 
 $c_i : [0,1] \to M_i$ be an integral curve of $X_i$ 
 such that $X_i^t$ is divergence free on $B_{10r}(c_i(t))$ 
 for all $t \in [0,1]$
 and 
 $\int_0^1 (\Mx_{2r}(\norm{\nabla.X_i^t}^{3/2}) (c_i(t))) ^{2/3}\dt 
 < \mydelta$ 
 for some $0 < r \leq 1$ and $\mydelta>0$.
 
 Assume that $x_i' := c_i(0)$ and $y_i' := c_i(1)$ satisfy
 $d(x_i',y_i') \leq 2$
 and let $\lambda_i \to \infty$ be scales such that
 $(\lambda_i M_i, x_i') \to (X, x_{\infty})$ 
 and $(\lambda_i M_i, y_i') \to (Y, y_{\infty})$ as $i \to \infty$.
 Then $\dim(X) = \dim(Y)$.
\end{prop}

\newcommand{\picturelipschitz}{
   \begin{tikzpicture}[auto,>=stealth]
     \node (x1) at (-0.2,3) 
       {$\B^{M_i}_{r/\lambda_i}(x_i')$};
     \node (hx) at (0.9,3) 
       {$\cong$};
     \node (x2) at (2,3) 
       {$\B^{\lambda_i M_i}_{r}(x_i')$};
     \node (x3) at (4.9,3)
       {$\B^X_{r}(x_\infty)$};
     \node (s1) at (-0.2,1.5)
       {$\overline{B^{M_i}_{r/\lambda_i}(x_i')''}$};
     \node (hs) at (0.9,1.5)
       {$~\cong$};
     \node (s2) at (2,1.5)
       {$G_i$};
     \node (s3) at (4.5,1.5) 
       {$S_r$};
     \node (y1) at (-0.2,0)
       {$\B^{M_i}_{\alpha r/\lambda_i}(y_i')$};
     \node (hy) at (0.9,0)
       {$\cong$};
     \node (y2) at (2,0)
       {$\B^{\lambda_i M_i}_{\alpha r}(y_i')$};
     \node (y3) at (4.9,0)
      {$\B^Y_{\alpha r}(y_\infty)$};
     \node[rotate=90] (ss1) at (0,2.25)
       {$\subseteq$};
     \node[rotate=90] (ss2) at (2,2.25)
       {$\subseteq$};
     \node[rotate=90] (ss3) at (4.5,2.25) 
       {$\subseteq$};
     \path 
       (3.0,3.0) edge[->] node [above] 
         {$i\to\infty$}
         (4.0,3.0)	
       (3.0,0.0) edge[->] node [above] 
         {$i\to\infty$}
         (4.0,0.0)
       (3.0,1.5) edge[->] node [above]
         {$i\to\infty$}
         (4.0,1.5)
       (0.0,1.15)edge[->] node [left]
         {$\phi_i^1$}
         (0.0,0.3)	
       (2.0,1.2) edge[->] node [left]
         {$f_i$}
         (2.0,0.3)	
       (4.5,1.2) edge[->] node [right]
         {$f_r$}
         (4.5,0.3)
     ;
    \end{tikzpicture}
}

\begin{proof}
 The proof consists of three steps: 
 First, for any radius $r > 0$, 
 construct a bi-Lipschitz map between subsets of
 $\B^X_r(x_\infty)$ and $\B^Y_{\alpha r}(y_\infty)$, 
 cf.~\autoref{pic:defn_of_f_R}.
 Next, observe that these subsets have positive volume. 
 In particular, they intersect the set of generic points.
 Finally, repeating the argument for the limit spaces proves the claim.
 
 Choose any $\alpha \in (1,2)$.
 Without loss of generality, let $i \in \nn$ be large enough 
 such that $r < \lambda_i \cdot \hat{r}_0$ 
 where $\hat{r}_0 = \hat{r}_0(\alpha)$ is the constant 
 from \autoref{lemma:flow_Lipschitz_on_subset}.
 Furthermore, let 
 $B^{M_i}_{r/\lambda_i}(x_i')'' \subseteq B^{M_i}_{r/\lambda_i}(x_i')$ 
 and 
 $\phi_i^1 : B^{M_i}_{r/\lambda_i}(x_i')'' 
 \to B^{M_i}_{\alpha r/\lambda_i}(y_i')$ 
 be as in \autoref{lemma:flow_Lipschitz_on_subset}.
 Since $\phi_i^1$ is $\alpha$-bi-Lipschitz, 
 it can be extended to an $\alpha$-bi-Lipschitz map 
 $\phi_i^1: \overline{B^{M_i}_{r/\lambda_i}(x_i')''} 
 \to \B^{M_i}_{\alpha r/\lambda_i}(y_i')$.
  
 In order to regard $\phi_i^1$ 
 as a map $\lambda_i M_i \to \lambda_i M_i$ 
 instead of $M_i \to M_i$,
 let $G_i$ denote this closure $\overline{B^{M_i}_{r/\lambda_i}(x_i')''}$ 
 regarded as a subset of $\B_r^{\lambda_i M_i}(x_i') \subseteq \lambda_i M_i$. 
 Correspondingly, define $f_i : G_i \to B^{\lambda_i M_i}_{\alpha r}(y_i')$ 
 by $f_i(q) := \phi_i^1(q)$, cf.~\autoref{pic:defn_of_f_R}. 
 By definition, this map is $\alpha$-bi-Lipschitz and measure preserving.
    
 \begin{figure}[t]
  \begin{center}
    \picturelipschitz
    \caption{Sets and maps used to construct $f_r: S_r \to \B_{\alpha r}(y_\infty)$.}
    \label{pic:defn_of_f_R} 
   \end{center}
  \end{figure}  

 After passing to a subsequence,
 the $G_i$ converge to a compact subset $S_r \subseteq \B_r^{X}(x_\infty)$ 
 and an $\alpha$-bi-Lipschitz homeomorphism $f_r: S_r \to f_r(S_r)$ 
 such that $f_r(S_r)$ is the limit of the $f_i(G_i)$, 
 cf.~again \autoref{pic:defn_of_f_R}. 
 
 Now find a point $x_0 \in S_r$ 
 such that both $x_0$ and $f_r(x_0)$ are generic: 
 There exists a constant $C > 0$ such that 
 \begin{align*}
  \vol_{Y}(f_r(S_r \cap \Xgen) \cap \Ygen)
  &= \vol_{Y}(f_r(S_r \cap \Xgen))  
  \\&= C \cdot \vol_X(S_r \cap \Xgen)
  \\&= C \cdot \vol_X(S_r) > 0.
 \end{align*}
 Hence, there exists $x_0 \in S_r \cap \Xgen$ with image $f_r(x_0) \in \Ygen$.
 By similar arguments as before, for $\lambda \to \infty$, 
 the sets $(\lambda S_r,x_0) \subseteq (\lambda X,x_0)$ 
 and $(\lambda f_r(S_r),f_r(x_0)) \subseteq (\lambda Y,f_r(x_0))$, respectively, 
 (sub)converge to limits $S_{\infty}$ and $S_{\infty}'$, respectively.
 Moreover, for $\lambda f_r(x) := f_r(x)$,
the $\alpha$-bi-Lipschitz maps 
 $\lambda f_r: (\lambda S_r, x_0) \to (\lambda f_r(S_r),f_r(x_0))$ 
 (sub)converge to an $\alpha$-bi-Lipschitz map 
 $f: S_{\infty} \to S_{\infty}'$ as $\lambda \to \infty$. 
 Since $x_0$ and $f_r(x_0)$ are generic, 
 one has $S_{\infty} \subseteq \rr^{\dim(X)}$ 
 and $S_{\infty}' \subseteq \rr^{\dim(Y)}$.
 Furthermore, $\vol(S_{\infty}) > 0$.
 This implies 
 \[
  \dim(X) = \dim(Y).
  \qedhere
 \]
\end{proof}

\subsection{Proof of \autoref{prop:main}}
\label{sec:proof_main_prop}
The idea is to intersect the sets constructed
in \autoref{lem:set_where_all_blow_ups_split__locally} 
and \autoref{nearly_splitting_after_rescaling_locally}.
For fixed base points $x_i$ in the intersection 
and the $\lambda_i$ of \autoref{nearly_splitting_after_rescaling_locally},
the $(\lambda_i M_i,x_i)$ are both 
close to products $(\rr^k \times K_i, \cdot)$ 
and converging to a product $(\rr^k \times Y, \cdot)$
where the $K_i$ are compact with diameter $1$ and $Y$ is some metric space.
The following (technical) lemmata show that this space $Y$ in fact is compact. 
Subsequently, the main proposition can be proven.

A map $f : (X,d_X) \to (Y,d_Y)$ between two metric spaces 
is called \emph{$\eps$-isometry}, where $\eps>0$, 
if $|d_Y(f(p),f(q)) - d_X(p,q)| < \eps$ for all $p,q \in X$.

\begin{lemma}\label{cont_eps_isometry_sphere_to_sphere_is_surjective}
 Fix 
 $R > r \geq 0$, 
 $k \in \nn$, 
 $S_R := S_R^{\rr^k}(0)$,
 $\eps > 0$, 
 and let $f: S_R \to \bar{B}_R^{\rr^k}(0)\setminus B_{R-r}^{\rr^k}(0)$ 
 be a continuous $\eps$-isometry with $\eps < 2 \cdot (R - r)$. 
 Define $\pr: \bar{B}_R^{\rr^k}(0)\setminus B_{R-r}^{\rr^k}(0) \to S_R$ 
 by $\pr(p) := \frac{R}{\norm{p}} \cdot p$. 
 Then $\pr \circ f : S_R \to S_R$ is surjective.
\end{lemma}

\begin{proof}
 Denote the distance function on $\rr^k$ by $d$
 and
 distinguish the two cases of $r = 0$ and $r>0$:
 First, let $r = 0$, i.e.~$f(S_R) \subseteq S_R$ and $\pr = \id$. 
 Assume that there exists a point $p \in S_R \setminus f(S_R)$ 
 and define $j: S_R\setminus\{p\} \to \rr^{k-1}$ 
 as the stereographic projection. 
 Then the composition $j \circ f : S_R \to \rr^{k-1}$ is continuous and, 
 by the theorem of Borsuk-Ulam, 
 there exist $\pm q \in S_R$ such that $j \circ f(q) = j \circ f(-q)$. 
 Since $j$ is a homeomorphism, $f(q) = f(-q)$. 
 Hence, 
 $\eps > |d(f(q),f(-q)) - d(q,-q)| = 2R$.
 This is a contradiction. Therefore, $f$ is surjective.
 
 Now let $r>0$ be arbitrary.
 For any $p,q \in \bar{B}_R^{\rr^k}(0)\setminus B_{R-r}^{\rr^k}(0)$, 
 \begin{align*}
  &|d(\pr \circ f(p), \pr \circ f(q)) - d(p,q)|\\
  &\leq |d(\pr \circ f(p), \pr \circ f(q)) - d(f(p),f(q))| 
    + |d(f(p), f(q)) - d(p,q)|\\
  &\leq d(\pr \circ f(p), f(p)) + d(\pr \circ f(q),f(q)) + \eps\\
  &\leq 2r + \eps.
 \end{align*}
 Therefore, $\pr \circ f$ is a continuous $2r + \eps$-isometry and, 
 by the first part, surjective. 
\end{proof}

The following lemma states that, 
if two products $\rr^k \times X$ and $\rr^k \times Y$ are sufficiently close 
and $X$ is compact,
then $Y$ is compact as well with similar diameter as $X$. 
\begin{lemma}\label{lemma:rrkxX:dGH-small->diamX_similar}
Let $(X,d_X)$ and $(Y,d_Y)$ be complete length spaces, $X$ be compact, 
$x_0 \in X$, $y_0\in Y$ and 
define 
\[\rad_Y(y_0) := \sup\{d_Y(y,y_0) \mid y \in Y\}.\]
Let $k \in \nn$, $R>\diam(X)$ and $\eps > 0$.
Then the following is true:
\begin{enumerate}
 \item \label{lemma:rrkxX:dGH-small->diamX_similar--preparation}
  If $\diam(X) + 4 \eps \leq \frac{2R}{3}$ and 
  $\dgh{R}{\rr^k \times X}{(0,x_0)}{\rr^k \times Y}{(0,y_0)} 
  < \frac{\eps}{2}$, 
  then
  \[
   \min\{R,\rad_Y(y_0)\}^2 
   \leq \diam(X)^2 + 2 \eps \cdot \diam(X) + 4 \eps(R + \eps).
  \]
 \item \label{lemma:rrkxX:dGH-small->diamX_similar--diam=0}
  If $\diam(X) = 0$ and 
  $\dgh{R}{\rr^k \times X}{(0,x_0)}{\rr^k \times Y}{(0,y_0)} < \frac{R}{12}$, 
  then $Y$ is compact with $\diam(Y) < 2R$.
 \item \label{lemma:rrkxX:dGH-small->diamX_similar--diam=1}
  If $\diam(X) = 1$ and 
  $\dgh{R}{\rr^k \times X}{(0,x_0)}{\rr^k \times Y}{(0,y_0)} 
  < \frac{1}{100R}$
  for some $R \geq 100$, 
  then $Y$ is compact with $c \leq \diam(Y) \leq 5c$ for a small constant $c>0$.
 \item \label{lemma:rrkxX:dGH-small->diamX_similar--diam=>dgh}
  If $1000 \cdot \diam(X) \leq R < \frac{1}{2} \cdot \diam(Y)$,
  then 
  \[
   \dgh{R}{\rr^k \times X}{(0,x_0)}{\rr^k \times Y}{(0,y_0)} 
   \geq 20 \cdot \diam(X).
  \]
 \end{enumerate}
\end{lemma}

\newcommand{\picturerrkxXS}[1][t]{
 \begin{figure}[#1]
  \begin{center}
   \begin{tikzpicture}[auto,>=stealth]
     \node (x1) at (-1.1,0)
       {$\partial\B_{R}^{\rr^k}(0) \times \{y_0\} = S_R \times \{y_0\}$};
     \node[rotate=90](xx1) at (0,0.5)
       {$\subseteq$};
     \node (x2) at (0,1)
       {$B_R^{\rr^k \times Y} ((0,y_0))$};
     \node (y1) at (4.2,0)
       {$g(S_R \times \{y_0\}) = S$};
     \node[rotate=90](yy1) at (4,0.5)
       {$\subseteq$};     
     \node (y2) at (4,1)
       {$B_R^{\rr^k \times X} ((0,x_0))$};     
     \path 
       (1.45,0) edge[->] node[above] {$g$} (2.65,0)
       (1.45,1) edge[->] node[above] {$g$} (2.65,1)
     ;
   \end{tikzpicture}
   \caption{Definition of $S$.}
   \label{pic:rrkxX:dGH-small->diamX_similar--def_S} 
  \end{center}
 \end{figure}
}

\newcommand{\picturerrkxXH}[1][t]{
\begin{figure}[#1]
 \begin{center}
  \begin{tikzpicture}[auto,>=stealth]
	\node 		(x2) 	at (-0.2,0.0)	
	  {$S_R \times \{y_0\}$};
	\node[rotate=90](xx1)	at (-0.3,0.5)	
	  {$\supseteq$};
	\node 		(x1) 	at (-0.45,1.0)	
	  {$\Gamma = \tilde{\Gamma} \times \{y_0\}$};
	\node[rotate=90](xx1)	at (-0.3,-0.5)	
	  {$\subseteq$};
	\node 		(y1) 	at (-0.45,-1.0)	
	  {$\Delta(\gamma_1,\ldots,\gamma_m)$};
	\node 		(y2)	at ( 5.0,0.0)	
	  {$H(S_R \times \{y_0\}) \subseteq \rr^k$};
	\node[rotate=90](yy1)	at ( 4.6,0.5)	
	  {$\supseteq$};
	\node 		(y1) 	at ( 4.6,1.0)	
	  {$h(\Gamma)$};
	\node[rotate=90](xx1)	at ( 4.6,-0.5)	
	  {$\subseteq$};
	\node 		(y1) 	at ( 5.0,-1.0)	
	  {$\Delta(h(\gamma_1),\ldots,h(\gamma_m))$};
	\path 
	  (0.8,1)  edge[->] node[above] {$h=\prrrk \circ g$}	(3.2,1)
	  (0.8,0)  edge[->] node[above] {$H$} 			(3.2,0)
	  (0.8,-1) edge[|->] node[above] {} 			(3.2,-1)
	  ;
  \end{tikzpicture}
  \caption{Definition of $H$.}
  \label{pic:rrkxX:dGH-small->diamX_similar--def_H} 
 \end{center}
\end{figure}
}

\newcommand{\picturerrkxXdistS}[1][t]{
\begin{figure}[#1] 
 \begin{center}
  \begin{tikzpicture}
  [auto,>=stealth,inner sep=0pt,thick,
  dot/.style={fill=black,circle,minimum size=3pt}]
    \coordinate[] 	(0) 		at (0.0,0.0);
    \node 	      	(0-label) 	at (-0.2,0)	
      {$0$};
    \coordinate[] 	(proj) 		at (0.84,0.84);
    \node 	      	(proj-label) 	at (-1.4,1.25)
      {$\prrrk(g(0,y_n))$};
    \path (-0.2,1.25) edge [->,densely dashed,bend left=15] (0.76,0.9);
    \coordinate[] 	(p) 		at (0.515,0.515);
    \node	 	(p-label)	at (0.5,0.25)	
      {$p$};
    \coordinate[] 	(q) 		at (0.3,0.66);
    \node	 	(q-label)	at (0.3,0.9)	
      {$q$};
    \coordinate[] 	(endpt) 	at (1.41,1.41);
    \node	 	(boundary)	at (1.6,2) 
      {$\partial\B_R(0)$};
    \draw[decorate,decoration={snake,amplitude=.2mm,
      segment length=2mm,post length=0.2mm}]  (0.7,0) arc (0:90:0.7);
    \draw[]  (2,0) arc (0:90:2);
    \node	 	(boundary)	at (1.5,-0.3)	
      {$H(S_R \times \{y_0\})$};
    \draw[dotted]	 (0) -- (endpt);
    \node	 	(sigma)	at (1,1.2)
      {$\sigma$};
    \foreach \point in {0,p,q,proj,endpt}
    \fill [black] (\point) circle (1.7pt);
  \end{tikzpicture}
  \caption{Points used to estimate $d_{\rr^k}(\prrrk(g(0,y_n)),\prrrk(S))$.}
  \label{pic:rrkxX:dGH-small->diamX_similar--dist_to_pr(S)} 
 \end{center}
\end{figure}
}

\begin{proof}
a) 
The idea is 
to approximate the diameter of $Y$ 
by taking points $y'$ as far away from $y_0$ as possible,
to map both the set $S_R \times \{y_0\}$ 
and the points $(0,y')$ to $\rr^k \times X$ via $\eps$-approximations,
to take the projection onto the Euclidean factor 
and to find an upper and lower estimate 
for the distance of the obtained set and point.
Finally, comparison of this upper and lower bound gives the result.

Let $(f,g)$ be $\eps$-approximations 
between the balls $\B_R^{\rr^k \times X}((0,x_0))$ 
and $\B_R^{\rr^k \times Y}((0,y_0))$ 
with $f((0,x_0)) = (0,y_0)$ and $g((0,y_0)) = (0,x_0)$.
Let 
\begin{align*}
d &:= \diam(X) \quad\aand\quad
\delta := \min\{R, \rad_Y(y_0)\}.
\end{align*}
For each $n \in \nn$, $n \geq 1$, choose $y_n \in Y$ 
such that $\delta - \frac{1}{n} \leq \delta_n := d_Y(y_n,y_0) \leq \delta$. 
(If $\rad_Y(y_0) > R$, choose $y_n \in \partial\B_{R}(0)$ 
which is nonempty since $Y$ is a length space; 
otherwise, by definition, there exists a sequence $\bar{y}_n$ 
satisfying $d_Y(\bar{y}_n,y_0) > \delta - \frac{1}{n}$.)
In particular, $\delta_n$ is convergent with limit $\delta$.

\picturerrkxXS

Let $S_R := \partial \B_R^{\rr^k}(0) \subseteq \rr^k$ 
and $S := g(S_R \times \{y_0\}) \subseteq \rr^k \times X$, 
cf.~\autoref{pic:rrkxX:dGH-small->diamX_similar--def_S}.
Since 
\begin{align*}
 &d_{\rr^k \times X}(g(0,y_n),p) 
 \\&= \sqrt{ d_{\rr^k}(\prrrk(g(0,y_n)),\prrrk(p))^2 
  + d_X(\pr_X(g(0,y_n)),\pr_X(p))^2}
 \\&\leq d_{\rr^k}(\prrrk(g(0,y_n)),\prrrk(S))^2 + d^2
 \intertext{for every $p \in S$,}
 &\sqrt{R^2 + \delta_n^2} - \eps\\
 &=d_{\rr^k \times Y}((0,y_n),S_R \times \{y_0\}) - \eps\\
 &\leq d_{\rr^k \times X}(g(0,y_n),S) \\
 &\leq \sqrt{ d_{\rr^k}(\prrrk(g(0,y_n)),\prrrk(S))^2 + d^2}
\end{align*}
and this proves the lower bound
\[
 d_{\rr^k}(\prrrk(g(0,y_n)),\prrrk(S)) 
 \geq \sqrt{(\sqrt{R^2 + \delta_n^2}-\eps)^2 - d^2}.
\]

In order to find the upper bound, choose a natural number $m \geq \frac{2}{\eps}$ 
and let $\Delta$ be a spherical triangulation of $S_R$ 
such that the set of vertices $\tilde{\Gamma}$ of $\Delta$ 
is a finite $\frac{1}{m}$-net in $S_R$
and each two vertices of a simplex have 
(spherical) distance at most $\frac{1}{m}$.
(Notice that their Euclidean distance is at most $\frac{1}{m}$ as well.)
Define $\Gamma := \tilde{\Gamma} \times \{y_0\}$ 
and $h := \prrrk \circ g: \Gamma \to \rr^k$
and extend $h$ to a continuous map 
$H : S_R \times \{y_0\} 
\to \B_R^{\rr^k}(0) \setminus B_{R-(d + 3\eps)}^{\rr^k}(0)$ 
by mapping each (spherical) simplex of $\Delta$ with vertices $\gamma_i$
continuously to the corresponding (Euclidean) simplex in $\rr^k$ 
with vertices $h(\gamma_i)$,
cf.~\autoref{pic:rrkxX:dGH-small->diamX_similar--def_H}.
Since $\Gamma$ is finite, $H$ is continuous. 

Then $h(\Gamma)$ defines an $(\frac{1}{m} + \eps)$-net 
in $H(S_R \times \{y_0\})$:
Since each two vertices of a simplex in $\Delta$ have (Euclidean) distance 
at most $\frac{1}{m}$,
their images have distance at most $\frac{1}{m} + \eps$. 
Hence, each point $x \in H(S_R\times\{y_0\})$ is contained 
in a Euclidean simplex 
whose vertices have pairwise distance at most $\frac{1}{m} + \eps$.
Recall that, since the simplex is Euclidean,
$x$ has distance at most $\frac{1}{m} + \eps$ to each of these vertices.
Let $h(\gamma)$ denote one of those vertices. 
In particular, $x \in \B_{1/m + \eps}(h(\gamma))$,
hence,
$H(S_R \times \{y_0\}) 
\subseteq \bigcup \{ \B_{1/m + \eps} (\gamma') \mid \gamma' \in H(\Gamma)\}$.

Furthermore, $H$ is a $(5\eps+d)$-isometry: Let $p,q \in S_R$ be arbitrary. 
Choose points $\gamma_p, \gamma_q \in \tilde{\Gamma}$ 
such that $d_{\rr^k}(p,\gamma_p)\leq \frac{1}{m}$ 
and $d_{\rr^k}(q,\gamma_q) \leq \frac{1}{m}$.
By construction, 
\[d_{\rr^k}(H(p,y_0),h(\gamma_p,y_0)) \leq \frac{1}{m} + \eps,\] 
and thus,
\begin{align*}
 &|d_{\rr^k}(H(p,y_0),H(q,y_0)) - d_{\rr^k \times Y}((p,y_0),(q,y_0))| \\
 &\leq |d_{\rr^k}(H(p,y_0),H(q,y_0)) -
 d_{\rr^k}(h(\gamma_p,y_0),h(\gamma_q,y_0))| 
 \\&\quad + |d_{\rr^k}(h(\gamma_p,y_0),h(\gamma_q,y_0)) 
 - d_{\rr^k \times X}(g(\gamma_p,y_0),g(\gamma_q,y_0)) | 
 \\&\quad + |d_{\rr^k \times X}(g(\gamma_p,y_0),g(\gamma_q,y_0)) 
 - d_{\rr^k \times Y}((\gamma_p,y_0),(\gamma_q,y_0)) | 
 \\&\quad + | d_{\rr^k}(\gamma_p,\gamma_q) - d_{\rr^k}(p,q)| 
 \displaybreak[0] \\
 &\leq d_{\rr^k}(H(p,y_0),h(\gamma_p,y_0)) 
 + d_{\rr^k}(H(q,y_0),h(\gamma_q,y_0)) 
 \\&\quad + \big(
   d_{\rr^k}(\prrrk \circ\,g\,(\gamma_p,y_0),
   \prrrk \circ\,g\,(\gamma_q,y_0))^2 
   \\&\quad\quad\quad+ d_{X}(\pr_X \circ\,g\,(\gamma_p,y_0),
   \pr_X \circ\,g\,(\gamma_q,y_0))^2\big)^{1/2}
  \\&\quad\quad - d_{\rr^k}(\prrrk \circ\,g\,(\gamma_p,y_0),
  \prrrk \circ\,g\, (\gamma_q,y_0)) 
 \\&\quad + \eps  
 \\&\quad +  d_{\rr^k}(p,\gamma_p) + d_{\rr^k}(q,\gamma_q) 
 \displaybreak[0] \\
 &\leq 2 \cdot \Big(\frac{1}{m} + \eps \Big) 
 + d_{X}(\pr_X \circ\,g\,(\gamma_p,y_0),\pr_X \circ\,g\,(\gamma_q,y_0))
 + \eps + 2 \cdot \frac{1}{m} \\
 &\leq 5 \eps +d.
\end{align*}

\picturerrkxXH

Finally, verify 
$H(S_R \times \{y_0\}) \subseteq \bar{B}_R(0) \setminus B_{R-(d + 3\eps)}(0)$:
Let $p \in S_R$ be arbitrary and choose $\gamma_p \in \Gamma$ 
such that $d(p,\gamma_p) \leq \frac{1}{m}$. 
Then
\begin{align*}
 &d_{\rr^k}(h(\gamma_p,y_0), \prrrk(0,x_0))\\
 &= \sqrt{d_{\rr^k \times X}(g(\gamma_p,y_0), g(0,y_0))^2 
 - d_X(\pr_X \circ\,g\,(\gamma_p,y_0), \pr_X \circ\,g\,(0,y_0))^2 }  
 \displaybreak[0]\\
 &\geq \sqrt{(d_{\rr^k \times Y}((\gamma_p,y_0), (0,y_0))-\eps)^2 - d^2 }  \\
 &= \sqrt{(R-\eps)^2 - d^2}, 
\intertext{and hence,}
 &d_{\rr^k}(H(p,y_0), 0)\\
 &\geq d_{\rr^k}(h(\gamma_p,y_0), \prrrk(0,x_0)) 
 - d_{\rr^k}(H(p,y_0), h(\gamma_p,y_0))
 \displaybreak[0]\\
 &\geq \sqrt{(R-\eps)^2 - d^2 } - \Big(\frac{2}{m} + \eps\Big) \\
 &\geq 
 R - (d+3\eps).
\end{align*}

Then the image of $H$ intersects each segment $\sigma$ 
from the origin to a point in $\partial \bar{B}_R(0)$: 
By \autoref{cont_eps_isometry_sphere_to_sphere_is_surjective}, 
the segment $\sigma$ intersects $\partial \bar{B}_R(0)$ 
in a point contained in $\pr \circ H(S_R \times \{y_0\})$
where $\pr$ is the radial projection to the sphere defined 
as in \autoref{cont_eps_isometry_sphere_to_sphere_is_surjective}. 
Since the projection is radial, 
the segment $\sigma$ intersects $H(S_R \times \{y_0\})$ as well.

\picturerrkxXdistS

Let $p$ be this intersecting point for the segment 
through $\prrrk{(g(0,y_n))}$, 
cf.~\autoref{pic:rrkxX:dGH-small->diamX_similar--dist_to_pr(S)}. 
Since $h(\Gamma \times \{y_0\})$ is a $(\frac{2}{m} + \eps)$-net in $H(S_R)$, 
there exists a point $q \in h(\Gamma)$ 
such that $d_{\rr^k}(p,q) \leq \frac{2}{m} + \eps$.
Thus, using 
that the segment from $\prrrk(g(0,y_n))$ to $p$ is part of a segment 
connecting the origin and the boundary of the $R$-ball
and 
$\prrrk(S) = \prrrk \circ g (S_R \times \{y_0\}) \supseteq h(\Gamma) \ni q$, 
\begin{align*}
 d_{\rr^k}(\prrrk(g(0,y_n)),\prrrk(S)) 
 &\leq d_{\rr^k}(\prrrk(g(0,y_n)),q) \\
 &\leq d_{\rr^k}(\prrrk(g(0,y_n)),p) + d_{\rr^k}(p,q) \\
 &\leq R + \Big(\frac{2}{m} + \eps\Big).
\end{align*}
Now $m \to \infty$ proves 
\[
 \sqrt{(\sqrt{R^2 + \delta_n^2}-\eps)^2 - d^2} 
 \leq d_{\rr^k}(\prrrk(g(0,y_n)),\prrrk(S)) 
 \leq R + \eps,
\]
and thus, 
\begin{align*}
 \delta_n 
 &\leq \sqrt{(\sqrt{(R + \eps)^2 + d^2} + \eps)^2-R^2} \\
 &= \sqrt{2R\eps + d^2 + 2\eps^2 + 2\sqrt{(R\eps+\eps^2)^2 + (\eps d)^2}}\\
 &\leq \sqrt{d^2 + 2 \eps d + 4\eps(R+\eps)}.
\end{align*}
Since this is true for all $n$ and $\delta_n \to \delta$ as $n \to \infty$, 
this proves the claim.

\par\smallskip\noindent b)
Let $\eps := \frac{R}{6}$. Then $\diam(X) + 4 \eps = \frac{2R}{3}$,
and by \ref{lemma:rrkxX:dGH-small->diamX_similar--preparation},
\begin{align*}
\min\{R,\rad_Y(y_0)\}^2 
&\leq 4 \eps (R + \eps)
= \frac{24}{25} \cdot R^2 < R^2.
\end{align*}
Thus, $\rad_Y(y_0) < R$,
and this implies $\diam(Y) \leq 2 \cdot \rad_Y(y_0) < 2R$.

\par\smallskip\noindent c) 
Let $\eps := \frac{1}{50R}$. Then
$\diam(X) + 4 \eps \leq 1 + \frac{2}{25} < R \cdot \frac{\pi}{3}$
and \ref{lemma:rrkxX:dGH-small->diamX_similar--preparation} can be applied.
Since $R \eps = \frac{1}{50} = 2 \cdot 10^{-2}$ 
and $\eps \leq 2 \cdot 10^{-4}$,
\begin{align*}
 \min\{R,\rad_Y(y_0)\}^2 
 &\leq \diam(X)^2 + 2 \eps \diam(X) + 4\eps(R+\eps)\\
 &\leq \diam(X)^2 + 4 \cdot 10^{-4} \cdot \diam(X) + 8 \cdot 10^{-2} 
   + 4 \cdot 10^{-8}.
\end{align*}
 Using $\diam(X) = 1$,
 \begin{align*}
 \min\{R,\rad_Y(y_0)\}^2 
 \leq 1 + 2 \cdot 10^{-2} + 10^{-4} + 10^{-8}
 < 1.05^2
 < R^2. 
 \end{align*}
 In particular, 
 $\diam(Y) \leq 2 \cdot \rad_Y(y_0) < 2 \cdot 1.05 = \frac{21}{10}$.
 
 On the other hand, by permuting $X$ and $Y$,
 \begin{align*}
  \frac{1}{4}
  = \Big(\frac{\diam(X)}{2}\Big)^2 
  &\leq \rad_X(x_0)^2 \\
  &= \min\{ R, \rad_X(x_0)\}^2 \\
  &\leq \diam(Y)^2 + 4 \cdot 10^{-4} \cdot \diam(Y) 
    + 8 \cdot 10^{-2} + 16 \cdot 10^{-8},
 \end{align*}
 and this implies $\diam(Y) \geq \frac{21}{50} =: c$. 

\par\smallskip\noindent d)
Assume $\dgh{R}{\rr^k \times X}{(0,x_0)}{\rr^k \times Y}{(0,y_0)} < 20 d$ 
for $d := \diam(X)$.
Let $\eps := 40 d$. By choice of $R$,
$\diam(X) + 4 \eps 
= 161 \cdot d \leq \frac{161}{1000} \cdot R 
< \frac{2R}{3}$.
Furthermore,
$2 \rad_Y(y_0) \geq \diam(Y) > 2R$.
By \ref{lemma:rrkxX:dGH-small->diamX_similar--preparation}, 
\begin{align*}
 R^2 &= \min\{R,\rad_Y(y_0)\}^2\\
 &\leq d^2 + 2 \eps d + 4 \eps (R + \eps)\\
 &\leq \frac{R^2}{10^6} + \frac{2R^2}{25 \cdot 10^3} 
   + \frac{4R}{25} \cdot \frac{26R}{25}
 = \frac{166481}{10^6} \cdot R^2
 < R^2.
\end{align*}
This is a contradiction.
\end{proof}

Using these results, 
the main proposition of this section finally can be proven.
\begin{mainprop*}
Let $(M_i)_{i \in \nn}$ be a sequence 
 of complete connected \ndim Riemannian manifolds
 with uniform lower Ricci curvature bound $\Ric_{M_i} \geq -(n-1)$, 
 and let $k < n$. 
 Given $\mydelta \in (0,1)$, 
 there is $\tildeeps =\tildeeps(\mydelta;n,k) > 0$ 
 such that for any $0 < r \leq \tildeeps$ 
 and $q_i \in M_i$ with \[\dghpt{\tildeeps}{r^{-1} M_i}{q_i}{\rr^k}{0}\]
 there are 
 a family of subsets of good points $G_{r}(q_i) \subseteq B_r(q_i)$ 
 with \[\vol(G_r(q_i)) \geq (1 - \mydelta) \cdot \vol( B_r(q_i))\] and
 a sequence $\lambda_i \to \infty$ 
 such that the following holds:
 \begin{enumerate}
 \item 
  For every choice of base points $x_i \in G_r(q_i)$
  and every sublimit $(Y,\cdot)$ of $(\lambda_i M_i, x_i)$ 
  there exists a compact metric space $K$ of dimension $l \leq n-k$
  satisfying $\frac{1}{5} \leq \diam(K) \leq 1$ 
  such that $Y$ splits isometrically as a product
  \[Y \cong \rr^k \times K.\]
 \item 
  If $x_i^1, x_i^2 \in G_r(q_i)$ are base points 
  such that, after passing to a subsequence,
  \[(\lambda_i M_i, x_i^j) \to (\rr^k \times K_j, \cdot)\] 
  for $1 \leq j \leq 2$ as before,
  then $\dim(K_1) = \dim(K_2)$.
 \end{enumerate}
\end{mainprop*}

\begin{proof}
 Given $\mydelta \in (0,1)$, 
 let 
 \begin{align*}
 \tildeeps_1 & = \tildeeps_1(\mydelta;n,k)>0 
 \text{ be the } \tildeeps_1\Big(\frac{\mydelta}{2};n,k\Big) 
 \text{ of \autoref{lem:set_where_all_blow_ups_split__locally},}\\
 \tildeeps_2 &= \tildeeps_2(\mydelta;n,k)>0 
 \text{ be the } \tildeeps_2\Big(\frac{\mydelta}{2};n,k\Big) 
 \text{ of \autoref{nearly_splitting_after_rescaling_locally},}
 \intertext{and define} 
 \tildeeps & = \tildeeps(\mydelta;n,k) 
 := \frac{1}{16} \cdot \min\{\tildeeps_1,\tildeeps_2\} > 0.
 \end{align*}
 Furthermore, let $\eps_0 = \eps_0(n,k) \in (0,\frac{1}{100})$ be 
 as in \autoref{lem:spaces_of_different_dimension_are_not_close}.
 Let $0 < r \leq \tildeeps$ and $q_i \in M_i$ with
 \[\dghpt{\tildeeps}{r^{-1} M_i}{q_i}{\rr^k}{0}.\]
 Let $G^1_{r}(q_i) \subseteq B_{r}(q_i)$ be 
 as in \autoref{lem:set_where_all_blow_ups_split__locally}
 and $G^2_{r}(q_i) \subseteq B_{r}(q_i)$ and $\lambda_i \to \infty$ 
 as in \autoref{nearly_splitting_after_rescaling_locally}.
 Define 
 \[G_{r}(q_i) := G^1_{r}(q_i) \cap G^2_{r}(q_i) \subseteq B_{r}(q_i).\]
 Clearly,
 $\vol(G_{r}(q_i)) \geq (1 - \mydelta) \cdot \vol(B_{r}(q_i))$.
  
 Let $x_i \in G_{r}(q_i)$ and $(Y,y)$ be a sublimit of $(\lambda_i M_i, x_i)$.
 Using $x_i \in G^1_{r}(q_i)$,
 there are a metric space $Y'$ and $y' \in \{0\} \times Y'$ such that
 $(Y,y) \cong (\rr^k \times Y',y')$. 
 On the other hand, since $x_i \in G^2_{r}(q_i)$, 
 \[
  \dgh{1/\eps_0}{\lambda_i M_i}{x_i}{\rr^k \times K_i}{\tilde{x}_i} 
  \leq \frac{\eps_0}{200}
 \] 
 for some compact metric space $K_i$ with diameter $1$ 
 and $\tilde{x}_i \in \{0\} \times K_i$.
 Hence, by the triangle inequality and for $i$ large enough,
 \[
  \dgh{1/\eps_0}{\rr^k \times Y'}{y'}{\rr^k \times K_i}{\tilde{x}_i} 
  < \frac{\eps_0}{100}.
 \] 
 By \autoref{lemma:rrkxX:dGH-small->diamX_similar} 
 \ref{lemma:rrkxX:dGH-small->diamX_similar--diam=1},
 there exists a constant $c > 0$ such that 
 $Y'$ is compact as well with $c \leq \diam(Y') \leq 5c$, 
 and after rescaling with $\frac{1}{5c}$ 
 this finishes the first part of the claim.
 
 So let $x_i,y_i \in G_r(q_i)$ and
 $K_1$ and $K_2$ be compact metric spaces such that, 
 after passing to a subsequence,
 \[
  (\lambda_i M_i, x_i) \to (\rr^k \times K_1, x_\infty) 
  \quad\aand\quad 
  (\lambda_i M_i, y_i) \to (\rr^k \times K_2, y_\infty).
 \]
 
 Because of $x_i,y_i \in G_r^2(q_i)$,
 there is a time-dependent, piecewise constant in time vector field $X_i$ 
 with compact support and an integral curve $c_i$ such that 
 the vector field $X_i^t$ is divergence free on $B_{10r}(c_i(t))$ 
 for all $0 \leq t \leq 1$, 
 \[
  d(x_i,c_i(0))< \frac{c(n)}{\lambda_i} 
  \quad\aand\quad 
  d(y_i,c_i(1)) < \frac{c(n)}{\lambda_i}
 \] 
 for $c(n) = 9^n \cdot 10^{n^2}$ and 
 \[
  \int_0^1 (\Mx_{2r}(\norm{\nabla.X_i^t}^{3/2}) (c_i(t)) )^{2/3} \dt 
  < \frac{\mydelta}{2}.
 \]
 Let $x_i' := c_i(0)$ and $y_i' := c_i(1)$.
 Because of $d_{\lambda_i M_i}(x_i,x_i')\leq c(n)$ 
 and $d_{\lambda_i M_i}(y_i, y_i') \leq c(n)$,
 there exists $x_\infty' \in \rr^k \times K_1$ 
 such that, after passing to a subsequence, 
 \[(\lambda_i M_i, x_i') \to (\rr^k \times K_1, x_\infty').\]
 After passing to a further subsequence, 
 \[(\lambda_i M_i, y_i') \to (\rr^k \times K_2, y_\infty')\] 
 for some $y_\infty' \in \rr^k \times K_2$.
 Then \autoref{prop:bilipschitz_implies_dimension} implies 
 \[
  \dim(K_1) 
  = \dim(\rr^k \times K_1) - k 
  = \dim(\rr^k \times K_2) - k 
  = \dim(K_2).
  \qedhere
 \]
\end{proof}


\section{Global construction}\label{sec:step2}

Based on the \myquote{local} version (\autoref{prop:main}) established in the last section,
the proof of the following main result can now be given.

\begin{mainthm}\label{thm:main}
 Let $(M_i,p_i)_{i \in \nn}$ be a collapsing sequence 
 of pointed complete connected \ndim Riemannian manifolds
 which satisfy the uniform lower Ricci curvature bound $\Ric_{M_i} \geq -(n-1)$ 
 and converge to a limit $(X,p)$ of dimension $k < n$ 
 in the measured \GH sense. 
 Then for every $\eps \in (0,1)$ there exist a subset of good points
 $G_1(p_i) \subseteq B_1(p_i)$ 
 satisfying \[\vol(G_1(p_i)) \geq (1 -\eps) \cdot \vol(B_1(p_i)),\]
 a sequence $\lambda_i \to \infty$ and
 a constant $D > 0$
 such that the following holds: 
 
 For any choice of base points $q_i \in G_1(p_i)$ 
 and every sublimit $(Y,q)$ of $(\lambda_i M_i, q_i)_{i \in \nn}$ 
 there is a compact metric space $K$ of dimension $l \leq n-k$ 
 and diameter $\frac{1}{D} \leq \diam(K) \leq D$ 
 such that $Y$ splits isometrically as a product
 \[Y \cong \rr^k \times K.\] 
 Moreover, for any base points $q_i, q_i' \in G_1(p_i)$ 
 such that, after passing to a subsequence, both
 $(\lambda_i M_i, q_i) \to (\rr^k \times K, \cdot)$ 
 and $(\lambda_i M_i, q_i') \to (\rr^k \times K', \cdot)$ 
 as before,
 $\dim(K) = \dim(K')$.
\end{mainthm}

The idea of the proof is to take (finitely many) sequences $(q_i)_{i \in \nn}$ 
satisfying the hypothesis of \autoref{prop:main} for some $r>0$
and to define $G_1(p_i)$ as the union of the $G_r(q_i)$ obtained from \autoref{prop:main}.
Instantly, the following question occurs:
\begin{enumerate}
 \item[(1)] 
 Why do sequences $(q_i)_{i \in \nn}$ satisfying the hypothesis of \autoref{prop:main} exist?
\end{enumerate}
It will turn out that sequences $p_i^x \to x$, where $x \in X$ is a generic point, 
are candidates for these $(q_i)_{i \in \nn}$:
If $x \in X$ is generic, 
then $(\frac{1}{r}X,x)$ is close to $(\rr^k,0)$ for sufficiently small $r > 0$
and so is $(\frac{1}{r} M_i, p_i^x)$ for sufficiently large $i \in \nn$.
In fact, decreasing $r$ only improves the situation.

Now assume that $x,x'$ are two such generic points, let $r>0$ be small enough 
and $\lambda_i \to \infty$ and $\lambda_i' \to \infty$, respectively,
be the sequences given by \autoref{prop:main}.
These sequences might be different, 
but \autoref{thm:main} calls for one single rescaling sequence.
This gives rise to the following question:
\begin{enumerate}
 \item[(2)] Does \autoref{prop:main} still hold for $(p_i^{x'})_{i \in \nn}$ 
 if $\lambda_i' \to \infty$ is replaced by $\lambda_i \to \infty$?
\end{enumerate}
In order to answer this question, 
first consider the special case $\lambda_i = 2 \lambda_i'$: 
Obviously, if $q_i \in G_r(p_i^{x'})$ 
and $(\rr^k \times K,\cdot)$ is a sublimit of $(\lambda_i' M_i, q_i)$,
then $(\rr^k \times 2K, \cdot)$ is a sublimit 
of $(\lambda_i M_i, q_i) = (2\lambda_i' M_i,q_i)$.
Conversely, every sublimit of $(\lambda_i M_i,q_i)$ 
has the form $(\rr^k \times 2K, \cdot)$ 
for a sublimit $(\rr^k \times K, \cdot)$ of $(\lambda_i'M_i,q_i)$.
It turns out that such a correspondence holds 
whenever the sequence $\big(\frac{\lambda_i'}{\lambda_i}\big)_{i \in \nn}$ is bounded. 
In this way, $\lambda_i'$ indeed can be replaced by $\lambda_i$ 
if one allows weaker diameter bounds for the compact factors of the sublimits.
Therefore, the question (2) can be reformulated in the following way:
\begin{enumerate}
 \item[(2')] Under which condition is the quotient 
 $\big(\frac{\lambda_i'}{\lambda_i}\big)_{i \in \nn}$ 
 of two such rescaling sequences bounded?
\end{enumerate}
In fact, one can prove the following: 
If the subsets $G_r(p_i^x)$ and $G_r(p_i^{x'})$ have non-empty intersection, 
then $\big(\frac{\lambda_i'}{\lambda_i}\big)_{i \in \nn}$ is bounded.
An obvious approach for comparing points where these subsets do not intersect is
to connect the points by a curve consisting of generic points only
and
to cover this curve by balls $B_{r_j}(y_j)$ 
such that for every two subsequent points $y_j$ and $y_{j+1}$ 
(and sufficiently large $i \in \nn$)
the subsets intersect.
If this can be done using only finitely many $y_j$,
an inductive argument proves the boundedness of 
$\big(\frac{\lambda_i'}{\lambda_i}\big)_{i \in \nn}$.
Usually, such covers are constructed by using a compactness argument:
Let $r^y$ denote the minimal radius 
such that all $r \leq r^y$ and $p_i^y$ 
satisfy the hypothesis of \autoref{prop:main}.
If this $r^y$ is continuous in $y$, 
there exists $r>0$ that can be used at each point of the (compact image of the) curve.
Unfortunately, there is no reason for this $r^y$ to depend continuously on $y$. 
It will turn out that a similar approach to compare $\lambda_i$ and $\lambda_i'$ 
can be performed
if $x$ and $x'$ lie in the interior of a minimising geodesic 
such that all the points of this geodesic lying between $x$ and $x'$ are generic.
The H\"older continuity result of Colding and Naber \cite[Theorem 1.2]{colding-naber} 
then allows a cover similar to the one described above.
The subsequent question
\begin{enumerate}
 \item[(3)] 
 Does there exists a minimising geodesic such that $x,x'$ lie in its interior?
\end{enumerate}
can be answered affirmatively for a set of full measure (in $X \times X$) 
by applying further results of \cite{colding-naber}.

This section is subdivided into several subsections answering the above questions:
First, \autoref{sec:application_to_generic_points} investigates generic points $x \in X$ 
and answers question (1) by applying \autoref{prop:main} to the sequence $p_i^x \to x$. 
Both questions (2) and (2') are dealt with in \autoref{sec:comparison_rescaling_sequences},
which discusses the comparison of the different $\lambda_i$.
Afterwards, \autoref{sec:geodesics} treats question (3) 
by proving that the necessary conditions for performing the comparison 
are given on a set of full measure.
Finally, \autoref{sec:proof_main_thm} deals with the proof of \autoref{thm:main}.

\subsection{Application to generic points}\label{sec:application_to_generic_points}
A very important property of generic points is that, after rescaling,
the manifolds with base points converging to a generic point 
are in some sense close to the Euclidean space.

\begin{lemma}\label{lemma:rescaling_close_to_tangent_cone}
 Let $(X,p)$ be the limit of a collapsing sequence 
 of pointed complete connected \ndim Riemannian manifolds 
 which satisfy the uniform lower Ricci curvature bound $-(n-1)$, 
 $k = \dim(X) < n$,
 $x \in \Xgen$ and $p_i^x \to x$. 
 For fixed $R>0$ and $\eps >0$ 
 there exists $\lambda_0=\lambda_0(x,R,\eps)$ 
 such that for all $\lambda \geq \lambda_0$,
 \[\dgh{R}{\lambda X}{x}{\rr^k}{0} \leq \eps.\]
\end{lemma}

\begin{proof}
 Since $x$ is a generic point, all tangent cones at $x$ equal $\rr^k$,
 i.e.~for every choice of $\lambda \to \infty$, 
 the sequence $(\lambda X,x)$ converges to $(\rr^k,0)$.
 In particular, the $R$-balls converge 
 and for sufficiently large $\lambda$ 
 the distance of these balls is bounded from above by $\eps$.
 This proves that there exists
 \begin{align*}
  \lambda_0(x) 
       :=& \min\{\lambda \geq 1 
       \mid \forall \mu \geq \lambda: \dgh{R}{\mu X}{x}{\rr^k}{0} \leq \eps\}
       < \infty.
 \qedhere
 \end{align*}
\end{proof}

\begin{not*}
From now on, for given $k < n$ and $\mydelta \in (0,1)$,
let $\tildeeps = \tildeeps(\mydelta;n,k)$ be as in \autoref{prop:main}.
For $r>0$,
define 
\[
 \mathcal{X}_r(\mydelta;n,k) 
 := \Big\{x \in \Xgen 
 \mid \dgh{1/\tildeeps}{r^{-1} X}{x}{\rr^k}{0} \leq \frac{\tildeeps}{2} \Big\}.
\]
\end{not*}

\begin{lemma}\label{prop:main--generic_points}
 Let $(M_i,p_i)_{i \in \nn}$ be a collapsing sequence 
 of pointed complete connected \ndim Riemannian manifolds
 which satisfy the uniform lower Ricci curvature bound $\Ric_{M_i} \geq -(n-1)$ 
 and converge to a limit $(X,p)$ of dimension $k < n$ 
 and let $\mydelta \in (0,1)$.
 \begin{enumerate}
 \item 
  For every $x \in \Xgen$ there is $0 < r^x = r(\mydelta,x;n,k) \leq \tildeeps $ 
  such that $x \in \mathcal{X}_r(\mydelta;n,k)$ for any $0 < r \leq r^x$. 
 \item 
  For $0<r\leq\tildeeps$, $x \in \mathcal{X}_r(\mydelta;n,k)$ and $p_i^x \to x$
  there is $i_0 \in \nn$ such that for $i \geq i_0$ there are
  a subset of good points $G_{r}(p_i^x) \subseteq B_{r}(p_i^x)$ with 
  \[\vol(G_{r}(p_i^x)) \geq (1 - \mydelta) \cdot \vol( B_{r}(p_i^x))\]
  and a sequence $\lambda_i \to \infty$
  satisfying the following:
  \begin{enumerate}
  \item 
   For any choice of base points $x_i \in G_r(p_i^x)$ 
   and all sublimits $(Y,\cdot)$ of $(\lambda_i M_i, x_i)$ 
   there exists a compact metric space $K$ 
   of dimension $l \leq n-k$
   and diameter $\frac{1}{5} \leq \diam(K) \leq 1$ 
   such that $Y$ splits isometrically as a product
   \[Y \cong \rr^k \times K.\]
  \item 
   If $x_i^1, x_i^2 \in G_r(p_i^x)$ are base points such that, after passing to a subsequence,
   \[(\lambda_i M_i, x_i^j) \to (\rr^k \times K_j, \cdot)\] for $1 \leq j \leq 2$ as before,
   then $\dim(K_1) = \dim(K_2)$.
  \end{enumerate}
  Moreover, if $\omega$ is a fixed ultrafilter on $\nn$,
  there exists $l \in \nn$
  such that the following holds:
  Given $q_i \in G_{r}(p_i^x)$, the ultralimit of $(\lambda_i M_i, q_i)$ 
  is a product $\rr^k \times K$ such that $K$ is compact with
  \[\frac{1}{5} \leq \diam(K) \leq 1  \quad\aand\quad  \dim(K) = l.\]
 \end{enumerate}
\end{lemma}

\begin{proof}
 a)
  Let $\lambda_0 = \lambda_0(\mydelta;n,k)$ 
  be the $\lambda_0(x,\frac{1}{\tildeeps},\frac{\tildeeps}{2})$ 
  appearing in \autoref{lemma:rescaling_close_to_tangent_cone}. 
  Then 
  $r^x := r(\mydelta,x;n,k) := \min\big\{\tildeeps, \frac{1}{\lambda_0}\big\} > 0$ 
  proves the claim.
 
 \par\smallskip\noindent b)
  Let $x \in \mathcal{X}_r(\mydelta;n,k)$ be arbitrary,
  i.e.~$\dgh{1/\tildeeps}{r^{-1} X}{x}{\rr^k}{0} \leq \frac{\tildeeps}{2}$.
  Since $(\frac{1}{r}M_i,p_i^x) \to (\frac{1}{r} X, x)$,
  there is $i_0 \in \nn$ with 
  $\dgh{1/\tildeeps}{r^{-1} M_i}{p_i^x}{r^{-1} X}{x} \leq \frac{\tildeeps}{2}$
  for all $i \geq i_0$.
  In particular, 
  $\dgh{1/\tildeeps}{r^{-1} M_i}{p_i^x}{\rr^k}{0} \leq \tildeeps$ by triangle in\-equality
  and \autoref{prop:main} implies the claim.
\end{proof}

\begin{not*}
 For $0 < r \leq \tildeeps$ and $x \in \mathcal{X}_r(\mydelta;n,k)$,
 let $\lambda_i^{\mydelta,x}(r)$ and $G_r^\mydelta(p_i^x)$ 
 be as in \autoref{prop:main--generic_points},
 i.e.~for $q_i \in G_r^\mydelta(p_i^x)$ the sublimits 
 of $(\lambda_i^{\mydelta,x}(r)\,M_i, q_i)$ are isometric to products $(\rr^k \times K, \cdot)$ 
 where the $K$ are compact metric spaces with $\diam(K) \in \big[\frac{1}{5},1\big]$.
 Moreover, for $x \in \Xgen$, 
 let $r^x(\mydelta;n,k)$ be as in \autoref{prop:main--generic_points},
 i.e.~$x \in \mathcal{X}_r(\mydelta;n,k)$ for all $0 < r \leq r^x(\mydelta;n,k)$.

 Furthermore, for a non-principal ultrafilter $\omega$ on $\nn$, 
 let $l_\omega^{\mydelta,x}(r)$ be as in \autoref{prop:main--generic_points}, 
 i.e.~for $q_i \in G_r^\mydelta(p_i^x)$, 
 $\limomega (\lambda_i^{\mydelta,x}(r)\,M_i, q_i) = (\rr^k \times K, \cdot)$ 
 and $\dim(K) = l_\omega^{\mydelta,x}(r)$ for some $K$ as above.

 All notations will be used throughout the remaining section without referring to 
 \autoref{prop:main--generic_points} explicitly.
 Occasionally, if they are fixed, the dependences on $n$, $k$ and $\mydelta$ will be suppressed.
\end{not*}

\subsection{Comparison of rescaling sequences}\label{sec:comparison_rescaling_sequences}
Throughout this subsection, let $k < n$ and $\mydelta \in (0,\frac{1}{2})$ be fixed 
and use the notation introduced in \autoref{sec:application_to_generic_points}. 

The following lemma states that 
for each two rescaling sequences corresponding to limit spaces which are 
products of the same Euclidean and a compact set, 
the quotient of these rescaling sequences is bounded.
Especially, this holds in the situation of \autoref{prop:main--generic_points}.

\begin{lemma}\label{prop-A}
 Let $(M_i,p_i)_{i \in \nn}$ be a collapsing sequence 
 of pointed complete connected \ndim Riemannian manifolds
 which satisfy the uniform lower Ricci curvature bound $\Ric_{M_i} \geq -(n-1)$ 
 and converge to a limit $(X,p)$ of dimension $k < n$. 
 \begin{enumerate}
 \item\label{prop-A--tilde(G)_i->quotients_are_bounded}
  Let $q_i \in M_i$ satisfy that 
  for any $\lambda_i \to \infty$ and 
  every sublimit $(Y,\cdot)$ of $(\lambda_i M_i, q_i)$
  there exists a metric space $Y'$ 
  such that $Y \cong \rr^k \times Y'$ isometrically. 
  For $1 \leq j \leq 2$,  
  let $\lambda_i^j \to \infty$ and $K_j$ be compact 
  with $(\lambda_i^j M_i, q_i) \to (\rr^k \times K_j, \cdot)$ 
  as $i \to \infty$.
  Then the sequence $\big(\frac{\lambda_i^2}{\lambda_i^1}\big)_{i \in \nn}$ is bounded.
 \item 
  For $1 \leq j \leq 2$, 
  let $r_j > 0$, $x_j \in \mathcal{X}_{r_j}$, $p_i^j := p_i^{x_j}$ 
  and $\lambda_i^j := \lambda_i^{\mydelta,x_j}(r_j)$. 
  Moreover, assume $G_{r_1}(p_i^1) \cap G_{r_2}(p_i^2) \ne \emptyset$.
  Then the sequence $\big(\frac{\lambda_i^2}{\lambda_i^1}\big)_{i \in \nn}$ is bounded.
 \end{enumerate}
\end{lemma}

\begin{proof}
 a)
  The proof is done by contradiction: 
  Without loss of generality, 
  assume $\lambda_i^j > 0$ for $1 \leq j \leq 2$ and all $i \in \nn$.
  Obviously, the sequence $\big(\frac{\lambda_i^2}{\lambda_i^1}\big)_{i \in \nn}$ 
  is bounded from below by $0$. 
  Assume the sequence is not bounded from above, 
  i.e.~$\frac{\lambda_i^2}{\lambda_i^1} \to \infty$ and, 
  without loss of generality, 
  $\lambda_i^1 < \lambda_i^2$ for all $i \in \nn$.
  
  There exists $\lambda_i \to \infty$ satisfying $\lambda_i^1 < \lambda_i < \lambda_i^2$ 
  such that
  \[(\lambda_i M_i, q_i) \to (\rr^k \times N, y)\]
  for some unbounded metric space $N$: 
  Let $\mu_i \to \infty$ be as in \autoref{lem:rescaling_converging_to_tangent_cone} 
  such that $(\mu_i' \lambda_i^1 M_i, q_i)$ subconverges 
  to a tangent cone of $(\rr^k \times K_1, x_1)$ 
  for any $\mu_i' \to \infty$ with $\mu_i' \leq \mu_i$.
  Define
  \[
   \mu_i' := \min \bigg\{ \mu_i, \sqrt{\frac{\lambda_i^2}{\lambda_i^1}} \bigg\} \to \infty 
   \quad\aand\quad 
   \lambda_i := \mu_i' \cdot \lambda_i^1 \to \infty.
  \] 
  Without loss of generality, assume $\mu_i > 1$. 
  Thus, $\lambda_i^1 < \lambda_i < \lambda_i^2$.  
  Furthermore, 
  there exist a metric space $N$ and $q \in \rr^k \times N$ 
  with $(\lambda_i M_i,q_i) \to (\rr^k \times N, q)$ 
  and this is a tangent cone of $(\rr^k \times K_1, x_1)$.
  Hence, for a sequence $\alpha_i \to \infty$,
  \[(\rr^k \times \alpha_i K_1,x_1) \to (\rr^k \times N,q).\]
  Assume this $N$ is compact 
  and let $N' := \frac{1}{\diam(N)} \cdot N$ 
  and $\beta_i := \frac{1}{\diam(N)} \cdot \alpha_i$.
  Then
  \[(\rr^k \times \beta_i K_1,x_1) \to (\rr^k \times N',q),\]
  and, for sufficiently large $i$,
  \[\dgh{100}{\rr^k \times \beta_i K_1}{x_1}{\rr^k \times N'}{q} \leq 10^{-4}.\] 
  By \autoref{lemma:rrkxX:dGH-small->diamX_similar} 
  \ref{lemma:rrkxX:dGH-small->diamX_similar--diam=1}, 
  the sequence $(\diam(\beta_i K_1))_{i \in \nn}$ is bounded.
  This is a contradiction to $\beta_i \to \infty$.
  Hence, $N$ is unbounded.
  
  Now let $D := \diam(K_2)$ 
  and fix $R > 1000 D$. 
  As $N$ is unbounded, $2R < \diam(N)$ and 
  $\dgh{R}{\rr^k \times N}{q}{\rr^k \times K_2}{x_2} \geq 20 D$
  by \autoref{lemma:rrkxX:dGH-small->diamX_similar} 
  \ref{lemma:rrkxX:dGH-small->diamX_similar--diam=>dgh}.
  Thus, for $i$ large enough,
  \begin{align*}
   \dgh{R}{\lambda_i M_i}{q_i}{\lambda_i^2 M_i}{q_i} 
   &\geq 
   \dgh{R}{\rr^k \times N}{q}{\rr^k \times K_2}{x_2} 
   \\&\quad - \dgh{R}{\lambda_i M_i}{q_i}{\rr^k \times N}{q}
   \\&\quad - \dgh{R}{\lambda_i^2 M_i}{q_i}{\rr^k \times K_2}{x_2} 
   \\&\geq
   10 D.
  \end{align*}
  Since the maps $h_i: (0,\infty) \to (0,\infty)$ defined by 
  \[h_i(\mu_i) := \dgh{R}{\mu_iM_i}{q_i}{\lambda_i^2 M_i}{q_i}\] 
  are continuous 
  with $h_i(\lambda_i^2) = 0$ and $h_i(\lambda_i) \geq 10 D$,
  by the intermediate value theorem, 
  there is a maximal $\lambda_i^1 < \lambda_i' \leq \lambda_i < \lambda_i^2$ 
  such that $h_i(\lambda_i') = 5D$.
  After passing to a subsequence,
  \[(\lambda_{i}' M_i, q_{i}) \to (\rr^k \times Y, y)\]
  for some metric space $Y$.
  In particular,
  \[\dgh{R}{\lambda_i' M_i}{q_i}{\lambda_i^2 M_i}{q_i}\\
   \to \dgh{R}{\rr^k \times Y}{y}{\rr^k \times K_2}{x_2} \]
   as $i \to \infty$.
   Hence, $\dgh{R}{\rr^k \times Y}{y}{\rr^k \times K_2}{x_2} = 5D$.
   Furthermore,
   \[\dgh{R}{\rr^k \times K_2}{x_2}{\rr^k}{0} \leq \diam(K_2) = D\]
   and the triangle inequality implies
  \begin{align*}
   4D
   \leq \dgh{R}{\rr^k \times Y}{y}{\rr^k}{0}
   \leq 6D < \frac{R}{12}.
  \end{align*}
  By \autoref{lemma:rrkxX:dGH-small->diamX_similar} 
  \ref{lemma:rrkxX:dGH-small->diamX_similar--diam=0},
  $Y$ is compact. Moreover, \[\diam(Y) \geq \dgh{R}{\rr^k \times Y}{y}{\rr^k}{0} \geq 4D > D,\] 
  in particular, $Y$ is not a point.
   
  Next, prove $\frac{\lambda_i^2}{\lambda_i'} \to \infty$: 
  Assume the quotient is bounded. 
  Hence, after passing to a subsequence, 
  $\frac{\lambda_i^2}{\lambda_i'} \to \alpha$ and $\alpha \geq 1$ 
  due to $\lambda_i' \leq \lambda_i^2$.
  Then
  \begin{align*}
   (\lambda_i^2 M_i, q_i) 
   = \Big(\frac{\lambda_i^2}{\lambda_i'} \cdot \lambda_i' M_i, q_i \Big) 
   \to (\rr^k \times \alpha Y, y) 
   \cong (\rr^k \times K_2, x_2).
  \end{align*}
  In particular, $\diam(Y) \leq \alpha \cdot \diam(Y) = \diam(K_2) = D$, 
  and this is a contradiction.

  Thus, $\frac{\lambda_i^2}{\lambda_i'} \to \infty$. 
  Analogously to the previous argumentation, 
  there exists some maximal $\lambda_i' < \tilde{\lambda}_i < \lambda_i^2$ 
  such that $h_i(\tilde{\lambda}_i) = 5D$.
  This is a contradiction to the maximal choice of $\lambda_i'$.
 
 \par\smallskip\noindent b)
  Let $q_i \in G_{r_1}(p_i^1) \cap G_{r_2}(p_i^2)$ 
  and $\alpha_i := \frac{\lambda_i^2}{\lambda_i^1}$.
  Assume $\alpha_i \to \infty$ and choose a subsequence $(i_j)_{j \in \nn}$ 
  such that $\alpha_{i_j} > j$ for all $j \in \nn$.
  After passing to a further subsequence,
  there exist compact metric spaces $K_m$, where $1 \leq m \leq 2$,
  such that $(\lambda_{i_{j}}^m M_{i_{j_l}}, q_{i_{j}}) \to (\rr^k \times K_m, \cdot)$.
  By \ref{prop-A--tilde(G)_i->quotients_are_bounded}, the sequence
  $(\alpha_{i_j})_{j \in \nn}$ is bounded. 
  This is a contradiction. 
\end{proof}

The following lemma gives a statement about the limit of such a bounded sequence of quotients.

\begin{lemma}\label{lem:cor_of_prop-E}
 Let $(M_i,p_i)_{i \in \nn}$ be a collapsing sequence 
 of pointed complete connected \ndim Riemannian manifolds
 satisfying the uniform lower Ricci curvature bound $\Ric_{M_i} \geq -(n-1)$
 and let $\lambda_i, \mu_i > 0$ such that 
 $\big(\frac{\lambda_i}{\mu_i}\big)_{i \in \nn}$ is bounded. 
 If 
 \[
  (\lambda_i M_i, p_i) \to (\rr^k \times L, p_L) 
  \quad\aand\quad 
  (\mu_i M_i, p_i) \to (\rr^k \times M, p_M)
 \] 
 for some bounded metric spaces $L$ and $M$, 
 then 
 \[
  \frac{\lambda_i}{\mu_i} \to \frac{\diam(L)}{\diam(M)} \as i \to \infty  
  \quad\aand\quad 
  \dim(L) = \dim(M).
 \]
\end{lemma}

\begin{proof}
 Let $a$ be any accumulation point of $\big(\frac{\lambda_i}{\mu_i}\big)_{i \in \nn}$ and 
 $\big(\frac{\lambda_{i_j}}{\mu_{i_j}}\big)_{j \in \nn}$ be 
 the corresponding converging subsequence.
 Then 
 \[
  (\lambda_{i_j} M_{i_j}, p_{i_j}) 
  = \Big(\frac{\lambda_{i_j}}{\mu_{i_j}} \cdot \mu_{i_j} M_{i_j}, p_{i_j}\Big) 
  \to (a \, (\rr^k \times M), p_M) 
  = (\rr^k \times\,(a M), p_M)
 \]
 as $j \to \infty$.
 Since this sequence converges to $(\rr^k \times L, p_L)$ as well, there is an isometry
 \[(\rr^k \times\,(a M), p_M) \cong (\rr^k \times L, p_L).\]
 Thus, $\dim(L) = \dim(M)$, $\diam(L) = a \cdot \diam(M)$ and
 all accumulation points of the bounded sequence $(\frac{\lambda_i}{\mu_i})_{i \in \nn}$ 
 equal $\frac{\diam(L)}{\diam(M)}$. 
 In particular, $\big(\frac{\lambda_i}{\mu_i}\big)_{i \in \nn}$ is convergent.
\end{proof}

Next, the question will be answered 
under which condition the quotient of rescaling sequences 
belonging to two points in $\mathcal{X}_r$ is bounded.
The first approach in order to prove this 
is the special case of their good subsets to intersect.
In the general case, the idea is to connect the points 
by a curve which itself is contained in $\mathcal{X}_r$
and can be covered by finitely many balls 
such that subsequent subsets of good points intersect.
In fact, this cannot be expected to be possible for the same $r>0$.
However, it turns out that the quotient of the rescaling sequences is bounded 
if the points are connected by a minimising geodesic 
contained in some $\mathcal{X}_{r'}$ of a possibly different $r'$.
Making all of this precise is the subject of the following lemma.

\begin{lemma}\label{lem:compare_rescalings}
 Let $(M_i,p_i)_{i \in \nn}$ be a collapsing sequence 
 of pointed complete connected \ndim Riemannian manifolds
 which satisfy the uniform lower Ricci curvature bound $\Ric_{M_i} \geq -(n-1)$ 
 and converge to a limit $(X,p)$ of dimension $k < n$. 
 \begin{enumerate}
 \item\label{lem:compare_rescalings-neighboured_points} 
  Let $r_m > 0$, $x_m \in \mathcal{X}_{r_m}$ and $p_i^{x_m} \to x_m$, where $1 \leq m \leq 2$, 
  such that $G_{r_1}(p_i^{x_1}) \cap G_{r_2}(p_i^{x_2}) \ne \emptyset$ for all $i \in \nn$.
  Then
  \[ \frac{1}{5} \leq \frac{\lambda_i^{x_1}(r_1)}{\lambda_i^{x_2}(r_2)} \leq 5\] 
  for almost all $i \in \nn$
  and $l_\omega^{x_1}(r_1) = l_\omega^{x_2}(r_2)$.
 \item\label{lem:compare_rescalings-different_rescalings_for_same_point}
  Let $x \in \Xgen$, $p_i^x \to x$ and $r^x \geq R > r > 0$.
  Then there is a natural number $m=m(n,\mydelta, r, R) \in \nn$ such that 
  \[ 5^{-m} \leq \frac{\lambda_i^x(r)}{\lambda_i^x(R)} \leq 5^m\]
  for almost all $i \in \nn$
  and $l_\omega^x(r) = l_\omega^x(R)$.
 \item\label{lem:compare_rescalings-along geodesic}
  Let $\gamma : [0,l] \to X$ be a minimising geodesic 
  with image $\im(\gamma) \subseteq \mathcal{X}_{r}$ for some $0 < r \leq \tildeeps$.
  Let $x=\gamma(0)$ and $y=\gamma(l)$.
  Then there is a natural number $m=m(n,\mydelta,l,r) \in \nn$ such that 
  \[ 5^{-m} \leq \frac{\lambda_i^x(r)}{\lambda_i^y(r)} \leq 5^m\]
  for almost all $i \in \nn$
  and $l_\omega^x(r) = l_\omega^y(r)$.
 \item\label{lem:compare_rescalings-points_connected_by_geodesic}
  Let $x,y \in \mathcal{X}_r$ and $\gamma : [0,l] \to X$ be a minimising geodesic 
  with end points
  $x = \gamma(0)$, $y = \gamma(l)$ 
  and image $\im(\gamma) \subseteq \mathcal{X}_{r'}$ for some $r' \leq r \leq \tildeeps$.
  Then there is a natural number $m=m(n,\mydelta, l, r, r') \in \nn$ such that 
  \[ 5^{-m} \leq \frac{\lambda_i^x(r)}{\lambda_i^y(r)} \leq 5^m\]
  for almost all $i \in \nn$
  and $l_\omega^x(r) = l_\omega^y(r)$.
 \end{enumerate}
\end{lemma}

\begin{proof} 
 a) 
  Without loss of generality, 
  for $1 \leq m \leq 2$,
  all $\lambda_i^{x_m}(r_m)$ are positive.
  Define 
  $a_i := \frac{\lambda_i^{x_1}(r_1)}{\lambda_i^{x_2}(r_2)} > 0$ 
  and let $q_i \in G_{r_1}(p_i^{x_1}) \cap G_{r_2}(p_i^{x_2})$ be arbitrary.  
  
  By \autoref{prop-A}, $(a_i)_{i \in \nn}$ is bounded. 
  Let $a$ be an arbitrary accumulation point 
  and $(a_{i_j})_{j \in \nn}$ be the subsequence converging to $a$.
  Since $q_{i_j} \in G_{r_1}(p_{i_j}^{x_1})$, after passing to a subsequence,
  \[
   (\lambda_{i_j}^{x_1}(r_1)\,M_{i_j}, q_{i_j}) 
   \to (\rr^k \times K_1, \cdot) 
   \as j \to \infty
  \]
  for some compact metric space $K_1$ with $\frac{1}{5} \leq \diam(K_1) \leq 1$.
  Because of $q_{i_j} \in G_{r_2}(p_{i_j}^{x_2})$, 
  after passing to a further subsequence,
  \[
   (\lambda_{i_j}^{x_2}(r_2)\,M_{i_j}, q_{i_j}) 
   \to (\rr^k \times K_2, \cdot) 
   \as j \to \infty
  \]
  and $K_2$ satisfies $\frac{1}{5} \leq \diam(K_2) \leq 1$. 
  By \autoref{lem:cor_of_prop-E},
  \[ 
   a 
   = \lim_{j \to \infty} \frac{\lambda_{i_j}^{x_1}(r_1)}{\lambda_{i_j}^{x_2}(r_2)} 
   = \frac{\diam(K_1)}{\diam(K_2)} 
   \in \Big[\frac{1}{5}, 5\Big]
  \]
  and $l_\omega^{x_1}(r_1) = l_\omega^{x_2}(r_2)$.
  
  Since $(a_i)_{i \in \nn}$ is a bounded sequence 
  and all accumulation points are contained in $[\frac{1}{5},5]$,
  only finitely many $a_i$ are not contained in $[\frac{1}{5},5]$.
  
 \par\smallskip\noindent b)
  Since $\CBG(n,-1,\beta R, R)$ is monotonically increasing for decreasing $\beta < 1$, 
  there exists $\beta = \beta(n,\mydelta,R) < 1$ 
  with $\CBG(n,-1,\beta R, R) < \frac{1}{\mydelta} - 1$.
  Fix this $\beta$.
  Because
  $\CBG(n,-1,\beta \rho, \rho)$ is monotonically increasing for increasing $\rho$, 
  all $\rho \leq R$ satisfy $\CBG(n,-1,\beta \rho, \rho) < \frac{1}{\mydelta} - 1$ as well.
   
  Let $m = m(n, \mydelta, r, R) \in \nn$ be maximal with $r \leq \beta^m \cdot R$.
  Define $r_j := \beta^j \cdot R$ 
  for $0 \leq j < m$
  and $r_m := r \leq \beta \cdot r_{m-1}$.
  Then $\CBG(n,-1,r_{j+1},r_j) < \frac{1}{\mydelta} - 1$ for all $0 \leq j < m$.
  Moreover, $x \in \mathcal{X}_{r_j}$ for all $0 \leq j \leq m$ due to $r_j \leq R \leq r^x$.

  Assume $G_{r_j}(p_i^x) \cap G_{r_{j+1}}(p_i^x) = \emptyset$ 
  for some $0 \leq j < m$ and $i \in \nn$.
  This implies \[G_{r_{j+1}}(p_i^x) \subseteq B_{r_{j}}(p_i^x) \setminus G_{r_{j}}(p_i^x),\]
  in particular,
  \begin{align*}
   (1- \mydelta)\cdot \vol(B_{r_j}(p_i^x))
   &\leq \vol(G_{r_{j+1}}(p_i^x)) 
   \\&\leq \vol(B_{r_{j}}(p_i^x) \setminus G_{r_{j}}(p_i^x))
   \\&\leq \mydelta \cdot \vol(B_{r_{j}}(p_i^x))
   \\&\leq \mydelta \cdot \CBG(n,-1,r_{j+1},r_j)\cdot \vol(B_{r_j}(p_i^x)).
  \end{align*}
  Hence, $1- \mydelta \leq \mydelta \cdot \CBG(n,-1,r_{j+1},r_j) < 1 -\mydelta$, 
  and this is a contradiction.
  
  Thus, $G_{r_j}(p_i^x) \cap G_{r_{j+1}}(p_i^x) \ne \emptyset$ 
  for all $0 \leq j < m$ and $i \in \nn$. 
  By \ref{lem:compare_rescalings-neighboured_points}, 
  \[ \frac{1}{5} \leq \frac{\lambda_i^x(r_j)}{\lambda_i^x(r_{j+1})} \leq 5\] for almost all $i$
  and $l_\omega^x(r_j) = l_\omega^x(r_{j+1})$.
  Inductively, 
  \[l_\omega^x(r) = l_\omega^x(r_m) = l_\omega^x(r_0) = l_\omega^x(R).\]
  Then
  \[\frac{\lambda_i^x(R)}{\lambda_i^x(r)}
  = \frac{\lambda_i^x(r_0)}{\lambda_i^x(r_m)}
  = \prod_{j=0}^{m-1} \frac{\lambda_i^x(r_j)}{\lambda_i^x(r_{j+1})}\]
  proves the claim.
  
 \par\smallskip\noindent c)  
  Let $d_0 = d_0(n,\mydelta,r)$ be as in \autoref{lemma-criterion_volume_intersection2} 
  and $m_0 = m_0(n, \mydelta, l, r) \in \nn$ be the minimal natural number with
  $l \leq m_0 \cdot d_0$.  
  
  Define a sequence $0 = t_0 < t_1 < \ldots < t_{m_0} = l$ 
  with pairwise $t_{j+1}-t_j \leq d_0$
  by $t_j := j \cdot d_0$ for $0 \leq j \leq m_0 -1$ and $t_{m_0} := l$.   
  For $0 \leq j \leq m_0$, define $y_j := \gamma(t_j) \in \mathcal{X}_{r}$.
  Now fix $0 \leq j < m_0$.
  By \autoref{lemma-criterion_volume_intersection2} 
  and \autoref{lemma-criterion_intersection_subsets},
  \[G_{r}(p_i^{y_j}) \cap G_{r}(p_i^{y_{j+1}}) \ne \emptyset.\]
  The rest of the proof can be done analogously 
  to the one of \ref{lem:compare_rescalings-different_rescalings_for_same_point}.
  
 \par\smallskip\noindent d) 
  Let $m_x := m_y := m(n, \mydelta, r', r)$ be 
  as in \ref{lem:compare_rescalings-different_rescalings_for_same_point} 
  and $m_0 = m_0(n, \mydelta, l, r')$ as in \ref{lem:compare_rescalings-along geodesic}.
  Then 
  $m = m(n, \mydelta, l, r, r') := m_x \cdot m_0 \cdot m_y$
  proves the claim.
\end{proof}

\subsection{Generic points and geodesics}\label{sec:geodesics}
Throughout this subsection, 
fix a collapsing sequence $(M_i,p_i)_{i \in \nn}$ 
of pointed complete connected \ndim Riemannian manifolds
which satisfy the uniform lower Ricci curvature bound $\Ric_{M_i} \geq -(n-1)$ 
and converge to a limit $(X,p)$ of dimension $k < n$
and use the notation introduced in \autoref{sec:application_to_generic_points}. 
Moreover, minimising geodesics are assumed to be parametrised by arc length.

By \autoref{lem:compare_rescalings},
rescaling sequences corresponding to two different points can be compared 
if those points are connected by a geodesic lying in some $\mathcal{X}_r$.
It remains to check for which points this is the case.
It will turn out that, if the strict interior of a minimising geodesic 
(i.e.~the interior bounded away from the endpoints) is generic,
then it is already contained in $\mathcal{X}_r$ for sufficiently small $r > 0$. 
In fact, nearly all pairs of points lie in the interior of such a geodesic 
such that the part of the geodesic connecting these points is generic.

\begin{not*}
Define 
\begin{align*}
 \mathcal{G} 
  &:= \{(x,y) \in \Xgen \times \Xgen 
  \mid \exists \text{~minim.~geod.~} \gamma:[0,l] \to X, 0<t_x<t_y<l: 
  \\&{\color{white}{:= \{(x,y) \in \Xgen \times \Xgen \mid \}}}
  x=\gamma(t_x), y=\gamma(t_y), \im(\gamma_{|[t_x,t_y]}) \subseteq \Xgen \},
 \intertext{and for $x \in \Xgen$ denote 
 the image under the projection to the second factor by}
 \mathcal{G}_x 
 &:= \{y \in \Xgen \mid (x,y) \in \mathcal{G}\}.
 \intertext{Finally, define}
 \mathcal{G}' 
 &:= \{x \in \Xgen \mid \mathcal{G}_x \text{ has full measure in } X \}.
 \end{align*}
\end{not*}

\begin{lemma}\label{lem:points_s.t._nearly_all_pairs_in_interior_of_good_geod}
 The set $\mathcal{G}'$ has full measure in $X$.
\end{lemma}

\begin{proof}
 First, prove that $\mathcal{G}$ has full measure in $X \times X$.
 Let 
 \begin{align*}
  S_1 
  &:= \{(x,y) \in \Xgen \times \Xgen 
  \mid \exists \text{~minim.~geodesic~} c:[0,d] \to X : 
    \\&{\color{white}{:= \{(x,y) \in \Xgen \times \Xgen \mid \}}} 
    x=c(0), y=c(d), \im(c) \subseteq \Xgen\} \\
  S_2 
  &:= \{(x,y) \in X \times X 
  \mid \exists \text{~minim.~geodesic~} \gamma: [0,l] \to X, 0<t_x<t_y<l: 
  \\& {\color{white}{:= \{(x,y) \in X \times X \mid \}}} 
  x=\gamma(t_x), y=\gamma(t_y)\}
 \end{align*}
 and define $S := S_1 \cap S_2$.
 By \cite[Theorem 1.20 (1), Theorem A.4 (3)]{colding-naber}, 
 \[\vol_X \times \vol_X(X \times X \setminus S_1) = 0
 \quad\aand\quad
 \vol_X \times \vol_X(\Xgen \times \Xgen \setminus S_2) = 0.\]
 In particular, using that $\vol_X(X \setminus \Xgen) = 0$,
 this proves \[\vol_X \times \vol_X(X \times X \setminus S) = 0.\]
 
 Next, prove $S \subseteq \mathcal{G}$: 
 Let $(x,y) \in S$,
 $c:[0,d] \to \Xgen$ and $\gamma: [0,l] \to X$ be geodesics and $0<t_x<t_y<l$ with 
 $x=c(0) = \gamma(t_x)$ and $y=c(d) = \gamma(t_y)$. 
 In particular, $d = d(x,y) = t_y - t_x \leq l$.
 
 Define $\tilde{\gamma}: [0,l] \to X$ by 
 \[\tilde{\gamma}(\tau) := 
  \begin{cases}
   \gamma(\tau)	& \text{if }\tau \in [0,t_x] \cup [t_y,l],\\
   c(\tau-t_x) 	& \text{if }\tau \in [t_x,t_y],
  \end{cases}
  \]
 cf.~\autoref{pic:points_s.t._nearly_all_pairs_in_interior_of_good_geod}.
 \begin{figure}[t]  
  \begin{center}
   \begin{tikzpicture}
   [auto,>=stealth,inner sep=0pt,thick,dot/.style={fill=black,circle,minimum size=3pt}]
	\coordinate[] 	(x) 		
	  at (0.0,0.0);
    	\node[gray]   	(x-label) 	
	  at (-1.3,0.35)
	  {$\gamma(t_x) = c(0) = x$};
	\coordinate[] 	(y) 		
	  at (3.0,0.0);
    	\node[gray]   	(y-label) 	
	  at (4.3,0.35)
	  {$y = c(d) = \gamma(t_y)$};
    	\coordinate[] 	(tau1)
	  at (-1,0.0);
    	\node 	      	(tau1-label)
	  at (-1.8,-0.35)
	  {$\tilde{\gamma}(\tau_1)=\gamma(\tau_1)$};
    	\coordinate[] 	(tau2)
	  at (2,0.4);
    	\node[gray]    	(tau2-label)
	  at (2,0.75)	
	  {$\gamma(\tau_2)$};
    	\coordinate[] 	(tau2')
	  at (2,-0.4);
    	\node 	      	(tau2'-label)
	  at (3,-0.75)	
	  {$\tilde{\gamma}(\tau_2) = c(\tau_2-t_x)$};
    	\draw
	    (-2,0) -- (0,0)
	    (3,0) -- (5,0);
	\path (x) edge [densely dashed,bend left=30] (y);
	\path (x) edge [bend right=30] (y);
	\foreach \point in {tau1,tau2'}
	  \fill [black] (\point) circle (1.7pt);
	\foreach \point in {x,y,tau2}
	  \fill [gray] (\point) circle (1.7pt);
   \end{tikzpicture}
   \caption{Construction of $\tilde{\gamma}$.}
   \label{pic:points_s.t._nearly_all_pairs_in_interior_of_good_geod}
  \end{center}
 \end{figure}
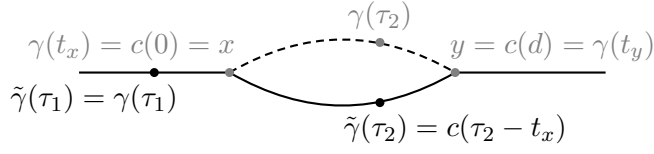
 A straightforward computation proves that $\tilde{\gamma}$ is a minimising geodesic.
 Then $\tilde{\gamma}$ verifies $(x,y) \in \mathcal{G}$, and this proves
 \[\vol_X \times \vol_X(X \times X \setminus \mathcal{G}) = 0.\]
 Using 
 $X \times X \setminus \mathcal{G} 
 = \bigcup_{x \in X} \{x\} \times (X \setminus \mathcal{G}_x)$,
 \begin{align*}
  0
  &=\vol_{X \times X}(X \times X \setminus \mathcal{G}) 
  \\&= \int_X \vol_X(X \setminus \mathcal{G}_x) \dV(x) 
  \\&= \int_{X \setminus \mathcal{G}'} \vol_X(X \setminus \mathcal{G}_x) \dV(x). 
 \end{align*}
 Since $\vol_X(X \setminus \mathcal{G}_x) > 0$ 
 for all $x \in X \setminus \mathcal{G}'$, 
 this proves $\vol_X(X \setminus \mathcal{G}') = 0$.
\end{proof}

So far it was seen that almost all points can be connected 
by a geodesic lying in $\Xgen$ which can be extended at both ends.
By applying the following theorem of Colding and Naber, 
which describes the Hölder continuity of the geometry 
of small balls with the same radius, 
to this situation,
one obtains that the interior of the regarded geodesics 
not only lies in $\Xgen$, but in $\mathcal{X}_r$ for some $r > 0$.
\begin{thm}
 [{\cite[Theorem 1.1, Theorem 1.2]{colding-naber}}]
 \label{thm_CN:hoelder}
 For $n \in \nn$ there are $\alpha(n)$, $C(n)$ and $r_0(n)$ 
 such that the following holds:
 Let $M$ be a complete \ndim Riemannian manifold 
 with $\Ric_M \geq -(n-1)$ or the limit space of a sequence of such manifolds,
 let $\gamma : [0,l] \to M$ a minimising geodesic (parametrised by arc-length)
 and fix $\beta \in(0,1)$.
 For $0 < r < r_0 \beta l$ and $\beta l < s < t < (1-\beta)l$,
 \[
  \dgh{r}{M}{\gamma(s)}{M}{\gamma(t)} 
  < \frac{C}{\beta l} \cdot r \cdot |s-t|^{\alpha(n)}.
 \]
\end{thm}

\begin{lemma}
 \label{stmt_glob:rescaling_st_interior_of_geod_uniformly_close_to_IRn}
 Let $\mydelta \in (0,\frac{1}{2})$, 
 $\gamma : [0,l] \to X$ be a minimising geodesic 
 and assume $0 < s < t < l$ to be times
 such that $\gamma_{|[s,t]}$ is contained in $\Xgen$. 
 Then there is $0 < r' = r'(\mydelta, l, s, t; n,k) \leq \tildeeps$ 
 such that for all $0 < r \leq r'$, 
 \[\im(\gamma_{|[s,t]}) \subseteq \mathcal{X}_r.\]
\end{lemma}
\begin{proof}
 Define $\beta = \beta(l,s,t) := \frac{1}{2l} \cdot \min\{s, l-t\} > 0$.
 Then $t,s \in (\beta l, (1-\beta) l)$. 
 Furthermore, let 
 $\alpha(n), C(n), r_0(n)$ be as in \autoref{thm_CN:hoelder} and 
 define 
 \[
  d 
  = d(\mydelta, l, s, t; n, k) 
  := \sqrt[\alpha(n)]{\frac{\beta  l \cdot \tildeeps^2}{2C(n)}}.
 \]
 Let $m = m(s,t)$ be the maximal natural number with $(m-1) d \leq t-s$ 
 and define
 $\tau_j := s + jd$ for $0 \leq j \leq m$.
 By definition, 
 $\tau_0 = s$ and $\tau_m = s + m d > t$.
 Therefore, $[s,t] \subseteq \bigcup_{j=0}^{m} (\tau_j - d, \tau_j + d)$.
 
 For every $0 \leq j \leq m$, 
 choose $\lambda_j = \lambda_j(\mydelta, l, s, t; n,k) > 1$ 
 as in \autoref{lemma:rescaling_close_to_tangent_cone} 
 and
 define 
 $r' 
 = r'(\mydelta, l, s, t; n,k) 
 := \min\{\tildeeps,\frac{1}{\lambda_0},\dots,\frac{1}{\lambda_m},
 \tildeeps \cdot r_0(n) \cdot \beta l\}$.
 Let $0 < r \leq r'$ and $\tau \in [s,t]$ be arbitrary. 
 Choose $0 \leq j \leq m$ with $|\tau - \tau_j| < d$.
 By definition of $d$,
 \[
  |\tau - \tau_j|^{\alpha(n)} 
  < d^{\alpha(n)} 
  = \frac{\beta l \cdot \tildeeps^2}{2 C(n)},
 \]
 and so, using \autoref{thm_CN:hoelder},
 \begin{align*}
  &\dgh{1/\tildeeps}{r^{-1} X}{\gamma(\tau)}{\rr^k}{0}\\
  &\leq \frac{1}{r} 
    \cdot \dgh{r/\tildeeps}{X}{\gamma(\tau)}{X}{\gamma(\tau_j)} 
 + \dgh{1/\tildeeps}{r^{-1} X}{\gamma(\tau_j)}{\rr^k}{0}\\
  &\leq \frac{1}{r} \cdot \frac{C(n)}{\beta l} \cdot \frac{r}{\tildeeps} 
    \cdot |\tau_j - \tau|^{\alpha(n)} + \frac{\tildeeps}{2} \\
  &< \tildeeps.
  \qedhere
 \end{align*}
\end{proof}

\subsection{Proof of the main theorem}\label{sec:proof_main_thm}

In order to prove \autoref{thm:main}, the following technical result is needed
which estimates the number of balls a point can be contained in 
if the base points of these balls form an $\eps$-net.

\begin{lemma}\label{prop-J}
 Let $X$ be an \ndim Riemannian manifold 
 with lower Ricci curvature bound $\Ric \geq (n-1) \cdot \kappa$ 
 or the limit of a sequence of such manifolds. 
 Then each point is contained in maximal $\CBG(n, \kappa, r, r+2R)$ balls 
 with radii $R$ whose base points have pairwise distance at least $2r$.
\end{lemma}

\begin{proof}  
 This result is an immediate consequence of \refBGThm: 
 Let $p_1, \dots, p_m \in X$ be points with pairwise distance at least $2r$ 
 and $q \in \bigcap_{i=1}^m B_R(p_i)$. 

 On the one hand, since $d(p_i,p_j) \geq 2r$, 
 one has $B_r(p_i) \cap B_r(p_j) = \emptyset$ for any $i \ne j$. 
 On the other hand, 
 $B_r(p_i) \subseteq B_{r+R}(q)$,
 hence, 
 $\coprod_{i=1}^m B_r(p_i) \subseteq B_{r+R}(q)$. 
 Furthermore, 
 $B_{r+R}(q) \subseteq B_{r+2R}(p_i)$
 for any $1 \leq i \leq m$.
 Together, 
 \begin{align*}
  1 
  \geq \frac{\vol(\coprod_{i=1}^m B_r(p_i))}{\vol(B_{r+R}(q))} 
  &= \sum_{i=1}^m \frac{\vol(B_r(p_i))}{\vol(B_{r+R}(q))} \\
  &\geq \sum_{i=1}^m \frac{\vol(B_r(p_i))}{\vol(B_{r+2R}(p_i))} 
  \geq 
  \frac{m}{\CBG(n,\kappa,r,r+2R)}.
 \end{align*}
 Thus, $m \leq \CBG(n,\kappa,r,r+2R)$.
\end{proof}

It remains to prove the main theorem. 
Again, the notation introduced 
in \autoref{sec:application_to_generic_points} 
and \autoref{sec:geodesics} is used.

\begin{proof}[Proof of \autoref{thm:main}]
 The idea of the proof is the following: 
 First, fix a bound $\mydelta \in (0,\frac{1}{2})$ and
 choose a radius $R$ 
 such that $\mathcal{X}_R(\mydelta;n,k)$ has sufficiently large volume.
 Inside of this set of points, 
 choose a point $x_0$ and a finite $R$-net of points $x_j$ 
 such that $(x_0,x_j) \in \mathcal{G}$,
 and take the union of the subsets $G_{R}(p_i^{x_j})$. 
 This has the required properties.
  
 Let $\eps \in (0,1)$ be arbitrary and define
 \[
  \mydelta 
  = \mydelta(n,\eps) 
  := \frac{\eps}{2 \cdot \CBG\big(n,-1,\frac{1}{8},\frac{17}{8}\big)} 
  \in \Big(0,\frac{1}{2}\Big).
 \] 
 For arbitrary $r > 0$, define 
 \[X'(r) := \{x \in B_{1-r}(p) \cap \Xgen \mid r^x \geq r\}.\]
 For $r_1 \leq r_2$, obviously $X'(r_2) \subseteq X'(r_1)$. 
 Further, 
 $\bigcup_{r > 0} X'(r) = B_1(p) \cap \Xgen$.
 Thus, there exists a radius $ 0 < R = R(\eps,X,p;n) \leq 1$ such that 
 \[
  \vol_X(X'(r)) 
  \geq \Big(1 - \frac{\eps}{4}\Big) \cdot \vol_X (B_1(p) \cap \Xgen) 
  = 1 - \frac{\eps}{4}
 \]
 for all $r \leq R$. 
 Fix this $0 < R \leq 1$. 
  
 Since $\mathcal{G}'$ has full measure 
 by \autoref{lem:points_s.t._nearly_all_pairs_in_interior_of_good_geod},
 $X'(R) \cap \mathcal{G}'$ is non-empty. 
 Fix an arbitrary point $x_0 \in X'(R) \cap \mathcal{G}'$, 
 define \[X' := X'(R) \cap \mathcal{G}_{x_0}\]
 and
 choose a maximal number of points $x_1, \dots, x_l \in X'$ 
 with pairwise distance at least $R$.
 By the maximality of the choice, 
 $X' \subseteq \bigcup_{j=1}^l B_{R}(x_j)$.
 Since 
 $\mathcal{G}_{x_0}$ has full measure
 by choice of $x_0$,
 \[
  \vol_X \big( \bigcup_{j=1}^l B_{R}(x_j) \big) 
  \geq \vol_X(X') 
  = \vol_X(X'(R)) 
  \geq 1 - \frac{\eps}{4}.
 \]
 On the other hand, by choice, 
 $B_{R}(x_j) \subseteq B_1(p)$.
 Thus,
 \[\vol_X(B_1(p) \setminus \bigcup_{j=1}^l B_{R}(x_j)) \leq \frac{\eps}{4}.\]
  
 Let $p_i^{x_j} \to x_j$ and $i_0 \in \nn$ be large enough 
 such that for all $i \geq i_0$ and all $1 \leq j < j' \leq l$,
 \begin{align*}
  &d(p_i^{x_j},p_i^{x_{j'}}) 
  \geq \frac{1}{2} \cdot d(x_j,x_{j'}) 
  \geq \frac{R}{2}
  \intertext{and}
  &\frac{\vol_{M_i}(B_1(p_i) 
  \setminus \bigcup_{j=1}^l B_{R}(p_i^{x_j}))}{\vol_{M_i}(B_1(p_i))}
  \leq 2 \cdot \frac{\vol_{X}(B_1(p) 
  \setminus \bigcup_{j=1}^l B_{R}(x_j))}{\vol_{X}(B_1(p))}.
 \end{align*}
 Fix $i \geq i_0$.
 Then
 \begin{align*}
  &\vol_{M_i}(B_1(p_i) \setminus \bigcup_{j=1}^l B_{R}(p_i^{x_j}))
  \\&\leq 2 \cdot \frac{\vol_{X}(B_1(p) 
  \setminus \bigcup_{j=1}^l B_{R}(x_j))}{\vol_{X}(B_1(p))} 
  \cdot \vol_{M_i}(B_1(p_i)) \\
  &\leq \frac{\eps}{2} \cdot \vol_{M_i}(B_1(p_i)).
  \intertext{By \autoref{prop-J}, every point in the union 
  $\bigcup_{j=1}^l B_{R}(p_i^{x_j})$ is contained 
  in at most $M$ different balls $B_{R}(p_i^{x_j})$ for
  \[
   M 
   = M(\eps;n,k) 
   := \CBG\Big(n,-1,\frac{R}{8}, \frac{17R}{8}\Big) 
   \leq \CBG\Big(n,-1,\frac{1}{8}, \frac{17}{8}\Big) 
   = \frac{\eps}{2 \mydelta}
  .\]
  Therefore,}
  \sum_{j=1}^l \vol_{M_i}(B_{R}(p_i^{x_j}))
  &\leq M \cdot \vol_{M_i}\big(\bigcup_{j=1}^l B_{R}(p_i^{x_j})\big) 
  \leq \frac{\eps}{2 \mydelta} \cdot \vol_{M_i}(B_1(p_i)).
 \end{align*} 
 Thus,
 \begin{align*}
  \vol_{M_i}\big(\bigcup_{j=1}^l B_{R}(p_i^{x_j}) 
  \setminus \bigcup_{j=1}^l G_{R}(p_i^{x_j})\big)
  &\leq \vol_{M_i}\big(\bigcup_{j=1}^l (B_{R}(p_i^{x_j}) 
  \setminus G_{R}(p_i^{x_j}))\big)
  \\&\leq \sum_{j=1}^l \vol_{M_i}(B_{R}(p_i^{x_j}) 
  \setminus G_{R}(p_i^{x_j})) 
  \displaybreak[0]\\
  &\leq \sum_{j=1}^l \mydelta \cdot \vol_{M_i}(B_{R}(p_i^{x_j})) 
  \\&\leq \frac{\eps}{2} \cdot \vol_{M_i}(B_1(p_i)).
 \end{align*}
 Hence,
 \begin{align*}
  \vol_{M_i}(B_1(p_i) \setminus \bigcup_{j=1}^l G_R(p_i^{x_j}))
  \leq \eps \cdot \vol_{M_i}(B_1(p_i)).
 \end{align*}
 Now define
 \[
  G_1(p_i) := \bigcup_{j=1}^l G_{R}(p_i^{x_j}) 
  \quad\aand\quad 
  \lambda_i := \lambda_i^{x_0}(R).
 \]
 By construction,
 \[\vol_{M_i}(G_1(p_i)) \geq (1-\eps) \cdot \vol_{M_i}(B_1(p_i)).\]
 From now on, let $\lambda_i^{x_j}$ denote $\lambda_i^{x_j}(R)$.
 
 Fix $1 \leq j \leq l$. 
 By construction, $(x_0,x_j) \in \mathcal{G}$ 
 and $x_0, x_j \in \mathcal{X}_{R}$.
 Thus, there exists a minimising geodesic $\gamma_j : [0,l_j] \to X$ 
 and $0 < s_j < t_j < l_j$ 
 such that ${\gamma_j}_{|[s_j,t_j]}$ is contained in $\Xgen$, 
 $\gamma_j(s_j) = x_0$ and $\gamma_j(t_j) = x_j$.
 By \autoref{stmt_glob:rescaling_st_interior_of_geod_uniformly_close_to_IRn},
 there is $r_j' > 0$ such that for all $0 < r \leq r_j'$, 
 ${\gamma_j}_{|[s_j,t_j]}$ is contained in $\mathcal{X}_r$.
 Let $r_j := \min\{r_j',R\}$.
 By \autoref{lem:compare_rescalings} 
 \ref{lem:compare_rescalings-points_connected_by_geodesic},
 there is $m_j := m(n, \mydelta, d_X(x_0, x_j), r_j, R)$ satisfying 
 \[ 5^{-m_j} \leq \frac{\lambda_i^{x_0}}{\lambda_i^{x_j}} \leq 5^{m_j}\]
 for almost all $i \in \nn$
 and $l_\omega^{x_0}(R) = l_\omega^{x_j}(R)$.
 From now on, let $i \geq i_0$ be large enough 
 such that the above estimate holds for all $1 \leq j \leq l$.
 
 Given $q_i \in G_1(p_i)$, 
 let $(Y, q)$ be an arbitrary sublimit of $(\lambda_i M_i, q_i)$, 
 i.e.~for a subsequence $(i_s)_{s \in \nn}$, 
 \[(\lambda_{i_s} M_{i_s}, q_{i_s}) \to (Y, q) \as s \to \infty.\]
 
 For a further subsequence $(i_{s_t})_{t \in \nn}$ 
 there is $1 \leq j \leq l$ 
 with $q_{i_{s_t}} \in G_{R}(p_{i_{s_t}}^{x_j})$ for all $t$
 and 
 \[
  (\lambda_{i_{s_t}}^{x_j} M_{i_{s_t}}, q_{i_{s_t}}) 
  \to (\rr^k \times \tilde{K}, \cdot) 
  \as t \to \infty
 \]
 for a compact metric space $\tilde{K}$ satisfying 
 $\diam(\tilde{K}) \in [\frac{1}{5},1]$. 
 
 On the other hand, 
  \[ 
   \bigg( \frac{\lambda_{i_{s_t}}}{\lambda_{i_{s_t}}^{x_j}} 
   \cdot \lambda_{i_{s_t}}^{x_j} M_{i_{s_t}}, q_{i_{s_t}}\bigg)  
   = \big(\lambda_{i_{s_t}} M_{i_{s_t}}, q_{i_{s_t}}\big) 
   \to (Y, q) 
   \as t \to \infty,
  \]
 the sequence $\frac{\lambda_{i_{s_t}}}{\lambda_{i_{s_t}}^{x_j}}$ 
 converges to some $\alpha$ 
 and $Y$ is isometric to the product $\rr^k \times K$ 
 for $K := \alpha \tilde{K}$.
 In particular, $5^{-m_j} \leq \alpha \leq 5^{m_j}$. 
 So,
 $\diam(K) \in [\frac{1}{D},D]$ 
 for $D:= 5^{\max\{m_j \mid 1 \leq j \leq l\} + 1}$.
 Moreover, for any non-principal ultrafilter $\omega$, 
 \[\dim(K) = \dim(\tilde{K}) = l_\omega^{x_j}(R) = l_\omega^{x_0}(R).\]
 
 In particular, 
 for any two sublimits $(\rr^k \times K_1, \cdot)$ 
 and $(\rr^k \times K_2, \cdot)$ 
 coming from the same subsequence of indices,
 let $\omega$ be a non-principal ultrafilter 
 such that these sublimits are ultralimits with respect to $\omega$.
 Then \[\dim(K_1) = \dim(K_2).\qedhere\]
\end{proof}

 


\end{document}